\documentclass[11pt]{article}
\usepackage{amsmath,amsfonts,amssymb,amsthm,bbm}
\usepackage{natbib,enumerate}
\usepackage[prependcaption,colorinlistoftodos,textsize=small]{todonotes}

\usepackage{amsmath,amsthm,amssymb,marvosym,enumitem,makeidx,hyperref,graphicx,xcolor,enumerate,soul,verbatim}
\renewcommand{\le}{\leqslant}  
\newcommand\m[1] {\begin{pmatrix}#1\end{pmatrix}} 
\newcommand{\indep}{\perp \!\!\! \perp} 

\newcommand{\bigO}{O}

\newcommand{\littleO}{o}

\newcommand{\iidsim}{\stackrel{{iid}}{\sim}}

\newcommand{\norm}[1]{\left\lVert #1 \right\rVert}

\newcommand{\indistribution}{\rightsquigarrow}
\newcommand{\inprobability}{\xrightarrow[]{P}}

\def\*#1{\mathbf{#1}} 
\def\set#1{\left\{ {#1} \right\}} 

\def\bG{{\mathbb G}}
\def\bQ{{\mathbb Q}}

\def\bR{{\mathbb R}}
\def\bE{{\mathbb E}}
\def\bP{{\mathbb P}}

\def\sA{{\mathcal A}}

\def\sC{{\mathcal C}}

\def\sF{{\mathcal F}}
\def\sG{{\mathcal G}}
\def\sH{{\mathcal H}}

\def\sL{{\mathcal L}}
\def\sM{{\mathcal M}}

\def\sO{{\mathcal O}}
\def\sP{{\mathcal P}}
\def\sQ{{\mathcal Q}}

\def\sS{{\mathcal S}}
\def\sT{{\mathcal T}}
\def\sU{{\mathcal U}}

\def\sX{{\mathcal X}}
\def\sY{{\mathcal Y}}
\def\sZ{{\mathcal Z}}

\def\Var{\mathrm{Var}}
\def\Cov{\mathrm{Cov}}

\theoremstyle{plain}
\newtheorem{theorem}{Theorem}[section]
\newtheorem{corollary}[theorem]{Corollary}
\newtheorem{lemma}[theorem]{Lemma}
\newtheorem{proposition}[theorem]{Proposition}

\theoremstyle{remark}
\newtheorem{remark}{Remark}[section]

\newtheorem{definition}{Definition}[section]
\newtheorem{assumption}{Assumption}[section]
\newtheorem{exmp}{Example}[section]

\def\vp{{\varphi}}
\def\tr{{\text{tr}}}
\def\dim{{\text{dim}}}

\def\labeled{{\sL_n}}
\def\unlabeled{{\sU_N}}

\def\thetafunctional{{\theta(\cdot)}}

\def\IF{{\vp_{\eta^*}}}

\def\IFplugin{{\vp_{\cetaest}}}
\def\cIF{{\phi_{\eta^*}}}
\def\EIF{{\vp^*_{\eta^*}}}

\def\cEIF{{\phi^*_{\eta^*}}}
\def\cIFest{{\hat{\phi}_n}}

\def\IFa{{\vp_{\eta^*}^{(1)}}}
\def\IFb{{\vp_{\eta^*}^{(2)}}}

\def\jointmodel{{\sP}}
\def\reducedmodel{{\{\bP_X^*\} \otimes \sP_{Y\mid X}}}
\def\marginalmodel{{\sP_{X}}}
\def\conditionalmodel{{\sP_{Y\mid X}}}
\def\joint{{\bP^*}}
\def\marginal{{\bP^*_{X}}}
\def\conditional{{\bP^*_{Y\mid X}}}
\def\lpx{{\sL^2_{p,0}(\marginal)}}
\def\lpyx{{\sL^2_{p,0}(\conditional)}}
\def\lp{{\sL^2_{p,0}(\joint)}}
\def\lpxone{{\sL^2_{1,0}(\marginal)}}
\def\lpyxone{{\sL^2_{1,0}(\conditional)}}
\def\lpone{{\sL^2_{1,0}(\joint)}}

\def\thetaest{\hat{\theta}_n}
\def\thetaestnN{\hat{\theta}_{n,N}}
\def\issthetaest{\hat{\theta}_{n,\marginal}^{\text{safe}}}
\def\issthetanpest{\hat{\theta}_{n,\marginal}^{\text{eff.}}}
\def\ossthetaest{\hat{\theta}_{n,N}^{\text{safe}}}
\def\ossthetaestppi{\hat{\theta}_{n,N}^{\text{PPI}}}
\def\ossthetanpest{\hat{\theta}_{n,N}^{\text{eff.}}}
\def\cthetaest{\Tilde{\theta}_n}
\def\cetaest{\Tilde{\eta}_n}
\def\iOLS{\hat{\*B}_n^g}
\def\npiOLS{\hat{\*B}_{K_n}}
\def\oOLS{\hat{\*B}_{n,N}^g}
\def\npoOLS{\hat{\*B}_{K_n,N}}
\def\OLS{\*B^g}
\def\npOLS{\*B_{K_n}}
\def\GKn{{G_{K_n}}}
\def\OLSppi{\*B^{g_{\eta^*}}}
\def\oOLSppi{\hat{\*B}_{n,N}^{g_{\cetaest}}}
\def\centeredGKn{{G_{K_n}^0}}
\def\ecenteredGKn{{\hat{G}_{K_n}^0}}
\def\centeredg{{g^0}}
\def\ecenteredg{{\hat{g}^0}}

\def\Sigmasafe{{\*\Sigma^{\text{safe}}}}
\def\Sigmaeff{{\*\Sigma^{\text{eff}}}}

\def\msecondderivative{{\*V_{\theta^*}}}
\def\msecondderivativeinv{{\*V_{\theta^*}^{-1}}}
\def\md{{\nabla m_{\theta^*}}}
\def\mdf{{\nabla m^f_{\theta^*}}}
\def\cmd{{\nabla \Tilde{m}_{\theta^*}}}

\def\tg{{g_{\eta^*}}}
\def\tcenteredg{{g^0_{\eta^*}}}
\def\tecenteredg{{\hat{g}^0_{\eta^*}}}
\def\cg{{g_{\cetaest}}}
\def\ccenteredg{{g^0_{\cetaest}}}
\def\cecenteredg{{\hat{g}^0_{\cetaest}}}

\def\empiricalPn{{\bP_n}}
\def\empiricalPN{{\bP_N}}
\def\empiricalPXN{{\bP_{N, X}}}
\def\empiricalPnN{{\bP_{n+N}}}
\def\tangentmarginal{{\sT_\marginalmodel(\marginal)}}
\def\tangentconditional{{\sT_\conditionalmodel(\conditional)}}
\def\tangent{{\sT_\jointmodel(\joint)}}
\def\tangentmissing{{\sT_\sQ(\joint \times \Pdelta)}}
\def\gradient{{D_\joint}}
\def\cgradient{{D^*_\joint}}

\def\ossjoint{{\bQ(\joint)}}
\def\ossempirical{{\bQ_{n,N}}}
\def\ossgradient{{D_{\ossjoint}}}
\def\cossgradient{{D^*_{\ossjoint}}}
\def\osstangent{{\sT_\jointmodel(\ossjoint)}}

\def\ts{{s_{\eta^*}}}
\def\cs{{s_{\Tilde{\eta}_n}}}

\def\Pdelta{{\bP_W^*}}

\usepackage{geometry}
\title{A Unified Framework for Semiparametrically Efficient Semi-Supervised Learning}

\author{Zichun Xu\footnote{Department of Biostatistics, University of Washington}, \; Daniela Witten\footnote{Department of Statistics, University of Washington} $^*$, \; Ali Shojaie$^{* \dagger }$}
 
\begin{document}
\maketitle

\begin{abstract}
We consider statistical inference under a semi-supervised setting where we have access to both a labeled dataset consisting of pairs $\{X_i, Y_i \}_{i=1}^n$ and an unlabeled dataset $\{ X_i \}_{i=n+1}^{n+N}$. We ask the question: under what circumstances, and by how much, can incorporating the unlabeled dataset improve upon inference using the labeled data? To answer this question, we investigate semi-supervised learning through the lens of semiparametric efficiency theory. We characterize the efficiency lower bound under the semi-supervised setting for an arbitrary inferential problem, and show that incorporating unlabeled data can potentially improve efficiency \emph{if the parameter is not well-specified}. We then propose two types of  semi-supervised estimators: a \emph{safe} estimator that imposes minimal assumptions, is simple to compute, and is guaranteed to be at least as efficient as the initial supervised estimator; and an \emph{efficient} estimator, which --- under stronger assumptions --- achieves the semiparametric efficiency bound. Our findings unify existing semiparametric efficiency results for particular special cases, and extend these results to a much more general class of problems.
Moreover, we show that our estimators can flexibly incorporate predicted outcomes arising from ``black-box" machine learning models, and thereby achieve the same goal as \emph{prediction-powered inference} (PPI), but with superior theoretical guarantees. We also provide a complete understanding of the theoretical basis for the existing set of PPI methods. Finally, we apply the theoretical framework developed to derive and analyze efficient semi-supervised estimators in a number of settings, including M-estimation, U-statistics, and average treatment effect estimation, and demonstrate the performance of the proposed estimators via simulations.
\end{abstract}



\section{Introduction}
\label{sec:intro}
Supervised learning involves modeling the association between a set of responses $Y \in \sY$ and a set of covariates $X \in \sX$ on the basis of $n$ \emph{labeled} realizations $\labeled = \{Z_i\}_{i=1}^n$, where $Z = (X, Y) \in \sZ$. In the semi-supervised setting, we also have access to $N$ additional \emph{unlabeled} realizations of $X$, $\unlabeled = \{X_i \}_{i=n+1}^{n+N}$, without associated realizations of $Y$. This situation could arise, for instance, when observing $Y$ is more expensive or time-consuming than observing $X$. 

Suppose that $X \sim \marginal$, and  $Y\mid X \sim \conditional$.The joint distribution of $Z$ can be expressed as $Z \sim \joint = \marginal\conditional$. Suppose that $\labeled$  consists of  $n$ independent and identically distributed  (i.i.d.) realizations of $\joint$, and $\unlabeled$ consists of $N$ i.i.d. realizations of $\marginal$ that are also independent of $\labeled$. 

We further assume that the marginal distribution and the conditional distribution are separately modeled. That is,  $\marginal$ belongs to a marginal model $\marginalmodel$, and $\conditional$ belongs to a conditional model $\conditionalmodel$. The model of the joint distribution can thus be expressed as
\begin{equation}
\label{eq:separable_model}
\jointmodel = \set{\bP = \bP_X\bP_{Y\mid X}:\bP_X \in \marginalmodel, \bP_{Y\mid X} \in \conditionalmodel}.
\end{equation}

Consider a general inferential problem where the parameter of interest is a Euclidean-valued functional $\theta(\cdot):\jointmodel \to \Theta \subseteq \bR^p$. The ground truth is $\theta(\joint)$, which we denote as $\theta^*$ for simplicity. This paper centers around the following question: \emph{how}, and \emph{by how much}, can inference using $(\labeled, \unlabeled)$ improve  upon inference on $\theta^*$ using only $\labeled$?

\subsection{The three settings}
\label{subsec:settings}
We now introduce three settings that will be used throughout this paper.
\begin{enumerate}
    \item In the \emph{supervised setting}, only labeled data $\labeled$ are available. Supervised estimators can be written as $\hat{\theta}_n(\labeled)$. 
\item In the \emph{ordinary semi-supervised} (OSS) \emph{setting}, both labeled data $\labeled$ and unlabeled data $\unlabeled$ are available. OSS estimators can be written as $\hat{\theta}_{n,N}(\labeled, \unlabeled)$.
\item In the \emph{ideal semi-supervised} (ISS) \emph{setting}  \citep{zhang2019semi, cheng2021robust},  we have access to labeled data $\labeled$, and furthermore \emph{the marginal distribution $\marginal$ is known}. ISS estimators can be written as $\hat{\theta}_n(\labeled, \marginal)$. 
\end{enumerate}
The ISS setting is primarily of theoretical interest: in reality, we rarely know $\marginal$. Analyzing the ISS setting will facilitate analysis of the OSS setting. 

\subsection{Main results}
\label{subsec:main_results}
 Our main contributions are as follows: 

\begin{enumerate}
    \item For an arbitrary inferential problem, we derive the semiparametric efficiency lower bound under the ISS setting. We show that knowledge of $\marginal$ can be used to  potentially improve upon a supervised estimator when the parameter of interest is not \emph{well-specified}, in the sense of  Definition~\ref{def:well_specification}. This generalizes  existing results  that semi-supervised learning cannot improve a correctly-specified linear model \citep{kawakita2013semi, buja2019models, song2023general, gan2023prediction}. 

\item In the OSS setting,  the data are non-i.i.d.,  and consequently classical semiparametric efficiency theory is not applicable. We establish the semiparametric efficiency lower bound in this setting: to our knowledge, this is the first such result in the literature.  As in the ISS setting, an efficiency gain over the supervised setting is possible when
the parameter of interest is not well-specified. 

\item To achieve efficiency gains over supervised estimation  when the parameter is not well-specified, we propose two types of semi-supervised estimators --- the \emph{safe estimator} and the \emph{efficient estimator} ---  both of which build upon an arbitrary initial supervised estimator. The safe estimator requires minimal assumptions, and is always at least as efficient as the initial supervised estimator. The efficient estimator imposes stronger assumptions, and can achieve the semi-parametric efficiency bound. 
 Compared to existing semi-supervised estimators, the proposed estimators are more general and simpler to compute, and enjoy optimality properties.

\item Suppose we have access to a ``black box" machine learning model, $f(\cdot): \mathcal{X} \rightarrow \mathcal{Y}$, that can be used to obtain predictions of $Y$ on unlabeled data. We show that our safe estimator can be adapted to make use of these predictions, thereby connecting the semi-supervised framework of this paper to    \emph{prediction-powered inference} (PPI) \citep{angelopoulos2023prediction,angelopoulos2023ppi++}, a setting of extensive recent interest. This contextualizes existing PPI proposals through the lens of semi-supervised learning, and directly leads to new PPI estimators with better theoretical guarantees and empirical performance. 
%
\end{enumerate}

\subsection{Related literature}
\label{subsec:literature}
While semi-supervised learning is not a new research area \citep{bennett1998semi, chapelle2006continuation, bair2013semi, van2020survey}, it has been the topic of extensive recent theoretical interest. 

Recently, \cite{zhang2019semi} studied semi-supervised mean estimation and proposed a semi-supervised regression-based estimator with reduced  asymptotic variance; \cite{zhang2022high} extended this result to the high-dimensional setting and developed an approach for bias-corrected inference.  Both \cite{chakrabortty2018efficient} and \cite{azriel2022semi}  considered semi-supervised linear regression, and proposed asymptotically normal estimators with improved efficiency over the supervised ordinary least squares estimator. 
\cite{deng2023optimal} further investigated semi-supervised linear regression under a high-dimensional setting and proposed a minimax optimal estimator. \cite{wang2023semi} and \cite{quan2024efficient} extended these findings to the setting of  semi-supervised generalized linear models. 
Some authors have considered semi-supervised learning for more general inferential tasks, such as M-estimation \citep{chakrabortty2016robust,song2023general, yuval2022semi} and U-statistics \citep{kim2024semi}, and in different sub-fields of statistics, such as causal inference \citep{cheng2021robust, chakrabortty2022general} and extreme-valued statistics \citep{ahmed2024extreme}. However,  there exists neither a unified theoretical framework for the semi-supervised setting, nor a unified approach to construct efficient estimators in this setting. 

\emph{Prediction-powered inference} (PPI) refers to a setting in which the data analyst has access to not only  $\labeled$ and $\unlabeled$, but also a ``black box" machine learning model $f(\cdot): \mathcal{X} \rightarrow \mathcal{Y}$ such that $Y \approx f(X)$.  The goal of PPI  is to conduct inference on the association between $Y$ and $X$ while making use of $\labeled$, $\unlabeled$, and the black box predictions given by $f(\cdot)$.  
To the best of our knowledge, despite extensive recent interest in the PPI setting \citep{angelopoulos2023prediction,angelopoulos2023ppi++,miao2023assumption,gan2023prediction, miao2024task, zrnic2024active,zrnic2024cross, gu2024local}, no prior work formally connects prediction-powered inference with the semi-supervised paradigm. 

\subsection{Notation}
For any natural number $K$,  $[K] := \set{1, \ldots, K}$. Let $\sO_1 \times \sO_2$ denote the Cartesian product of  two sets $\sO_1$ and $\sO_2$. Let $\lambda_{\min}(\*A)$ denote the smallest eigenvalue of a matrix $\*A$. For two symmetric matrices $\*A, \*B \in \bR^{p \times p}$, we write that $\*A \succeq \*B$ if and only if $\*A-\*B$ is positive semi-definite. We use uppercase letters $X, Y, Z$ to denote random variables and lowercase letters $x, y, z$ to denote their realizations. For a vector $x \in \bR^p$, let $\norm{x}$ be its Euclidean norm. The random variable $X$ takes values in the probability space $(\sX, \sF_X, \bP)$. For a measurable function $f:\sX \to \bR^p$, let $\bE[f(X)] = \int fd\bP \in \bR^p$ denote its expectation, and  $\Var[f(X)] = \bE[\set{f(X)-\bE[f(X)]}\set{f(X)-\bE[f(X)]}^\top] \in \bR^{p \times p}$  its variance-covariance matrix. For another measurable function $g:\sX \to \bR^q$, let $\Cov[f(X), g(X)] = \bE[\set{f(X)-\bE[f(X)]}\set{g(X)-\bE[g(X)]}^\top] \in \bR^{p \times q}$ denote the covariance matrix between $f(X)$ and $g(X)$. Let $\sL_p^2(\bP)$ be the space of all vector-valued measurable functions $f: \sX \to \bR^p$ such that $\bE[\norm{f(X)}^2] < \infty$, and let $\sL^2_{p, 0}(\bP) \subset \sL_p^2(\bP)$ denote the sub-space of functions with expectation 0, i.e. any $f \in \sL^2_{p, 0}(\bP)$ satisfies $\bE[f(Z)] = 0$.

\subsection{Organization of the paper}
The rest of our paper is organized as follows.  Section~\ref{sec:efficiency_theory} contains a very brief review of some concepts in  semiparametric efficiency theory. In Sections~\ref{sec:ISS} and \ref{sec:OSS}, we derive semi-parametric efficiency bounds and construct efficient estimators under the ISS and OSS settings. 
In Sections~\ref{sec:ppi} and~~\ref{sec:missing}, we connect the proposed framework to prediction-powered inference and missing data, respectively. In Section~\ref{sec:app}, we instantiate our theoretical and methodological findings in the context of specific examples: M-estimation, U-statistics, and the estimation of average treatment effect. Numerical experiments and concluding remarks are in Sections~\ref{sec:simu} and \ref{sec:dis}, respectively. Proofs of theoretical results can be found in the Supplement.

\section{Overview of semiparametric efficiency theory}
\label{sec:efficiency_theory}

We provide a brief introduction to concepts and results from semiparametric efficiency theory that are essential for reading later sections of the paper. See also Supplement~\ref{subsec:apdx_efficiency_iid}. A more comprehensive introduction can be found in Chapter 25 of \cite{van2000asymptotic} or in \cite{tsiatis2006semiparametric}. 

Consider i.i.d. data $\set{Z_i}_{i=1}^n$ sampled from  $\joint \in \jointmodel$,  and suppose there exists a dominating measure $\mu$ such that each element of $\jointmodel$ can be represented by its corresponding density with respect to $\mu$. Interest lies in a Euclidean functional $\theta(\cdot):\jointmodel \to \Theta \subseteq \bR^p$, and  $\theta^* := \theta(\joint)$. Consider a one-dimensional regular parametric sub-model $\sP_T = \set{\bP_t: t \in T \subset \bR} \subset \jointmodel$ such that $\bP_{t^*} = \joint$ for some $t^* \in T$. $\sP_T$ defines a score function at $\joint$. The set of score functions at $\joint$ from all such one-dimensional regular parametric sub-models of $\jointmodel$ forms the \emph{tangent set} at $\joint$ relative to $\jointmodel$, and the closed linear span of the tangent set is the \emph{tangent space} at $\joint$ relative to $\jointmodel$, which we denote as $\sT_\jointmodel(\joint)$. By definition, the tangent space $\sT_\jointmodel(\joint) \subseteq \lpone$. 

An estimator $\thetaest = \hat{\theta}_n(Z_1, \ldots, Z_n)$ of $\theta^*$ is \emph{regular} relative to $\jointmodel$ if, for any regular parametric sub-model $\sP_T = \set{\bP_t: t \in T \subset \bR} \subset \jointmodel$ such that $\joint = \bP_{t^*}$ for some $t^* \in T$,
$$\sqrt{n}\left[\thetaest-\theta\left(\bP_{t^*+\frac{h}{\sqrt{n}}}\right)\right] \indistribution V_\joint$$
for all $h \in \bR$ under $\bP_{t^*+\frac{h}{\sqrt{n}}}$, where $V_\joint$ is a random variable whose distribution depends on $\joint$ but not on $h$. $\thetaest$ is  \emph{asymptotically linear}  with \textit{influence function} $\IF(z) \in \lp$ if
$$\sqrt{n}\left(\thetaest-\theta^*\right) = \frac{1}{\sqrt{n}}\sum_{i=1}^n\IF(z) + \littleO_p(1).$$
The influence function $\IF(z)$ depends on $\joint$ through some functional $\eta: \jointmodel \to \Omega$, where $\Omega$ is a general metric space with metric $\rho$, and $\eta^* = \eta(\joint)$. The functional $\eta(\cdot)$ need not  be finite-dimensional. The functional of interest $\thetafunctional$ is typically involved in $\eta(\cdot)$, but they are usually not the same. 
 A regular and asymptotically linear estimator is \emph{efficient} for $\thetafunctional$ at $\joint$ relative to $\jointmodel$ if its influence function $\EIF(z)$ satisfies  $a^\top\EIF(z) \in \sT_\jointmodel(\joint)$ for any $a \in \bR^p$. If this holds, then  $\EIF(z)$ is the \emph{efficient influence function} of $\thetafunctional$ at $\joint$ relative to $\jointmodel$.

Semiparametric efficiency theory aims to provide a lower bound for the asymptotic variances of regular and asymptotically linear estimators. 
\begin{lemma}
\label{lem:classic_efficiency}
Suppose there exists a regular and asymptotically linear estimator of $\theta^*$. Let $\EIF(z)$ denote the efficient influence function of $\thetaest$ at $\joint$ relative to $\jointmodel$. Then, it follows that for any regular and asymptotically linear estimator $\thetaest$ of $\theta^*$ such that $\sqrt{n}(\thetaest-\theta^*) \indistribution N(0, \*\Sigma)$,
$$\*\Sigma \succeq \Var[\EIF(Z)].$$
\end{lemma}
Lemma~\ref{lem:classic_efficiency} states that any regular and asymptotically linear estimator of $\theta^*$ relative to $\jointmodel$ has asymptotic variance no smaller than the variance of the efficient influence function, $\Var[\EIF(Z)]$. This is referred to as the \emph{semiparametric efficiency lower bound}. For a proof of Lemma~\ref{lem:classic_efficiency}, see Theorem 3.2 of \cite{bickel1993efficient}.

\section{The ideal semi-supervised setting}
\label{sec:ISS}

\subsection{Semiparametric efficiency lower bound}
\label{subsec:ISS_efficiency}
Under the ISS setting, we have  access to labeled data $\labeled$, and  the marginal distribution $\marginal$ is known. Thus,  the  model $\jointmodel$ reduces from \eqref{eq:separable_model}  to
\begin{equation}
\label{eq:reduced_model}
\reducedmodel = \set{\bP= \marginal\bP_{Y\mid X}, \bP_{Y\mid X} \in \conditionalmodel}.
\end{equation}
A smaller model space leads to an easier inferential problem, and thus a lower efficiency bound. We let 
\begin{equation}
\label{eq:cEIF}
\cEIF(x): =\bE\left[\EIF(Z)\mid X = x\right]
\end{equation}
denote the \emph{conditional} efficient influence function. 
Our first result establishes the semiparametric efficiency bound under the ISS setting. 
\begin{theorem}
\label{thm:efficiency_ISS}
Suppose the efficient influence function of $\thetafunctional$ at $\joint$  relative to $\jointmodel$ defined in  \eqref{eq:separable_model}  is $\EIF(z)$, where $\eta: \jointmodel \to \Omega$ is a functional, $\Omega$ is a metric space equipped with metric $\rho(\cdot, \cdot)$, and $\eta^* = \eta(\joint)$. Then, the efficient influence function of $\thetafunctional$ at $\joint$ relative to $\reducedmodel$ defined as \eqref{eq:reduced_model} is 
\begin{equation}
\label{eq:EIF_ISS}
\EIF(z)-\cEIF(x).
\end{equation}
Furthermore, the semiparametric efficiency lower bound  satisfies
\begin{equation}
\label{eq:EIF_var_ISS}
\Var[\EIF(Z)-\cEIF(X)] =  \Var[\EIF(Z)]-\Var[\cEIF(X)] \preceq \Var[\EIF(Z)].
\end{equation}
\end{theorem}
Thus, knowledge of $\marginal$ can improve efficiency at $\theta^*$ if and only if $\Var[\cEIF(X)] \succ 0$, i.e. if and only if  $\cEIF(X)$ is not a constant almost surely. 

\subsection{Inference with a well-specified parameter}
\label{subsec:well_specification}
It is well-known that a correctly-specified regression model cannot be improved with additional unlabeled data \citep{kawakita2013semi, buja2019models, song2023general, gan2023prediction};  \cite{chakrabortty2016robust} further extended this result to the setting of M-estimation. In this section, we formally characterize this phenomenon, and generalize it to an arbitrary inferential problem, in the ISS setting. 

The next definition is motivated by \cite{buja2019modelsII}.  
\begin{definition}[Well-specified parameter]
\label{def:well_specification}
Let $\joint = \marginal\conditional$ be the data-generating distribution, and  $\marginalmodel$ a model of the marginal distribution such that $\bP_X^* \in \sP_X$. A functional $\thetafunctional: \marginalmodel \otimes \conditional \to \bR^p$ is \emph{well-specified} at $\conditional$ relative to $\marginalmodel$ if any $\bP_X \in \marginalmodel$ satisfies
$$\theta(\marginal\conditional) = \theta(\bP_X\conditional).$$
\end{definition}
We emphasize that well-specification is a joint property of the functional of interest  $\thetafunctional$, the conditional distribution $\conditional$, and the marginal model $\marginalmodel$. 
Definition ~\ref{def:well_specification} states that if the parameter is well-specified, then a change to the marginal model does not change $\theta(\cdot)$. Intuitively, if this is the case, then knowledge of $\marginal$ in the  ISS setting will not affect inference on $\theta^*$.

\begin{theorem}
\label{thm:no_improvement_efficiency}
Under the conditions of Theorem \ref{thm:efficiency_ISS}, let $\EIF(z)$ be the efficient influence function of $\thetafunctional$ at $\joint$ relative to $\jointmodel$. If $\thetafunctional$ is well-specified at $\conditional$ relative to $\marginalmodel$, then
\begin{equation}
\label{eq:ceif-zero}
\cEIF(X) = 0, \quad \marginal\text{-a.s.},
\end{equation}
where $\cEIF(x)$ is the conditional efficient influence function as in \eqref{eq:cEIF}. Moreover, if $\tangentmarginal = \lpxone$, then any influence function $\IF$ of a regular and asymptotically linear estimator of $\theta^*$ satisfies
\begin{equation}
\cIF(x) = 0, \quad \marginal\text{-a.s.},
\label{eq:cif-zero}
\end{equation}
where the notation $\cIF(x)$ denotes the conditional influence function, 
\begin{equation}
\label{eq:cIF}
\cIF(x) := \bE[\IF(Z)\mid X = x].
\end{equation}
\end{theorem}
As a direct corollary of Theorem~\ref{thm:no_improvement_efficiency}, if $\thetafunctional$ is well-specified, then knowledge of $\marginal$ does not improve the semiparametric efficiency lower bound.

\begin{corollary}
\label{cor:well_specified_ISS}
Under the conditions of Theorem~\ref{thm:efficiency_ISS}, suppose that $\thetafunctional$ is well-specified at $\conditional$ relative to $\marginalmodel$. Then, the semiparametric efficiency lower bound of 
$\joint$ relative to $\jointmodel$ is the same as the semiparametric efficiency lower bound of $\joint$ relative to $\reducedmodel$.
\end{corollary}

\subsection{Safe and efficient estimators}
\label{subsec:ISS_est}
Together, Theorems~\ref{thm:efficiency_ISS} and~\ref{thm:no_improvement_efficiency} imply that when  the parameter is not well-specified, we can potentially use knowledge of $\marginal$ for more efficient inference. We will now present two such approaches, both of which  build upon an initial supervised estimator. Under minimal assumptions, the  \emph{safe estimator}  is always at least as efficient as the initial supervised estimator. By contrast, under a stronger set of assumptions,  the \emph{efficient estimator}  achieves the efficiency lower bound \eqref{eq:EIF_var_ISS} under the ISS setting. 

We first provide some intuition behind the two estimators. Suppose that $\eta: \jointmodel \to \Omega$ is a functional that takes values in a metric space $\Omega$ with metric $\rho(\cdot, \cdot)$ and $\eta^* = \eta(\joint)$. Let $\thetaest = \hat{\theta}(\labeled)$ denote an initial supervised estimator of $\theta^*$ that is regular and asymptotically linear with influence function $\IF(z)$. Motivated by the form of the efficient influence function \eqref{eq:EIF_ISS} under the ISS setting, we aim to use knowledge of $\marginal$ to obtain an estimator $\cIFest(x)$ of the conditional influence function $\cIF(x) = \bE[\IF(Z)\mid X=x]$. This will lead to a new estimator of the form
\begin{equation}
\label{eq:corrected_est}
\hat\theta_{n,\marginal} = \thetaest - \frac{1}{n}\sum_{i=1}^n\cIFest(X_i).
\end{equation}
In what follows, we let $\cetaest$ denote an estimator of $\eta^*$. 

\subsubsection{The safe estimator}

Let $g: \sX \to \bR^d$ denote an arbitrary measurable function of $x$, and define 
\begin{equation}
\label{eq:centered_ISS}
\centeredg(x) := g(x)-\bE[g(X)]. 
\end{equation}
(Since $\marginal$ is known in the ISS setting, we can compute $\centeredg(\cdot)$ given  $g(\cdot)$.)

To construct the safe estimator, we will estimate $\cIF(x)$ by regressing $\set{\IFplugin(Z_i)}_{i=1}^n$ onto $\set{\centeredg(X_i)}_{i=1}^n$, leading to regression coefficients 
\begin{equation}
\label{eq:OLS_ISS}
\iOLS = \left[\frac{1}{n}\sum_{i=1}^n\IFplugin(Z_i)\centeredg(X_i)^\top\right]\bE\left[\centeredg(X)\centeredg(X)^\top\right]^{-1},
\end{equation}
where the expectation in \eqref{eq:OLS_ISS} is possible since $\marginal$ is known. 
 Then the safe ISS estimator takes the form 
\begin{equation}
\label{eq:corrected_est_proj}
\issthetaest = \thetaest - \frac{1}{n}\sum_{i=1}^n\iOLS\centeredg(X_i).
\end{equation}

To establish the convergence of $\issthetaest$, we made the following assumptions.
\begin{assumption}
\label{asu:IF_Lipschitz}
    (a) The initial supervised estimator $\thetaest = \hat{\theta}(\labeled)$ is a regular and asymptotically linear estimator of $\theta^*$ with influence function $\IF(z)$. 
    (b) There exists a set $\sO \subset \Omega$ such that $\eta^* \in \sO$, the class of functions $\set{\vp_\eta(z): \eta \in \sO}$ is $\joint$-Donsker, and for all $\set{\eta_1, \eta_2} \subset \sO$, $\norm{\vp_{\eta_1}(z)-\vp_{\eta_2}(z)} \le L(z)\rho(\eta_1, \eta_2)$, where $L:\sZ \to \bR^+$ is a square integrable function.
    (c) There exists an estimator $\cetaest$ of $\eta^*$ such that $\joint\set{\cetaest \in \sO} \to 1$ and $\rho(\cetaest, \eta^*) = \littleO_p(1)$. 
\end{assumption}


Assumptions~\ref{asu:IF_Lipschitz} (b) and (c) state that we can find a consistent estimator $\cetaest$ of the functional $\eta^*$ on which the influence function of $\thetaest$ depends, such that   $\cetaest$ asymptotically belongs to a realization set $\sO$ on which the class of functions $\set{\vp_\eta(z):\eta \in \sO}$ is a Donsker class. The Donsker condition is standard in semiparametric statistics, and leads to $\sqrt{n}$-converegnce while allowing for rich classes of functions. We validate Assumption~\ref{asu:IF_Lipschitz} for specific inferential problems in Section~\ref{sec:app}. In the special case when $\eta(\cdot)$ is finite-dimensional, the next proposition provides sufficient conditions for   Assumption~\ref{asu:IF_Lipschitz}.

\begin{proposition}
\label{prop:equi_assumption1}
Suppose that (i) $\eta(\cdot)$ is a finite-dimensional functional; (ii) there exists a bounded open set $\sO$ such that $\eta^* \in \sO$, and for all $\set{\eta_1, \eta_2} \subset \sO$, $\norm{\vp_{\eta_1}(z)-\vp_{\eta_2}(z)} \le L(z)\norm{\eta_1-\eta_2}$, where $L:\sZ \to \bR^+$ is a square integrable function; and (iii) there exists an estimator $\cetaest$ of $\eta^*$ such that $\norm{\cetaest-\eta^*} = \littleO_p(1)$. Then, Assumptions~\ref{asu:IF_Lipschitz} (b) and (c) hold.
\end{proposition}

Define the population regression coefficient as
\begin{equation}
\label{eq:OLS}
\OLS := \bE\left[\IF(Z)\centeredg(X)^\top\right]\bE\left[\centeredg(X)\centeredg(X)^\top\right]^{-1}.
\end{equation}
The next theorem establishes the asymptotic behavior of $\issthetaest$ in \eqref{eq:corrected_est_proj}.
\begin{theorem}
\label{thm:ISS_estimator}
Suppose that $\hat{\theta}_n = \hat{\theta}(\labeled)$ is a supervised estimator that satisfies Assumption~\ref{asu:IF_Lipschitz}, and $\cetaest$ is an estimator of $\eta^*$ as in Assumption~\ref{asu:IF_Lipschitz} (c). Let $g: \sX \to \bR^d$ be a square-integrable function such that $\bE\left[\norm{g(X)}^2\right] < \infty$ and $\Var\left[g(X)\right]$ is non-singular, and let $\centeredg(x)$ be its centered version \eqref{eq:centered_ISS}. Then, for $\OLS$ in \eqref{eq:OLS},  $\issthetaest$ defined in \eqref{eq:corrected_est_proj} is a regular and asymptotically linear estimator of $\theta^*$ with influence function 
$$\IF(z)- \OLS\centeredg(x),$$ and asymptotic variance  $\Var\left[\IF(Z)- \OLS\centeredg(X)\right]$. Moreover,  
\begin{align}
\label{eq:ISS_est_var}
\Var\left[\IF(Z)- \OLS\centeredg(X)\right] &= \Var\left[\IF(Z)- \cIF(X)\right] + \Var\left[\cIF(Z)- \OLS\centeredg(X)\right] \nonumber \\
&\preceq \Var\left[\IF(Z)\right].
\end{align}
\end{theorem}
Theorem~\ref{thm:ISS_estimator} establishes that $\issthetaest$ is always at least as efficient as the initial supervised estimator $\thetaest$ under the ISS setting.

\begin{remark}[Choice of  regression basis function $g(x)$]
\label{rmk:choice_g_ISS}
 Theorem~\ref{thm:ISS_estimator} reveals that the  asymptotic variance of $\issthetaest$ is the sum of two terms. 
The first term, $\Var\left[\IF(Z)- \cIF(X)\right]$,  does not depend on $g(x)$. The second term, $\Var\left[\cIF(Z)- \OLS\centeredg(X)\right]$,  can be interpreted as the approximation error that arises from approximating $\cIF(x)$ with  $\OLS\centeredg(x)$. Thus, $\issthetaest$ is more efficient when $\OLS\centeredg(x)$ is a better approximation of $\cIF(x)$.
\end{remark}

\subsubsection{The efficient estimator}

To construct the efficient estimator, we will approximate $\cIF(x)$ by regressing $\set{\IFplugin(Z_i)}_{i=1}^n$   onto a growing number of basis functions. 
 Let $\sC^{\alpha}_{M}(\sX)$ denote the Hölder class of functions on $\sX$, 
\small
\begin{equation}
\label{eq:holder_space}
\sC^{\alpha}_{M}(\sX):= \set{f: \sX \to \bR,\text{ } \sup_\sX\lvert D^kf(x)\rvert + \sup_{x_1,x_2 \in \sX}\frac{\lvert D^kf(x_1)-D^kf(x_2)\rvert}{\norm{x_1-x_2}} \le M, \text{ }\forall k \in [\alpha]}.    
\end{equation}
\normalsize

\begin{assumption}\label{asu:IF_holder}
Under Assumption~\ref{asu:IF_Lipschitz}, the set $\sO$ additionally satisfies that: $\set{\phi_\eta(x):\eta\in\sO} \subset \sS^{\alpha}_{M}$, where $\phi_\eta(x)$ is the conditional influence function $\phi_\eta(x) = \bE[\vp_\eta(Z)\mid X = x]$ and 
$\sS^{\alpha}_{M}$ is the class of functions defined as, $$\sS^{\alpha}_{M} := \set{f: \sX \to \bR^p, \quad [f]_j \in \sC^{\alpha_j}_{M_j}(\sX),\quad \forall j \in [p]}.$$ 
In the definition of $\sS^{\alpha}_{M}$, $[f]_j$ is the $j$-th coordinate of a vector-valued function $f$, $\alpha = \min_{j \in [p]}\set{\alpha_k}$, and $M = \max_{j \in [p]}\set{M_j}$. 
\end{assumption}


Under Assumption~\ref{asu:IF_holder}, there exists a set of basis functions $\set{g_k(x)}_{k=1}^\infty$ of $\sS_M^\alpha$, such as a spline basis or a polynomial basis, such that any $f \in \sS_M^\alpha$ can be represented as an infinite linear combination
$f(x) = \sum_{k=1}^\infty \*A_k g_k(x)$,
for coefficients $\*A_k \in \bR^{p \times p}$. Define $\GKn(x) = \left[g_1(x)^\top, \ldots, g_{K_n}(x)^\top\right]^\top$ as the concatenation of the first $K_n$ basis functions. The optimal $\sL_2(\marginal)$-approximation of $f$ by the first $K_n$ basis functions is then $\hat{f}_{K_n} = \npOLS\GKn(x)$, where $\npOLS$ is the population regression coefficient,
$$\npOLS = \bE\left[f(X)G_k(X)^\top\right]\bE\left[G_k(X)[G_k(X)^\top\right]^{-1}.$$
It can be shown that the approximation error arising from the first $K_n$ basis functions satisfies
\begin{equation}
\label{eq:approx_error}
\norm{f-\hat{f}_{K_n}}^2_\lpx = O\left(K_n^{-\frac{2\alpha}{\dim(\sX)}}\right),
\end{equation}
where $\dim(\sX)$ is the dimension of $\sX$ \citep{newey1997convergence}. 


To estimate the conditional influence function $\cIF(x) = \bE[\IF(Z)\mid X =x]$, we define the centered basis functions 
\begin{equation}
\label{eq:centered_basis}
\centeredGKn := \GKn -\bE[\GKn(X)],
\end{equation}
 and regress $\set{\IFplugin(Z_i)}_{i=1}^n$ onto $\set{\centeredGKn(X_i)}_{i=1}^n$. (Note that \eqref{eq:centered_basis} can be computed  because $\marginal$ is known.) The resulting regression coefficients  are
\begin{equation}
\label{eq:B_np_OLS_ISS}
\npiOLS = \left[\frac{1}{n}\sum_{i=1}^n\IFplugin(Z_i)\centeredGKn(X_i)^\top\right]\bE\left[\centeredGKn(X_i)\centeredGKn(X_i)^\top\right]^{-1}.
\end{equation}
The nonparametric least squares estimator of $\cIF(x)$ is $\cIFest(x) = \npiOLS\centeredGKn(x)$, and the efficient ISS estimator is
\begin{equation}
\label{eq:est_NP_ISS}
\issthetanpest = \thetaest -\frac{1}{n}\sum_{i=1}^n\npiOLS\centeredGKn(X_i).
\end{equation}
Define 
$\zeta_n^2 := \sup_{x \in \sX} \left\{ \centeredGKn(x)^\top\Var[\centeredGKn(X)]^{-1}\centeredGKn(x) \right\}$. The next theorem establishes the theoretical properties of $\issthetanpest$. 
\begin{theorem}
\label{thm:ISS_efficient_estimator}
Suppose that $\hat{\theta}_n = \hat{\theta}(\labeled)$ is a supervised estimator that satisfies Assumptions~\ref{asu:IF_Lipschitz}  and~\ref{asu:IF_holder} and $\cetaest$ is an estimator of $\eta^*$ as in Assumption~\ref{asu:IF_Lipschitz} (c). Further suppose that $\set{g_k(x)}_{k=1}^\infty$ is a set of basis functions of $\sS_M^{\alpha}$ for which $\GKn(x) = \left[g_1(x)^\top, \ldots, g_{K_n}(x)^\top\right]^\top$ satisfies \eqref{eq:approx_error} and $$\inf_{K_n}\set{\lambda_{\min}\left(\Var\left[\GKn(X)\right]\right)} > 0.$$ Let $\centeredGKn(x)$ be the centered version of $\GKn(x)$ as in \eqref{eq:centered_basis}. 
If $\alpha > \dim(\sX)$, $K_n \to \infty$, $K_n\rho(\cetaest, \eta^*) \to 0$, and  $\frac{\zeta_n^2}{n} \to 0$, then  $\issthetanpest$ in \eqref{eq:est_NP_ISS} is a regular and asymptotically linear estimator of $\theta^*$ with influence function 
$$\IF(z)-\cIF(x)$$
under the ISS setting. Moreover, its asymptotic variance  is $\Var\left[\IF(Z)- \cIF(X)\right]$, which satisfies
\begin{equation}
\label{eq:ISS_np_est_var}
\Var\left[\IF(Z)- \cIF(X)\right] \preceq \Var\left[\IF(Z)- \OLS\centeredg(X)\right] \preceq \Var\left[\IF(Z)\right],
\end{equation}
for arbitrary $g(x)$, $\centeredg(x)$, and $\OLS$ as in Theorem~\ref{thm:ISS_estimator}.
\end{theorem}
Theorem~\ref{thm:ISS_efficient_estimator} shows that the efficient  estimator $\issthetanpest$ is always at least as efficient as both the safe  estimator $\issthetaest$ and the initial supervised estimator $\thetaest$. Further, \eqref{eq:ISS_np_est_var} and Theorem~\ref{thm:efficiency_ISS} suggest that if the initial supervised estimator is efficient under the supervised setting, then the efficient estimator is efficient under the ISS setting. 

\begin{remark}[Role of the marginal distribution $\marginal$]
\label{rmk:dependence_marginal}
The safe estimator  $\issthetaest$ involves  $\bE[g(X)]$ and $\bE\left[\centeredg(X)\centeredg(X)^\top\right]$, and the efficient estimator $\issthetanpest$
involves $\bE[\GKn(X)]$ and $\bE\left[\centeredGKn(X)\centeredGKn(X)^\top\right]$. All four of these quantities  require knowledge of $\marginal$.  
However, Theorems~\ref{thm:ISS_estimator} and \ref{thm:ISS_efficient_estimator} would hold
 even if $\bE\left[\centeredg(X)\centeredg(X)^\top\right]$ and 
 $\bE\left[\centeredGKn(X)\centeredGKn(X)^\top\right]$
 were replaced by their empirical counterparts,  $\frac{1}{n}\sum_{i=1}^n\centeredg(X_i)\centeredg(X_i)^\top$ and $\frac{1}{n}\sum_{i=1}^n\centeredGKn(X_i)\centeredGKn(X_i)^\top$. 
\end{remark}

\section{The ordinary semi-supervised setting}
\label{sec:OSS}

\subsection{Semiparametric efficiency lower bound}
\label{subsec:OSS_efficiency}

Because the OSS setting offers more information about the unknown parameter than the supervised setting but less information than the ISS setting, it is natural to think that the semiparametric efficiency bound in the OSS setting should be lower than in the supervised setting but higher than in the ISS setting. In this section, we will formalize this intuition. 


In the OSS setting, the full data $\labeled\cup \unlabeled$  are not jointly i.i.d.. Therefore, the classic theory of semiparametric efficiency for i.i.d. data, as described in Section~\ref{sec:efficiency_theory}, is not applicable. \cite{bickel2001inference} provides a framework for efficiency theory for non-i.i.d. data. Inspired by their results, our next theorem establishes the semiparametric efficiency lower bound under the OSS setting. 
To state Theorem~\ref{thm:efficiency_OSS}, we must  first adapt some definitions from Section~\ref{sec:efficiency_theory} to the OSS setting. 
Additional definitions necessary for its proof are deferred to Appendix~\ref{sec:apdx_efficiency_noniid}. 
\begin{definition}[Regularity in the OSS setting]
\label{def:regular_OSS}
An estimator $\thetaestnN = \thetaestnN(\labeled\cup\unlabeled)$ of $\theta^*$ is \emph{regular} relative to a model $\jointmodel$
if, for every regular parametric sub-model $\set{\bP_t:t \in T \subset \bR}\subset \jointmodel$ such that $\bP_{t^*} = \joint$ for some $t^* \in T$,
\begin{equation}
\label{eq:regularity_OSS}
\sqrt{n}\left[\thetaestnN-\theta\left(\bP_{t^*+\frac{h}{\sqrt{n}}}\right)\right] \indistribution V_\joint, 
\end{equation}
for all $h \in \bR$ under $\bP_{t^*+\frac{h}{\sqrt{n}}}$ and $\frac{N}{n+N} \to \gamma \in (0,1)$, where $V_\joint$ is some random variable whose distribution depends on $\joint$ but not on $h$. 
\end{definition}

\begin{definition}[Asymptotic linearity in the OSS setting]
\label{def:AL_OSS}
An estimator $\thetaestnN$ of $\theta^*$ is \emph{asymptotically linear} with influence function $\IF: \sZ \times \sX \rightarrow \bR^p$, where
\begin{equation}
\label{eq:OSS_IF}
\IF(z_1, x_2)=\IFa(z_1)+ \IFb(x_2)
\end{equation}
for some $\IFa(z) \in \lp$ and $\IFb(x) \in \lpx$, if
\begin{equation}
\label{eq:AL}
\sqrt{n}\left(\thetaestnN-\theta^*\right) = \frac{1}{\sqrt{n}}\sum_{i=1}^n\IFa(Z_i) + \frac{1}{\sqrt{N}}\sum_{i=n+1}^{n+N}\IFb(X_i)+\littleO_p(1),
\end{equation}
for $\frac{N}{n+N} \to \gamma \in (0,1]$.
\end{definition}

\begin{remark} \label{remark:subscripts} In the OSS setting, the arguments of the influence function are in the product space $\sZ \times \sX$. Furthermore, since $\sZ = (\sX, \sY)$,  this product space can also be written as $(\sX, \sY) 
\times \sX$. Thus, to avoid notational confusion, when referring to a function on this space, we use arguments $(z_1, x_2)$: that is, these subscripts are intended to disambiguate the role of the two arguments in the product space.
\end{remark}
The form of the influence function \eqref{eq:OSS_IF} arises from the fact that the data consist of two i.i.d. parts, $\labeled$ and $\unlabeled$, which are independent of each other.
 By Definition~\ref{def:AL_OSS}, if $\thetaestnN$ is asymptotically linear, then 
\begin{equation}
\notag
\sqrt{n}\left(\thetaestnN-\theta^*\right) \indistribution N\left(0, \Var\left[\IFa(Z)\right] + \Var\left[\IFb(X)\right]\right).
\end{equation}
That is, the asymptotic variance of an asymptotically linear estimator is the the sum of the variances of the two components of its influence function. 

Armed with these two definitions,  Lemma~\ref{lem:classic_efficiency} can be generalized to the OSS setting. Informally, there exists a unique efficient influence function of the form \eqref{eq:OSS_IF}, and any regular and asymptotically linear estimator whose influence function equals the efficient influence function is an efficient estimator. 
Additional details are provided in Appendix~\ref{sec:apdx_efficiency_noniid}. 

The next theorem identifies the efficient influence function and the semiparametric efficiency lower bound under the OSS.   
\begin{theorem}
\label{thm:efficiency_OSS}
Let $\jointmodel$ be defined as in \eqref{eq:separable_model}, and let $\frac{N}{n+N} \to \gamma \in (0,1)$. If the efficient influence function of $\thetafunctional$ at $\joint$ relative to $\jointmodel$ is $\EIF$, then 
the efficient influence function of $\thetafunctional$ at $\joint$ relative to $\jointmodel$ under the OSS setting is 
\begin{equation}
\label{eq:EIF_OSS}
\left[\EIF(z_1) - \gamma\cEIF(x_1)\right]+ \sqrt{\gamma(1-\gamma)}\cEIF(x_2),
\end{equation}
where $\cEIF(x)$ is the conditional efficient influence function defined in \eqref{eq:cEIF}. 
Moreover, the semiparametric efficiency bound can be expressed as
\begin{equation}
\label{eq:EIF_var_OSS}
\begin{aligned}
&\Var\left[\EIF(Z) -\gamma\cEIF(X)\right] +  \Var\left[\sqrt{\gamma(1-\gamma)}\cEIF(X)\right]\\
&= \Var\left[\EIF(Z)\right] - \gamma\Var\left[\cEIF(X)\right]\\
&= \Var\left[\EIF(Z)-\cEIF(X)\right] + (1-\gamma)\Var\left[\cEIF(X)\right].\\
\end{aligned}
\end{equation}
\end{theorem}
Theorem~\ref{thm:efficiency_OSS} confirms our intuition for the efficiency bound under the OSS setting. By definition, the efficiency bound under the supervised setting is $\Var\left[\EIF(Z)\right]$. The second line of \eqref{eq:EIF_var_OSS} reveals that the efficiency bound under the OSS setting is smaller by the amount  $\gamma\Var\left[\cEIF(X)\right]$. Moreover, in Theorem~\ref{thm:efficiency_ISS} we showed that the efficiency  bound under the ISS setting is $\Var\left[\EIF(Z)-\cEIF(X)\right]$. The third line of \eqref{eq:EIF_var_OSS} shows that the efficiency bound under the OSS setting exceeds this by the amount $(1-\gamma)\Var\left[\cEIF(X)\right]$.

The efficiency bound \eqref{eq:EIF_var_OSS} depends on the limiting proportion of unlabeled data, $\gamma$.  Intuitively, when $\gamma$ is large, we have more unlabeled data, and  the efficiency bound in ~\eqref{eq:EIF_var_OSS}  improves.  
The special cases where $\gamma=0$ and $\gamma=1$ are particularly instructive. When $\gamma = 0$, the amount of unlabeled data is negligible, and hence we should expect no improvement in efficiency over the supervised setting: this intuition is confirmed by setting $\gamma \to 0$ in~\eqref{eq:EIF_var_OSS}. When $\gamma = 1$, there are many more labeled than unlabeled observations; thus, it is as if we know the marginal distribution $\marginal$. Letting $\gamma \rightarrow 1$,  the efficiency lower bound agrees with that under the ISS~\eqref{eq:EIF_var_ISS}.

We saw in Corollary~\ref{cor:well_specified_ISS} that when the functional of interest is well-specified, the efficiency bound under the ISS setting is the same as in the supervised setting. As an immediate corollary of Theorem~\ref{thm:efficiency_OSS} and Theorem~\ref{thm:no_improvement_efficiency}, we now show that the same result holds under the OSS setting.
\begin{corollary}
\label{cor:well_specified_OSS}
Under the conditions of Theorem~\ref{thm:efficiency_OSS}, let $\EIF(z)$ be  the efficient influence function  of $\thetafunctional$ at $\theta^*$ relative to $\jointmodel$. If $\thetafunctional$ is well-specified at $\conditional$ relative to $\marginalmodel$ in the sense of  Definition~\ref{def:well_specification}, 
then the efficient influence function of $\thetafunctional$ at $\theta^*$ relative to $\jointmodel$ under the OSS is $\EIF(z_1)$.
\end{corollary}

Thus, an efficient supervised estimator of a well-specified parameter can never be improved via the use of unlabeled data.

\subsection{Safe and efficient estimators}
\label{subsec:OSS_est}

 Similar to Section~\ref{subsec:ISS_est}, in this section we provide two types of OSS estimators, a safe estimator and an efficient estimator, both of which build upon and improve  an initial supervised estimator.


Recall from Remark~\ref{rmk:dependence_marginal} that the safe and efficient estimators proposed in the ISS setting require knowledge of  $\marginal$. Of course, in the OSS setting, $\marginal$ is unavailable. 
Thus, we will simply replace $\marginal$ in \eqref{eq:corrected_est_proj} and \eqref{eq:est_NP_ISS} with $\bP_{n+N,X}$, the empirical marginal distribution of the labeled and unlabeled covariates,  $\set{X_i}_{i=1}^{n+N}$. 

\subsubsection{The safe estimator}

For an arbitrary measurable function of $x$, $g: \sX \to \bR^d$, recall that in the ISS setting, the safe estimator \eqref{eq:corrected_est_proj} made use of  its centered version $\centeredg(x) := g(x) -\bE[g(X)]$. Since $\bE[g(X)]$ is unknown under the OSS setting, we  define the empirically centered version of $g$ as  
\begin{equation}
\label{eq:center_OSS}
\ecenteredg(x) = g(x) - \frac{1}{n+N}\sum_{i=1}^{n+N}g(X_i).
\end{equation}
Now, regressing $\set{\IFplugin(Z_i)}_{i=1}^n$ onto $\set{\ecenteredg(X_i)}_{i=1}^n$ yields the regression coefficients 
\begin{equation}
\label{eq:OLS_OSS}
\oOLS = \left[\frac{1}{n}\sum_{i=1}^n\vp_{\cetaest}(Z_i)\ecenteredg(X_i)^\top\right]\left[\frac{1}{n+N}\sum_{i=1}^{n+N}\ecenteredg(X_i)^\top\ecenteredg(X_i)\right]^{-1}.
\end{equation}
The regression estimator of the conditional influence function is thus $\cIFest(x) = \oOLS \ecenteredg(x)$, and the corresponding safe OSS estimator is defined as 
\begin{equation}
\label{eq:est_OSS}
\ossthetaest = \thetaest - \frac{1}{n}\sum_{i=1}^n\oOLS \ecenteredg(X_i).
\end{equation}
Under the same conditions as in Theorem~\ref{thm:ISS_estimator}, we establish the asymptotic behavior of $\ossthetaest$.
\begin{theorem}
\label{thm:OSS_estimator}
Suppose that $\hat{\theta}_n = \hat{\theta}(\labeled)$ is a supervised estimator that satisfies Assumption~\ref{asu:IF_Lipschitz}, and $\cetaest$ is an estimator of $\eta^*$ as in Assumption~\ref{asu:IF_Lipschitz} (c). Let $g: \sX \to \bR^d$ be a square-integrable function such that $\bE\left[\norm{g(X)}^2\right] < \infty$ and $\Var\left[g(X)\right]$ is non-singular, and let $\ecenteredg(x)$ be its empirically centered version \eqref{eq:center_OSS}. Suppose that $\frac{N}{n+N} \to \gamma \in (0,1]$. Then, the  estimator $\ossthetaest$ defined in \eqref{eq:est_OSS} is a regular and asymptotically linear estimator of $\theta^*$ in the sense of Definitions~\ref{def:regular_OSS} and~\ref{def:AL_OSS},  with influence function 
\begin{equation}
\label{eq:IF_OSS}
\left[\IF(z_1)- \gamma\OLS\centeredg(x_1)\right]+ \sqrt{\gamma(1-\gamma)}\OLS\centeredg(x_2)
\end{equation}
under the OSS setting, where $\OLS$ is defined in \eqref{eq:OLS}. 
Furthermore, the asymptotic variance of $\ossthetaest$ takes the form  
\begin{equation}
\label{eq:sigma_safe}
\Sigmasafe(\gamma) = \Var\left[\IF(Z)- \gamma\OLS\centeredg(X)\right] + \Var\left[\sqrt{\gamma(1-\gamma)}\OLS\centeredg(X)\right],
\end{equation}
and satisfies
\begin{equation}
\label{eq:OSS_est_var}
\begin{aligned}
&\Var\left[\IF(Z)-\OLS\centeredg(X)\right] \preceq \Sigmasafe(\gamma)\preceq \Var\left[\IF(Z)\right],\quad \forall \gamma \in [0,1].
\end{aligned}
\end{equation}
\end{theorem}
Theorem~\ref{thm:OSS_estimator} shows that $\ossthetaest$ is always at least as efficient as the initial supervised estimator $\thetaest$ under the OSS setting. Thus, it is  a safe alternative to $\thetaest$ when additional unlabeled data are available.

\subsubsection{The efficient estimator}

As in Section~\ref{subsec:ISS_est}, we require that the initial supervised estimator $\thetaest$ satisfy  Assumption~\ref{asu:IF_holder}.
Then, for a suitable set of basis functions $\set{g_k(x)}_{k=1}^\infty$ of $\sS_M^\alpha$ (defined in Assumption~\ref{asu:IF_holder}), we define $\GKn(x) = \left[g_1(x)^\top, \ldots, g_{K_n}(x)^\top\right]^\top$. 
We then center $\GKn$ with its empirical mean, $\frac{1}{n+N}\sum_{i=1}^{n+N}\GKn(X_i)$, leading to 
\begin{equation}
\label{eq:ecenteredGKn}
\ecenteredGKn = \GKn(x) - \frac{1}{n+N}\sum_{i=1}^{n+N}\GKn(X_i).
\end{equation}
The nonparametric least squares estimator of the conditional influence function is  $\cIFest(x) = \npoOLS\ecenteredGKn(x)$, where 
\begin{equation}
\label{eq:npoOLS}
\npoOLS = \left[\frac{1}{n}\sum_{i=1}^n\IFplugin(Z_i)\ecenteredGKn(X_i)^\top\right]\left[\frac{1}{n+N}\sum_{i=1}^{n+N}\ecenteredGKn(X_i)\ecenteredGKn(X_i)^\top\right]^{-1} \in \bR^{p \times (pK_n)}
\end{equation}
 are the coefficients obtained from regressing $\set{\IFplugin(Z_i)}_{i=1}^n$ onto $\set{\ecenteredGKn(X_i)}_{i=1}^n$. 
The efficient OSS estimator is then
\begin{equation}
\label{eq:est_NP_OSS}
\ossthetanpest = \thetaest -\frac{1}{n}\sum_{i=1}^n\npoOLS\ecenteredGKn(X_i).
\end{equation}
The next theorem establishes the asymptotic properties  of $\ossthetanpest$. 

\begin{theorem}
\label{thm:OSS_efficient_estimator}
Suppose that the supervised estimator $\hat{\theta}_n = \hat{\theta}(\labeled)$  satisfies Assumptions~\ref{asu:IF_Lipschitz} and~\ref{asu:IF_holder}, and $\cetaest$ is an estimator of $\eta^*$ as in Assumption~\ref{asu:IF_Lipschitz} (c).  Further, suppose that $\set{g_k(x)}_{k=1}^\infty$ is a set of basis functions of $\sS_M^{\alpha}$ such that $\GKn(x) = \left[g_1(x)^\top, \ldots, g_{K_n}(x)^\top\right]^\top$  satisfies \eqref{eq:approx_error} for any $f \in \sS_M^\alpha$, and $\inf_{K_n}\set{\lambda_{\min}\left(\Var\left[\GKn(X)\right]\right)} > 0$. 
Suppose that $\frac{N}{n+N} \to \gamma \in (0,1]$, $\alpha > \dim(\sX)$,  $K_n \to \infty$ and $K_n\rho(\cetaest, \eta^*) \to 0$, and  $\frac{\zeta_n^2}{n} \to 0$. Then, $\ossthetanpest$ defined in \eqref{eq:est_NP_OSS} is a regular and asymptotically linear estimator of $\theta^*$ in the sense of Definitions~\ref{def:regular_OSS} and~\ref{def:AL_OSS},  and has influence function 
\begin{equation}
    [\IF(z_1)-\gamma\cIF(x_1)] + \sqrt{\gamma(1-\gamma)}\cIF(x_2).
\end{equation}
Furthermore,  the asymptotic variance of $\ossthetanpest$ takes the form   
\begin{equation}
\label{eq:sigma_eff}
\Sigmaeff(\gamma) = \Var\left[\IF(Z)- \gamma\cIF(X)\right] + \Var\left[\sqrt{\gamma(1-\gamma)}\cIF(X)\right],
\end{equation}
and satisfies
\begin{equation}
\label{eq:OSS_np_est_var}
\begin{aligned}
&\Var\left[\IF(Z)-\cIF(X)\right] \preceq \Sigmaeff(\gamma)\preceq \Sigmasafe(\gamma) \preceq  \Var\left[\IF(Z)\right], \quad \forall \gamma \in [0,1],
\end{aligned}
\end{equation}
where $\Sigmasafe(\gamma)$ is defined as ~\eqref{eq:sigma_safe}.
\end{theorem}
As in Section~\ref{subsec:ISS_est},  the efficient OSS estimator, $\ossthetanpest$, is always at least as efficient as both the safe OSS estimator, $\ossthetaest$, and the initial supervised estimator $\thetaest$.
Furthermore, \eqref{eq:OSS_np_est_var} together with Theorem~\ref{thm:efficiency_OSS} show that if the initial supervised estimator is efficient under the supervised setting, then the efficient OSS estimator is efficient under the OSS setting. 

\begin{remark}
We noted in Section~\ref{subsec:OSS_efficiency} that the efficiency bound in the OSS falls between the efficiency bounds in the ISS and supervised settings. We can see from Theorems~\ref{thm:OSS_estimator} and \ref{thm:OSS_efficient_estimator} that a similar property holds for the safe and efficient estimators. Specifically, \eqref{eq:OSS_est_var} and \eqref{eq:OSS_np_est_var} show that the safe and efficient OSS estimators are more efficient than the initial supervised estimator $\thetaest$, but are less efficient than the  safe and efficient ISS estimators, respectively. This is again due to the fact that under the OSS, unlabeled observations provide more information than is available under the supervised setting, but less than is available under the ISS. Similarly, as $\gamma \to 0$, i.e.  as the proportion of unlabeled data becomes negligible, the OSS estimators are asymptotically equivalent to the initial supervised estimator. On the other hand,  when $\gamma = 1$, the OSS estimators are asymptotically equivalent to the corresponding ISS estimators. 
\end{remark}

\section{Connection to prediction-powered inference}
\label{sec:ppi}

Suppose now that in addition to  $\labeled$ and $\unlabeled$, the data analyst also has access to  $K \geq 1$ machine learning prediction models $f_k:\sX \to \sY$, $k \in [K]$, which are \emph{independent} of $\labeled$ and $\unlabeled$ (e.g., they were trained on independent data). For instance, $f_1,\ldots,f_K$ may arise from  black-box machine learning models such as neural networks or large language models. It is clear that this is a special case of semi-supervised learning, as $\set{f_k}_{k=1}^K$ can be treated as fixed functions conditional on the data upon which they were trained. 

Recently, \cite{angelopoulos2023prediction} proposed \emph{prediction-powered inference} (PPI), which provides a principled approach for making use of $\set{f_k}_{k=1}^K$. Subsequently, a number of PPI variants have been proposed to further improve statistical efficiency or extend PPI to other settings \citep{angelopoulos2023ppi++, miao2023assumption, gan2023prediction,miao2024task,gu2024local}. In this section, we re-examine the PPI problem through the lens of our results in previous sections, and apply these insights to improve upon existing PPI estimators.
 

Since $f_1,\ldots,f_K$ are  independent of $\labeled\cup\unlabeled$, existing PPI estimators fall into the category of OSS estimators, and can be shown to be regular and asymptotically linear in the sense of Definitions~\ref{def:regular_OSS} and~\ref{def:AL_OSS}. Therefore, Theorem~\ref{thm:efficiency_OSS} suggests that their asymptotic variances are lower bounded by the efficiency bound \eqref{eq:EIF_var_OSS}. We show in Supplement~\ref{sec:apdx_ppi} that existing PPI estimators cannot achieve the efficiency bound \eqref{eq:EIF_var_OSS} in the OSS setting, unless strong assumptions are made on the machine learning prediction models. Furthermore, if the parameter of interest is well-specified in the sense of Definition~\ref{def:well_specification}, then by Corollary~\ref{cor:well_specified_OSS} these estimators cannot be more efficient than the efficient supervised estimator. In other words, \emph{independently trained machine learning models, however sophisticated and accurate, cannot improve inference when the parameter is well-specified}.  

\begin{remark}
Our insight that independently trained machine learning models cannot improve inference in a well-specified model stands in apparent contradiction to the simulation results of \cite{angelopoulos2023ppi++}, who find that PPI does lead to improvement over supervised estimation in generalized linear models. This is because they have simulated data such that $f(X)=Y+\epsilon$, i.e. the machine learning model is \emph{not} independent of $\labeled$ and $\unlabeled$. A modification to their simulation study to achieve independence (in keeping with the setting of their paper) corroborates our insight, i.e., PPI does not outperform the supervised estimator.  
\end{remark}

Next, we take advantage of the insights developed in previous sections to propose a class of OSS estimators that  incorporates  the machine learning models $\set{f_k}_{k=1}^K$ and improves upon the existing PPI estimators. We begin with an initial supervised estimator $\thetaest$ that is regular and asymptotically linear with influence function $\IF(z)=\IF(x,y)$, and we suppose that $\cetaest$ is an estimator of $\eta^*$. As in Section~\ref{subsec:OSS_est}, we estimate the conditional influence function $\cEIF(x)$ defined in \eqref{eq:cEIF}  with regression. Specifically, consider 
\begin{equation}
\label{eq:ppi_g_est}
    \cg(x) = \left[\IFplugin(x, f_1(x)),\ldots, \IFplugin(x, f_K(x))\right]^\top,
\end{equation}
which arises from replacing the true response in  $\IFplugin(x, y)$ with the machine learning model $f_k(x)$, for $k \in [K]$. 
Its empirically centered version is 
\begin{equation}
\label{eq:ppi_g_est_centered}
\cecenteredg(x) = \cg(x) - \frac{1}{n+N}\sum_{i=1}^{n+N}\cg(X_i).
\end{equation}
Then the  regression estimator of the conditional influence function is $\cIFest(x) = \oOLSppi\cecenteredg(x)$, where
\begin{equation}
\label{eq:ppi_oOLS}
\oOLSppi = \left[\frac{1}{n}\sum_{i=1}^n\IFplugin(Z_i)\cecenteredg(X_i)^\top\right]\left[\frac{1}{n+N}\sum_{i=1}^{n+N}\cecenteredg(X_i)^\top\cecenteredg(X_i)\right]^{-1}
\end{equation}
are the coefficients obtained from regressing $\set{\IFplugin(Z_i)}_{i=1}^n$ onto $\set{\cecenteredg(X_i)}_{i=1}^n$. Motivated by the safe OSS estimator, $\ossthetaest$ in \eqref{eq:est_OSS}, the safe PPI estimator is defined as
\begin{equation}
\label{eq:ppi_est}
\ossthetaestppi = \thetaest - \frac{1}{n}\sum_{i=1}^n\oOLSppi \cecenteredg(X_i).
\end{equation}

We now investigate the asymptotic behavior of the estimator \eqref{eq:ppi_est}. Note that Theorem~\ref{thm:OSS_estimator} is not applicable, as the regression basis $\cecenteredg$ 
in \eqref{eq:ppi_g_est_centered} 
is random due to the involvement of $\cetaest$. 
Consider an arbitrary class of measurable functions indexed by $\eta$,
\begin{equation}
\label{eq:g_class}
\sG = \set{g_\eta(x): \sX \to \bR^d, \eta \in \Omega}.
\end{equation}

We make the following assumptions on $\sG$.
\begin{assumption}
\label{asu:g_Lipschitz}
(a) $\tg(x) \in \sL_d^2(\marginal)$ and $\Var\left[\tg(X)\right]$ is non-singular; 
(b) Under Assumption~\ref{asu:IF_Lipschitz},  $\set{g_\eta(x):\eta \in \sO}$ is $\joint$-Donsker, and for all $\set{\eta_1,\eta_2} \subset \sO$, $\norm{g_{\eta_1}(x)-g_{\eta_2}(x)} \le G(x)\rho(\eta_1, \eta_2)$, where $G:\sX \to \bR^+$ is a square-integrable function.
\end{assumption}


Similar to Assumption~\ref{asu:IF_Lipschitz}, Assumption~\ref{asu:g_Lipschitz} requires that the class of functions $\set{g_\eta(x): \eta \in \sO}$ is a Donsker class; when $\eta(\cdot)$ is finite-dimensional, the next proposition provides sufficient conditions for Assumption~\ref{asu:g_Lipschitz}.
\begin{proposition}
\label{prop:equi_assumption3}
When $\eta(\cdot)$ is a finite-dimensional functional, the following condition implies Assumption~\ref{asu:g_Lipschitz} (b): under the conditions of Proposition~\ref{prop:equi_assumption1}, for all $\set{\eta_1,\eta_2} \subset \sO$, $\norm{g_{\eta_1}(x)-g_{\eta_2}(x)} \le G(x)\norm{\eta_1-\eta_2}$, where $G:\sX \to \bR^+$ is a square integrable function.
\end{proposition}

The proof of Proposition~\ref{prop:equi_assumption3} is similar to that of Proposition~\ref{prop:equi_assumption1} and is hence omitted.

Define $\tcenteredg(x) = \tg(x) - \bE[\tg(X)]$ as the centered version of $\tg(x)$, and 
\begin{equation}
\label{eq:ppi_OLS}
\OLSppi = \bE\left[\IF(Z)\tcenteredg(X)^\top\right]\bE\left[\tcenteredg(X)\tcenteredg(X)^\top\right]^{-1}
\end{equation}
as the population coefficients for the regression of  $\tcenteredg(X)$ onto $\IF(Z)$. 
The next proposition establishes the asymptotic behavior of $\ossthetaestppi$.
\begin{proposition}
\label{prop:ppi_OSS}
Suppose that $\hat{\theta}_n = \hat{\theta}(\labeled)$ is a supervised estimator that satisfies Assumption~\ref{asu:IF_Lipschitz}, $\sG$ defined as \eqref{eq:g_class} is a class of measurable functions that satisfies Assumption~\ref{asu:g_Lipschitz}, and $\cetaest$ is an estimator of $\eta^*$ as in Assumption~\ref{asu:IF_Lipschitz} (c). Suppose further that $\frac{N}{n+N} \to \gamma \in (0,1]$. Then the estimator $\ossthetaestppi$ defined in \eqref{eq:ppi_est} is a regular and asymptotically linear estimator of $\theta^*$ in the sense of Definitions~\ref{def:regular_OSS} and~\ref{def:AL_OSS}, and has  influence function  
\begin{equation}
\left[\IF(z_1)- \gamma\OLSppi\tcenteredg(x_1)\right]+ \sqrt{\gamma(1-\gamma)}\OLSppi\tcenteredg(x_2)
\end{equation}
in the OSS setting, where $\OLSppi$ is defined as in \eqref{eq:ppi_OLS}. Furthermore, the asymptotic variance of $\ossthetaestppi$ takes the form 
$$\*\Sigma(\gamma) = \Var\left[\IF(Z)- \gamma\OLSppi\tcenteredg(X)\right] + \Var\left[\sqrt{\gamma(1-\gamma)}\OLSppi\tcenteredg(X)\right],$$
and satisfies 
\begin{equation}
\begin{aligned}
&\Var\left[\IF(Z)-\OLSppi\tcenteredg(X)\right] \preceq \*\Sigma(\gamma)\preceq \Var\left[\IF(Z)\right],\quad \forall \gamma \in [0,1].
\end{aligned}
\end{equation}
\end{proposition} 
Proposition~\ref{prop:ppi_OSS} can be viewed as an extension of Theorem~\ref{thm:OSS_estimator}, where in Theorem~\ref{thm:OSS_estimator} the function class $\sG$ is a singleton class $\set{g(x)}$ that does not depend on $\eta^*$.  Our proposed PPI estimator \eqref{eq:ppi_est} flexibly incorporates multiple black-box machine learning models, and enjoys  several advantages over existing PPI estimators: 
\begin{enumerate}
\item We show in Supplement~\ref{sec:apdx_ppi} that \eqref{eq:ppi_est} is optimal within a class of PPI estimators including \cite{angelopoulos2023prediction, angelopoulos2023ppi++, miao2023assumption, gan2023prediction}. 
\item Unlike the proposals of  \cite{angelopoulos2023prediction, angelopoulos2023ppi++, miao2023assumption}, our estimator is safe: that is, it is at least as efficient as the initial supervised estimator,  regardless of the quality of the machine learning models and the proportion of unlabeled data. 
\item While existing PPI estimators are only applicable to M- and Z-estimators,  our proposal is much more general: it is applicable to arbitrary inferential problems and requires only a regular and asymptotically linear initial supervised estimator.
\end{enumerate}
We provide a detailed discussion of the efficiency of existing PPI estimators in Appendix~\ref{sec:apdx_ppi}.

\section{Connection with missing data}
\label{sec:missing}
The missing data framework provides an alternative approach for modeling the semi-supervised setting  \citep{robins1995semiparametric, chen2008semiparametric, zhou2008estimating, li2023efficient, graham2024towards}.  Here we relate the proposed framework to the classical theory of missing data. 

To establish a formal relationship between the two paradigms, we consider a missing completely at random (MCAR) model under which the semiparametric efficiency bound coincides with that derived in Theorem~\ref{thm:efficiency_OSS} in the OSS setting.  Let $\set{(Z_i, W_i)}_{i=1}^{n+N}$ be i.i.d. data, where $Z \sim \joint$, $W \sim \Pdelta$ is a binary missingness indicator such that the response $Y_i$ is observed if and only if $W_i = 1$, and $\Pdelta$ is a Bernoulli distribution with known probability $1-\gamma$ where $\gamma = \lim_{n\to\infty}\frac{N}{n+N} \in (0,1)$. Assume that $\joint \in \jointmodel$, where $\jointmodel$ is defined in \eqref{eq:separable_model} as in  previous sections. The underlying model of $(Z,W)$ is thus
\begin{equation}
\label{eq:missing_model}
\sQ = \set{\bP\times\Pdelta: \bP \in \jointmodel}.
\end{equation}
The next proposition derives the efficiency bound relative to \eqref{eq:missing_model}. 
\begin{proposition}
\label{prop:efficiency_missing}
Let $\jointmodel$ be defined as in \eqref{eq:separable_model} and let $\frac{N}{n+N} \to \gamma \in (0,1)$. Suppose that the efficient influence function of $\thetafunctional$ at $\joint$ relative to $\jointmodel$ is $\EIF$, and recall that $\cEIF(x)=\bE\left[\EIF(Z) \mid X=x\right]$ is the conditional efficient influence function. Consider i.i.d. data $\set{(Z_i, W_i)}_{i=1}^{n+N}$ generated from $\joint \times \Pdelta$ with model $\sQ$ \eqref{eq:missing_model}. Then the efficient influence function of $\thetafunctional$ at $\joint \times \Pdelta$ relative to $\jointmodel$ is 
$$\frac{w}{\sqrt{1-\gamma}}[\EIF(z)-\cEIF(x)] + \sqrt{1-\gamma}\cEIF(x),$$
and the corresponding semiparametric efficiency lower bound is
\begin{equation}
\notag
\begin{aligned}
&\Var\set{\frac{W}{\sqrt{1-\gamma}}[\EIF(Z)-\cEIF(X)] + \sqrt{1-\gamma}\cEIF(X)}\\
&=\Var\left[\EIF(Z)-\cEIF(X)\right] + (1-\gamma)\Var[\cEIF(X)]. \label{eq:missing-bound}
\end{aligned}
\end{equation}
\end{proposition}
A comparison of \eqref{eq:missing-bound} and \eqref{eq:EIF_var_OSS} reveals that, 
 for any fixed $\gamma \in (0,1)$, the efficiency bound under the MCAR model \eqref{eq:missing_model} is the same as  under the OSS setting. Thus, the amount of information useful for inference under the two paradigms is the same. However, in the MCAR model the data 
$\set{(Z_i, W_i)}_{i=1}^{n+N}$ are fully i.i.d.; by contrast, in the OSS setting the  data are not fully i.i.d., and instead consist of two independent i.i.d. parts from $\joint$ and $\marginal$, respectively. In fact, the OSS setting
corresponds to the MCAR model conditional on the event  $\set{\sum_{i=1}^{n+N}W_i = N}$.  
 As the distribution of $W$ is completely known under the MCAR model, conditioning does not alter the information available, and so  it is not surprising that the efficiency bounds are the same. 

The semi-supervised setting allows for $\gamma = 1$, i.e., for the possibility that the sample size of the unlabeled data far exceeds that of  the labeled data,  $N \gg n$. As  discussed in Section~\ref{subsec:OSS_efficiency}, this case corresponds to the ISS setting, with  efficiency lower bound given by Theorem~\ref{thm:efficiency_ISS}. Our proposed safe and efficient OSS estimators also allow for $\gamma = 1$, as shown in Theorems~\ref{thm:OSS_estimator} and \ref{thm:OSS_efficient_estimator}. By contrast, it is difficult to theoretically analyze the efficiency lower bound when $\gamma = 1$ using the missing data framework, and estimators developed for missing data require the probability of missingness to be strictly smaller than 1.

\section{Applications}
\label{sec:app}

In this section, we apply the proposed framework to a variety of inferential problems, including M-estimation, U-statistics, and average treatment effect estimation. 
 Constructing the safe OSS estimator $\ossthetaest$ \eqref{eq:corrected_est_proj} or  the PPI estimator $\ossthetaestppi$ \eqref{eq:ppi_est} requires finding an initial supervised estimator $\thetaest$ and an estimator $\cetaest$ that satisfies Assumption~\ref{asu:IF_Lipschitz}. To construct the efficient OSS estimator \eqref{eq:est_NP_OSS}, the initial supervised estimator also needs to satisfies Assumption~\ref{asu:IF_holder}, which cannot be verified in practice. 

\subsection{M-estimation}
First, we apply the proposed framework to M-estimation, a setting also considered in Chapter 2 of \cite{chakrabortty2016robust} and \cite{song2023general}. For a function $m_\theta(z): \Theta \times \sZ \to \bR$, we define the target parameter  as the maximizer of 
$$\theta^* = \arg\max_{\theta}\bE[m_\theta(Z)].$$
Given  i.i.d. labeled data $\labeled$, we use the M-estimator 
\begin{equation}
\label{eq:m-estimator}
\thetaest = \arg\max_{\theta}\set{\frac{1}{n}\sum_{i=1}^nm_\theta(Z_i)}
\end{equation}
as the initial supervised estimator of $\theta^*$.  
 Under regularity conditions such as those stated in Theorems 5.7 and 5.21 of \cite{van2000asymptotic}, the M-estimator \eqref{eq:m-estimator} is a regular and asymptotically linear estimator of $\theta^*$ with influence function
\begin{equation}
 \label{eq:m-estimator_IF}  
 \IF(z) = -\*V_{\theta^*}^{-1}\nabla m_{\theta^*}(z),
\end{equation}
where  $\*V_{\theta}(\bP) = \frac{\partial^2 \bE_\bP[m_\theta(Z)]}{\partial \theta \partial \theta^\top}$ and  $\msecondderivative = \*V_{\theta(\joint)}(\joint)$. The functional $\eta(\bP)$ that appears  in \eqref{eq:m-estimator_IF} is $$\eta(\bP) = \left(\*V^{-1}_{\theta(\bP)}(\bP), \theta(\bP)\right),$$
which is finite-dimensional. Therefore, validating Assumption~\ref{asu:IF_Lipschitz} is equivalent to validating the conditions stated in Proposition~\ref{prop:equi_assumption1}, which we now do for two canonical examples of M-estimation problems.


\begin{exmp}[Mean]
\label{exmp:mean}
Define $m_{\theta}(z) = \frac{1}{2}[h(y) -\theta]^2$ for some function $h(y) \in \sL^2_p(\bP^*_Y)$. The target parameter is the expectation
$\theta^* = \bE[h(Y)]$. The M-estimator \eqref{eq:m-estimator} in this case is the sample mean 
$$\thetaest = \frac{1}{n}\sum_{i=1}^nh(Y_i),$$ 
which is regular and asymptotically linear with influence function $\IF(z) = h(y)-\theta^*$. The functional $\eta(\bP)$ in this example is $\eta(\bP)  = \theta(\bP)$, as $\*V_\theta(\bP) \equiv \*I_p$ and does not depend on the underlying distribution. A consistent estimator $\cetaest$ of $\eta^* = \theta^*$ is the sample mean $\cetaest = \thetaest$. Furthermore, $\vp_\eta(z)$ is $1$-Lipschitz in $\eta$, and the conditions of Proposition~\ref{prop:equi_assumption1} are satisfied.
\end{exmp}

\begin{exmp}[Generalized linear models]
\label{exmp:glm}
Define $m_\theta(z) = y\theta^\top x - b\left(x^\top\theta\right)$, where $b:\bR \to \bR$ is a convex and infinitely-differentiable function. Let $b^{(1)}(\cdot)$ and $b^{(2)}(\cdot)$ denote its first and second-order derivatives, respectively. Here $m_\theta(z)$ corresponds to the log-likelihood of a canonical exponential family distribution with natural parameter $\theta^\top x$ and log partition function $b(\cdot)$, i.e., a generalized linear model (GLM). However, we do not assume that the underlying distribution belongs to this model. The target parameter $\theta^*$ maximizes $\bE[m_\theta(Z)]$, and can be viewed as the best Kullback–Leibler approximation of the underlying distribution by the GLM model. The M-estimator \eqref{eq:m-estimator} in this case is the GLM estimator
$$\hat{\theta}_n = \underset{\theta}{\arg\max}\set{\frac{1}{n}\sum_{i=1}^n\left[Y_i\theta^\top X_i - b\left(X_i^\top\theta\right)\right]},$$
which is regular and asymptotically linear with influence function $$\IF(z) = \bE\left[b^{(2)}(X^\top\theta^*)XX^\top\right]^{-1}\left[yx-b^{(1)}(x^\top\theta^*)x\right].$$
A consistent estimator of the functional  $\eta(\bP) = \left(\bE_\bP\left[b^{(2)}(X^\top\theta(\bP))XX^\top\right]^{-1}, \theta(\bP)\right)$ is $$\cetaest = \left(\left[\frac{1}{n}\sum_{i=1}^nb^{(2)}(X_i^\top\hat{\theta}_n)X_iX_i^\top\right]^{-1}, \hat{\theta}_n\right).$$ 
Under mild conditions, it can be shown that $\hat{\theta}_n$ satisfies the conditions of Proposition~\ref{prop:equi_assumption1}. Claims made in this example are proved in Supplement~\ref{sec:apdx_app}.
\end{exmp}

Application to Z-estimation or estimating equations is similar and thus omitted.

\subsection{U-statistics}
\label{subsec:u_statistics}
Next, we apply the proposed framework to U-statistics. Let $h(y_1, \ldots, y_R): \sY^R \to \bR$ be a symmetric kernel function. The target parameter is defined as 
$$\theta^* = \bE[h(Y_1, \ldots, Y_R)],$$
where the expectation is taken over $R$ i.i.d. random variables $Y_i \iidsim \bP^*_{Y}$.
To estimate $\theta^*$, the supervised estimator is
\begin{equation}
\label{eq:u-statistics}
\thetaest = \frac{1}{\binom{n}{r}}\sum_{\set{i_1 < \ldots < i_r} \subset [n]}h(X_{i_1},\ldots,X_{i_r}).
\end{equation}
We assume that $h$ is non-degenerate:\begin{equation}
\label{eq:u-statistics_degeneracy}  
0<\Var[h_1(Y; \bP_Y^*)]<\infty,
\end{equation}
where $h_1(y; \bP_{Y})$ is the conditional expectation $\bE_\bP[h(Y_1, \ldots, Y_R)\mid Y_1 = y]$ with the first argument fixed at $y$, i.e.,
\begin{equation}
\label{eq:u-statistics_h1}
h_1(y; \bP_{Y}) = \int h(y, y_2, \ldots, y_R)\prod_{r = 2}^Rd\bP_{Y}(y_r).
\end{equation}
Under this condition, $\thetaest$ is a regular and asymptotically linear estimator with influence function
\begin{equation}
\label{eq:u-statistics_IF}
\IF(z) = R[h_1(y;\bP^*_Y)-\theta^*],
\end{equation}
where $R$ is again the order of the kernel \citep[for a proof, see Theorem 12.3 of ][]{van2000asymptotic}.
The functional $\eta(\bP)$ in \eqref{eq:u-statistics_IF}  is $\eta(\bP) = \left(h_1(y; \bP_Y), \theta(\bP_Y)\right)$. This is an infinite-dimensional functional that takes values in $\sY_\infty \times \bR$, where $\sY_\infty $ is the space of uniformly bounded functions $h: \sY \to \bR$ equipped with the uniform metric $\norm{h_1-h_2}_\infty = \sup_{\sY}|h_1(y)-h_2(y)|$. Denoting $\eta_1 = (h_1, \theta_1)$ and $\eta_2 = (h_2, \theta_2)$, it follows that \eqref{eq:u-statistics_IF} is  $R$-Lipschitz with respect to the metric $\rho(\eta_1, \eta_2) = \norm{h_1-h_2}_\infty + \norm{\theta_1 - \theta_2}$. 

We have already established that U-statistics \eqref{eq:u-statistics} are regular and asymptotically linear with $R$-Lipschitz influence functions. Therefore, it remains to validate the remaining parts of (b) and (c) of Assumption~\ref{asu:IF_Lipschitz}, which we now do for two canonical examples of U-statistics. 

\begin{exmp}[Variance]
\label{exmp:var}
Define $h(y_1, y_2) = \frac{1}{2}(y_1-y_2)^2$. The target of inference is then the variance $\theta^* = \Var(Y)$. The U-statistic in this case is the sample variance,
$$\hat{\theta}_n = \frac{1}{n(n-1)}\sum_{i<j}(Y_i-Y_j)^2,$$
which is regular and asymptotically linear  
with influence function $\IF(z) = (y-\bE[Y])^2 - \theta^*$ when $h(y_1,y_2)$ is non-degenerate. In this simple example, the functional $\eta(\bP) = \left(\bE[Y], \theta(\bP_Y)\right)$ is finite-dimensional, and a consistent estimator of $\eta^*$ is
$\cetaest = \left(\frac{1}{n}\sum_{i=1}^nY_i, \thetaest\right)$. Therefore, invoking Proposition~\ref{prop:equi_assumption1}, Assumption~\ref{asu:IF_Lipschitz} is satisfied.
\end{exmp}

\begin{exmp}[Kendall's $\tau$]
\label{exmp:kendall}
Consider $Y = (U,V) \in \bR^2$ and define $h(y_1,y_2) = I\set{(u_1-u_2)(v_1-v_2)>0}$. The target of inference is $$\theta^* = \bP\set{(U_1-U_2)(V_1-V_2)>0},$$ 
which measures the dependence between $U$ and $V$. The U-statistic is  Kendall's $\tau$,
$$\hat{\theta}_n = \frac{2}{n(n-1)}\sum_{i<j}I\set{(U_i-U_j)(V_i-V_j)>0},$$
which is the average number of pairs $(U_i, V_i)$ with concordant sign. 
When $h(y_1,y_2)$ is non-degenerate, $\hat{\theta}_n$ is regular and asymptotically linear with influence function
$$\IF(z) = 2\bP_Y^*\set{(U-u)(V-v)>0} -2\theta^*,$$
where $\eta(\bP) = \left(\bP_Y\set{(U-u)(V-v)>0}, \theta(\bP_Y)\right)$. A natural estimator of $\eta^*$ in this case is
$$\cetaest = \left(\frac{1}{n}\sum_{i=1}^nI\set{(U_i-u)(V_i-v)>0}, \thetaest\right).$$ 
In Supplement~\ref{sec:apdx_app}, we validate the conditions in Assumption~\ref{asu:IF_Lipschitz} under additional assumptions on $\bP_Y^*\set{(U-u)(V-v)>0}$.
\end{exmp}

\subsection{Average treatment effect}
We consider application of the proposed framework to the estimation of the average treatment effect (ATE). Suppose we have $Z = (U, A, Y)$ for confounders $U \in \sU$, binary treatment $A \in \set{0,1}$, and outcome $Y \in \sY$. Let $Y^{(0)}$ and $Y^{(1)}$ represent the counterfactual outcomes under control and treatment, respectively.  Under appropriate assumptions, the ATE, defined as $\bE\left[Y^{(1)}\right]-\bE\left[Y^{(0)}\right]$,  can be expressed as $\bE\set{\mu^{(1)}(U;\joint)-\mu^{(0)}(U;\joint)}$,
where $\mu^{(j)}(u;\bP) = \bE_\bP[Y\mid U=u, A = j]$ for $j \in \set{0,1}$. For simplicity, we consider the target of inference
\begin{equation}
\label{eq:ATE_1}
\theta^* = \bE\left[Y^{(1)}\right] =\bE\set{\mu^{(1)}(U;\joint)}.
\end{equation}
Inference for the ATE, $\bE\left[Y^{(1)}\right]-\bE\left[Y^{(0)}\right]$, follows similarly.

We define $\mu(u;\bP) = \mu^{(1)}(u;\bP)$ and $\pi(u;\bP) = \bP\set{A = 1\mid U}$, and consider the model
$$\sP = \set{p_U(u) \pi(u; \sP)^a[1-\pi(u; \sP)]^{1-a} p_{Y\mid A,U}(y\mid u,a): \forall u, p_U(u)>0; \exists\epsilon \in (0,0.5), \pi(u;\bP) \in [\epsilon, 1-\epsilon]}.$$

It can be shown \citep{robins1994estimation, hahn1998role} that the efficient influence function relative to $\sP$ is
\begin{equation}
\label{eq:EIF_ATE}
\begin{aligned}
&\EIF(z) = \frac{a}{\pi(u;\joint)}\left[y-\mu(u;\joint)\right]+ \mu(u;\joint)-\theta^*,    
\end{aligned}
\end{equation}
where $\eta(\bP) = \set{\pi(u;\bP), \mu(u;\bP)}$. In this case, $\eta(\bP)$ is infinite-dimensional and takes values in the product space $\sU_\infty \times \sU_\infty$, where $\sU_\infty$ is the space of uniformly bounded functions $h: \sU \to \bR$ equipped with the uniform metric $\norm{h_1 - h_2}_\infty =  \sup_{\sY}|h_1(y)-h_2(y)|$.

We consider three examples.   

\begin{exmp}[Additional data on confounders]
Suppose that additional observations of the confounders are available, i.e. $\unlabeled = \set{U_i}_{i=n+1}^{n+N}$.   The conditional efficient influence function $\cEIF(u) = \bE[\EIF(Z)\mid U = u]$ in this case is
$$\cEIF(u) = \mu(u,\joint) - \theta^*.$$
By Theorem~\ref{thm:efficiency_OSS}, the efficient influence function under the OSS is
\begin{equation}
\label{eq:efficiency_OSS_U}
\begin{aligned}
&\frac{a_1}{\pi(u_1;\joint)}\left[y_1-\mu(u_1;\joint)\right] + (1-\gamma)\cEIF(u_1) + \sqrt{\gamma(1-\gamma)} \cEIF(u_2),
\end{aligned}
\end{equation}
where --- in keeping with Remark~\ref{remark:subscripts} ---  we  have used $(z_1, u_2)=((u_1, a_1, y_1), u_2)$ as arguments to the influence function. We note that if $U$ has no confounding effect, i.e. $\mu(u,\joint) = \theta^*$, then $\cEIF(u) = 0$ and the efficiency bound \eqref{eq:efficiency_OSS_U} is the same as in the supervised setting. 
\end{exmp}

\begin{exmp}[Additional data on confounders and treatment]
Suppose that additional observations of both the confounders and the treatment indicators are available, i.e. $\unlabeled = \set{(U_i, A_i)}_{i=n+1}^{n+N}$. The conditional efficient influence function $\cEIF(u,a) = \bE[\EIF(Z)\mid U = u, A = a]$ in this case is
$$\cEIF(u,a) = \frac{a}{\pi(u;\joint)}\set{\bE[Y\mid U = u, A = a]-\mu(u;\joint)} + \mu(u;\joint) -\theta^*.$$
By Theorem~\ref{thm:efficiency_OSS}, the efficient influence function under the OSS is
\begin{equation}
\begin{aligned}
\label{eq:efficiency_OSS_UA}
&\frac{a_1}{\pi(u_1)}\set{y_1-\bE[Y\mid U = u_1, A = a_1]} + (1-\gamma)\cEIF(u_1,a_1) \\
&\quad + \sqrt{\gamma(1-\gamma)} \cEIF(u_2,a_2),
\end{aligned}
\end{equation}
where  again the subscripts in \eqref{eq:efficiency_OSS_UA} are in keeping with Remark~\ref{remark:subscripts}.
If $U$ has no confounding effect and $A$ has no treatment effect, i.e. if $\bE[Y\mid U = u, A = a] = \mu(u,\joint) = \theta^*$, then $\cEIF(u,a) = 0$ and the efficiency bound \eqref{eq:efficiency_OSS_UA} is the same as in the supervised setting. 
\end{exmp}

\begin{exmp}[Additional data on confounders and treatment, and availability of surrogates]
When measuring the primary outcome $Y$ is time-consuming or expensive, we may use a \emph{surrogate marker} $S \in \bR$ as a  replacement for $Y$, to facilitate more timely decisions on the treatment effect \citep{wittes1989surrogate}. 
Suppose that $\labeled=\set{(U_i, A_i, S_i, Y_i)}_{i=1}^{n}$, and that additional observations of the confounders, the treatment indicators, and the surrogate markers are available: that is,  $\unlabeled = \set{(U_i, A_i, S_i)}_{i=n+1}^{n+N}$. The conditional efficient influence function $\cEIF(u,a,s) = \bE[\EIF(Z)\mid U = u, A = a, S = s]$ in this case is
\begin{equation}
\notag
\begin{aligned}
&\cEIF(u,a,s) =\frac{a}{\pi(u;\joint)}\set{\bE[Y\mid U = u, A = a, S = s]-\mu(u;\joint)} + \mu(u;\joint) -\theta^*.
\end{aligned}
\end{equation}
By Theorem~\ref{thm:efficiency_OSS}, the efficient influence function under the OSS is
\begin{equation}
\label{eq:efficiency_OSS_UAS}
\begin{aligned}
&\frac{a_1}{\pi(u_1)}\set{y_1-\bE[Y\mid U = u_1, A = a_1, S = s_1]} + (1-\gamma)\cEIF(u_1,a_1,s_1) 
\\&+ \sqrt{\gamma(1-\gamma)} \cEIF(u_2,a_2,s_2).
\end{aligned}
\end{equation}
Here $Z=(U,A,S,Y)$, and the subscripts in \eqref{eq:efficiency_OSS_UAS} are again in keeping with Remark~\ref{remark:subscripts}.
If $\bE[Y\mid U = u, A = a, S = s] = \mu(u,\joint) = \theta^*$, then $\cEIF(u,a,s) = 0$ and the efficiency bound \eqref{eq:efficiency_OSS_UA} is the same as in the supervised setting. 
\end{exmp}

\begin{remark}
\label{rmk:efficiency_ATE}
We show in Supplement~\ref{sec:apdx_app} that the semiparametric efficiency bound \eqref{eq:efficiency_OSS_U} is greater than \eqref{eq:efficiency_OSS_UA}, which is greater than \eqref{eq:efficiency_OSS_UAS}. In other word, efficiency improves when the unlabeled data $\unlabeled$ contains more information. 
\end{remark}

\section{Numerical experiments}
\label{sec:simu}

In this section, we illustrate the proposed framework numerically in the context of  mean estimation and  generalized linear models. Numerical experiments for variance estimation and Kendall's $\tau$ can be found in Sections~\ref{subsec:apdx_simu_var} and~\ref{subsec:apdx_simu_kendall} of the Appendix. 

In each example, the covariates $X =[X_1,X_2]$ are two-dimensional and generated as i.i.d. $\text{Unif}(0,1)$. We compute the proposed estimators $\ossthetaest$ \eqref{eq:est_OSS}, $\ossthetanpest$ \eqref{eq:est_NP_OSS}, and $\ossthetaestppi$ \eqref{eq:ppi_est} as follows:
\begin{itemize}
  \item For the estimator $\ossthetaest$, we use $g(x) = x$.
    \item For the estimator $\ossthetanpest$, we use a basis of  tensor product natural cubic splines, with $K_n \in \set{4,9,16}$ basis functions. 
    \item For the estimator $\ossthetaestppi$, we use $\cg(x) = \IFplugin(x,f(x))$ where $f:\sX \to \sY$ is a prediction model. 
    We consider two prediction models: (i) a random forest model trained on independent data, which represents an informative prediction model; and (ii) randomly-generated Gaussian noise, which represents a non-informative prediction model. 
    \end{itemize}
Each of these estimators is constructed by modifying an efficient supervised estimator, whose performance we also consider. Additionally, we include the PPI++ estimator proposed by \cite{angelopoulos2023ppi++}, using the same two prediction models as for $\ossthetaestppi$. 
    
In each simulation setting, we also estimate the semiparametric efficiency lower bounds in the ISS setting \eqref{eq:EIF_var_ISS} and in the OSS setting \eqref{eq:EIF_var_OSS}.

For each method, we report the coverage of the 95\% confidence interval, as well as the standard error; all results are averaged over 1,000 simulated datasets.

\subsection{Mean estimation}
\label{subsec:simu_mean}
We consider Example~\ref{exmp:mean} with $h(y)=y$. In this example, the supervised estimator is the sample mean, which  has influence function $\IF(z) = y-\theta^*$. The conditional influence function is
$$\cIF(x) = \bE[\IF(Z)\mid X=x] = \bE[Y\mid X=x] - \theta^*,$$
which depends on $x$ through the conditional expectation. We generate the response $Y$ as $Y = \bE[Y\mid X] + \epsilon$, where $\epsilon \sim N(0,1)$. We consider three settings for $\bE[Y\mid X]$:
\begin{enumerate}
    \item \textit{Setting 1 (linear model):} 
    $$\bE[Y\mid X=x] = 1.05 + 4.76x_1-6.2x_2.$$
    \item \textit{Setting 2 (non-linear model):} 
    $$\bE[Y\mid X=x] = -1.70x_1x_2 -6.94x_1x_2^2-1.35x_1^2x_2+2.28x_1^2x_2^2.$$
    \item \textit{Setting 3 (well-specified model, in the sense of Definition~\ref{def:well_specification}):} 
    $$\bE[Y\mid X=x] = 0.$$
\end{enumerate}
For each model, we set  $n = 1,000$, and vary the proportion of unlabeled observations, $\gamma =\frac{N}{n+N} \in \set{0.1,0.3,0.5,0.7,0.9}$.
The semiparametric efficiency lower bound in the ISS setting \eqref{eq:EIF_var_ISS} and in the OSS setting \eqref{eq:EIF_var_OSS} are estimated  separately on a sample of  $100,000$ observations.

Table~\ref{tab:apdx_mean_CI} of Appendix~\ref{subsec:apdx_simu_mean_glm} reports the coverage of 95\% confidence intervals for each method. All methods achieve the nominal coverage. 

Figure~\ref{fig:mean} displays the standard error of each method, averaged over 1,000 simulated datasets. 

\begin{figure}
    \centering
    \includegraphics[width=\linewidth]{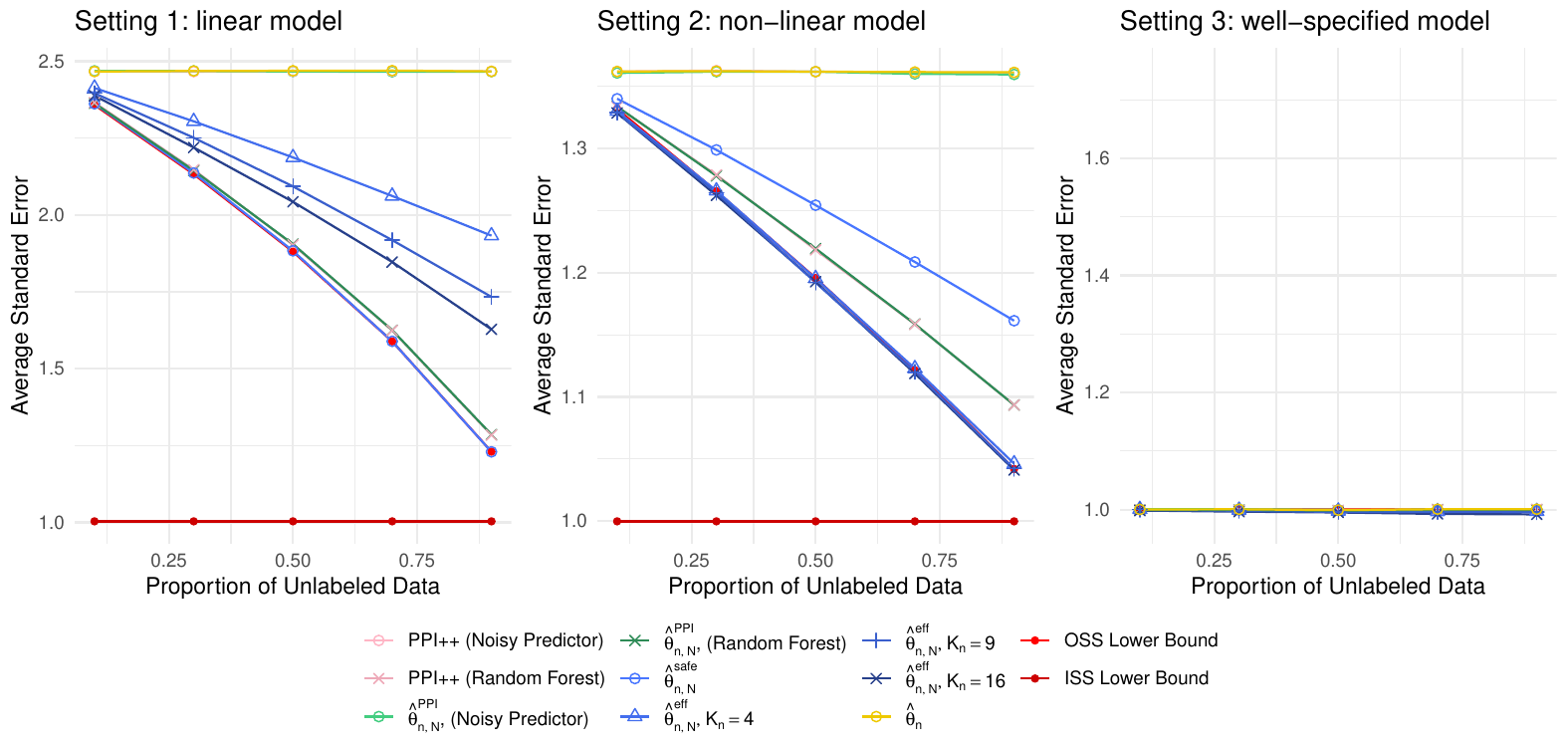}
    \caption{Standard errors of estimators of the mean, as described in Section~\ref{subsec:simu_mean}, averaged over 1,000 simulations with $n=1,000$. \emph{Left:} When the conditional influence function is linear (Setting 1), $\ossthetaest$ with $g(x) = x$ achieves the efficiency lower bound in the OSS setting. \emph{Center:} When the conditional influence function is non-linear (Setting 2), $\ossthetaest$ with $g(x) = x$ is no longer efficient, whereas $\ossthetanpest$ is efficient with a sufficient number of basis functions. \emph{Right:} For a well-specified estimation problem in the sense of Definition~\ref{def:well_specification} (Setting 3), no semi-supervised method can improve upon the supervised estimator, as shown in  Corollary~\ref{cor:well_specified_OSS}.}
    \label{fig:mean}
\end{figure}

We begin with a discussion of Setting 1 (linear model). The estimator $\ossthetaest$ achieves the efficiency lower bound in the OSS setting, since $\ossthetaest$ with $g(x) = x$ suffices to accurately approximate the true conditional influence function. 
In fact, $\ossthetanpest$, which uses a greater number of basis functions, performs worse: those additional basis functions contribute to increased variance without improving bias. Both PPI++ and $\ossthetaestppi$ using a random forest prediction model perform well, whereas the version of those methods that uses a pure-noise prediction model performs comparably to the supervised estimator, i.e., incorporating a useless prediction model does not lead to deterioration of performance. Since there is only one parameter of interest in the context of mean estimation, i.e. $p=1$, the PPI++ estimator is asymptotically equivalent to $\ossthetaestppi$ with $\cg(x) = \IFplugin(x,f(x))$ with the same prediction model $f(\cdot)$: thus, there is no difference in performance between PPI++ and $\ossthetaestppi$. However, as we will show in Section~\ref{subsec:simu_glm}, when there are multiple parameters, i.e. $p > 1$, $\ossthetaestppi$ can outperform the PPI++ estimator. 

We now consider Setting 2 (non-linear model). 
Because the conditional influence function is non-linear,  the best performance  for $\ossthetanpest$ is achieved when the number of basis functions $K_n$ is sufficiently large. Furthermore, this performance is substantially better than that of $\ossthetaest$ with $g(x) = x$. Other than this, the results are quite similar to Setting 1.

Finally, we consider Setting 3 (well-specified model), in which the model is well-specified in the sense of Definition~\ref{def:well_specification}. In keeping with Corollary~\ref{cor:well_specified_OSS},  no method improves upon the OSS lower bound.

\subsection{Generalized linear model}
\label{subsec:simu_glm}
We consider Example ~\ref{exmp:glm} with a Poisson GLM. In this example, the supervised estimator is the Poisson GLM estimator, which has influence function 
$$\IF(z) = \bE\left[e^{X^\top\theta^*}XX^\top \right]^{-1}x\left(y-e^{x^\top\theta^*}\right).$$
The conditional influence function is then
$$\cIF(x) = \bE\left[e^{X^\top\theta^*}XX^\top \right]^{-1}x\left(\bE[Y\mid X=x]-e^{x^\top\theta^*}\right).$$ We generate the response $Y$ as $Y\mid X \sim \text{Poisson}\left(\bE[Y\mid X]\right)$. However, unlike in Section~\ref{subsec:simu_mean},  the conditional influence function is not a linear function of $x$ regardless of the form of $\bE[Y\mid X=x]$. Therefore, we only consider two settings for $\bE[Y\mid X]$:
\begin{enumerate}
    \item \textit{Setting 1 (non-linear model):} $$\bE[Y\mid X=x] = \exp\set{-1.70x_1x_2 -6.94x_1x_2^2-1.35x_1^2x_2+2.28x_1^2x_2^2}.$$
    \item \textit{Setting 2 (well-specified model, in the sense of Definition~\ref{def:well_specification}):} $$\bE[Y\mid X=x] = \exp\set{1.05 + 4.76x_1-6.2x_2}.$$
\end{enumerate}
For each model, we set $n = 1,000$, and vary the proportion of unlabeled observations, $\gamma =\frac{N}{n+N} \in \set{0.1,0.3,0.5,0.7,0.9}$.
The semiparametric efficiency lower bound in the ISS setting \eqref{eq:EIF_var_ISS} and in the OSS setting \eqref{eq:EIF_var_OSS} are estimated  separately on a sample of $100,000$ observations. 

Table~\ref{tab:apdx_glm1_CI} of Appendix~\ref{subsec:apdx_simu_mean_glm} reports the coverage of 95\% confidence intervals for each method. All methods achieve the nominal coverage.

\begin{figure}
    \centering
    \includegraphics[width=0.8\linewidth]{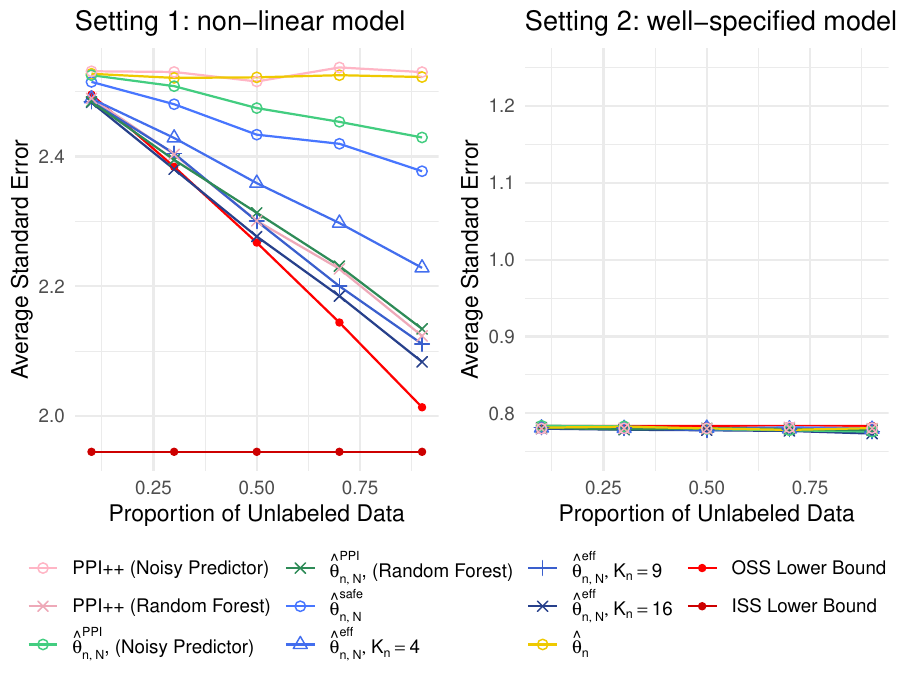}
    \caption{Standard errors of estimators of the first parameter of the Poisson GLM, as described in Section~\ref{subsec:simu_glm},  averaged over 1,000 simulations, with $n=1,000$. \emph{Left:} When the model in non-linear (Setting 1), and with a sufficient number of basis functions, $\ossthetanpest$ nearly achieves the OSS semiparametric efficiency lower bound. Moreover, $\ossthetaestppi$ with $\cg(x) = \IFplugin(x,f(x))$ outperforms the PPI++ estimators with the noisy prediction model. \emph{Right:}  For  a well-specified estimation problem in the sense of Definition~\ref{def:well_specification} (Setting 2), no semi-supervised method can improve upon the supervised estimator. This agrees with Corollary~\ref{cor:well_specified_OSS}. }
    \label{fig:glm}
\end{figure}

 Figure~\ref{fig:glm}
  displays the standard error of the first parameter for each method, averaged over 1,000 simulated datasets. Results for the standard error of the second parameter are similar, and are displayed in Figure~\ref{fig:apdx_glm2} in Appendix~\ref{subsec:apdx_simu_mean_glm}. 

We first consider Setting 1 (non-linear model). Consistent with the results for Setting 2 of Section~\ref{subsec:simu_mean}, the estimator $\ossthetanpest$ approximates the efficiency lower bound in the OSS setting when the number of basis functions is sufficiently large. Meanwhile, the estimator $\ossthetaest$ with $g(x)=x$ improves upon the supervised estimator, but is not efficient
as it cannot accurately approximate the true conditional influence function $\IFplugin(x)$. Both PPI++ and $\ossthetaestppi$ perform well with a random forest prediction model.
 In addition, we make the following observations that differ from the results shown in Section~\ref{subsec:simu_mean}: 
\begin{enumerate}
\item  $\ossthetaestppi$ significantly improves upon the supervised estimator $\thetaest$ even with a pure-noise prediction model. To see this, recall that $\ossthetaestppi$ estimates the conditional influence function $\cIF(x)$ by regressing $\IF(x,y)$ onto $\IF(x,f(x))$ where $f(\cdot)$ is a prediction model. (In practice, the unknown function $\eta^*$ is replaced with a consistent estimator $\cetaest$.) Define $\*V_\theta^* = \bE\left[XX^\top e^{X^\top\theta^*}\right]$. In the case of a Poisson GLM, we have that
 $$\IF(z) =  \*V_{\theta^*}^{-1}x\left(y-e^{x^\top\theta^*}\right)$$
    and
    $$\IF(x,f(x)) =  \*V_{\theta^*}^{-1}x\left(f(x)-e^{x^\top\theta^*}\right).$$ Thus, in the case of a Poisson GLM, the projection of $\IF(z)$ onto $\IF(x,f(x))$
     in $\sL^2_p(\joint)$ is non-zero, even if --- in the extreme case --- $f(X)=0$.
     
     By contrast, in the case of mean estimation in Section \ref{subsec:simu_mean}, the projection of  $\IF(z) =  y-\theta^*$ onto $\IF(x,f(x)) =  f(x)-\theta^*$
     in $\sL^2_p(\joint)$ equals zero when $f(X)$ is independent of $Y$. Thus, for mean estimation, a pure-noise  prediction model $f(\cdot)$ is completely useless. 
    \item $\ossthetaestppi$ is more efficient than the PPI++ estimator of \cite{angelopoulos2023ppi++}  when both use the noisy prediction model.  This is because PPI++ uses a scalar weight to minimize the trace of the $p \times p$ asymptotic covariance matrix, which is sub-optimal in efficiency when $p > 1$. On the other hand, $\ossthetaestppi$ considers a regression approach to find the best linear approximation of the conditional influence function $\cIF(x)$, as in \eqref{eq:ppi_est}.
    \end{enumerate}

Finally, we consider Setting 2 (well-specified model). No method improves upon the OSS lower bound, in keeping with Corollary~\ref{cor:well_specified_OSS}.

\section{Discussion}
\label{sec:dis}
We have  proposed a general framework to study statistical inference in the semi-supervised setting. We established the semiparametric efficiency lower bound for an arbitrary inferential problem under the semi-supervised setting, and showed that no improvement can be made when the model is well-specified. 
 Furthermore,  we proposed a class of easy-to-compute estimators that build upon existing supervised estimators and that can  achieve the efficiency lower bound for an arbitrary inferential problem. 

This paper leaves open several directions for future work. First, our results require a  Donsker condition on the influence function of the supervised estimator. This may not hold,  for example, when  the functional $\eta(\cdot)$ is high-dimensional or infinite-dimensional.  
Recent advances in double/debiased machine learning \citep{chernozhukov2018double,foster2023orthogonal} may provide an  avenue for  obtaining efficient semi-supervised estimators in the presence of high-  or infinite-dimensional nuisance parameters.   Second, a  general theoretical framework for efficient semi-supervised estimation in the presence of covariate shift  also remains a relatively open problem, despite some promising preliminary work \citep{ryan2015semi,aminian2022info}. 

\section*{Acknowledgments} We thank Abhishek Chakrabortty for pointing out that  an earlier version of this paper overlooked connections with Chapter 2 of his dissertation \citep{chakrabortty2016robust} in the special case of M-estimation. 

\bibliography{Bibliography-MM-MC}

\begin{appendix}
\section{Additional notation}
We introduce additional notation that is used in the appendix. For a matrix $\*A \in \bR^{p \times q}$, let $\norm{\*A}_2$ denote its operator norm and let $\norm{\*A}_F$ denote its Frobenius norm. For a set $\sS$, let $\text{int}(\sS)$ denote its interior. Consider the probability space $(\sZ, \sF, \joint)$. Letting $f(z)$ be a measurable function $f: \sZ \to \bR^p$ over $(\sZ, \sF, \joint)$, we adopt the following notation from the empirical process literature: $\bP(f) = \bE_\bP[f(Z)]$, $\empiricalPn(f) = \bE_{\empiricalPn}[f(Z)] = \frac{1}{n}\sum_{i=1}^n f(Z_i)$, and $\bG_n(f) = \sqrt{n}(\empiricalPn-\joint)(f)$. Similarly, for a measurable function $g:\sX \to \bR^p$, $\empiricalPnN(g) = \bE_{\empiricalPnN}[g(X)] = \frac{1}{n+N}\sum_{i=1}^{n+N} g(X_i)$, $\empiricalPN(g) = \bE_{\empiricalPN}[g(X)] = \frac{1}{N}\sum_{i=n+1}^{n+N} g(X_i)$, and $\bG_{n+N}(g) = \sqrt{n+N}(\empiricalPnN-\joint)(g)$. For two subspaces $\sT_1\in \lp$ and $\sT_2\in \lp$ such that $\sT_1 \perp \sT_2$, let $\sT_1 \oplus\sT_2$ represent their direct sum in $\lp$.

\section{Proof of main results}
\begin{proof}[Proof of Theorem \ref{thm:efficiency_ISS}.]
In the first step, we characterize the tangent space of the reduced model $\reducedmodel$, which is defined in \eqref{eq:reduced_model}. Note that the reduced model $\reducedmodel$ satisfies the conditions of Lemma~\ref{lem:apdx_tangent_separable}, i.e., the marginal distribution and the conditional distribution are separately modeled. Therefore, by Lemma~\ref{lem:apdx_tangent_separable}, the tangent space of the reduced model $\reducedmodel$ can be expressed as $\sT_{\set{\marginal}}(\marginal) \oplus \tangentconditional$, where $\conditionalmodel$ is the conditional model. Since $\set{\marginal}$ is a singleton set, $\sT_{\set{\marginal}}(\marginal) = \set{0}$, and hence the tangent space of $\reducedmodel$ becomes $\tangentconditional$. 

Recall that $\cEIF(x)$ is the conditional efficient influence function $\bE[\EIF(Z)\mid X=x]$. In the second step, we show that for all $a \in \bR^p$, 
\begin{equation}
\label{eq:apdx_conditional_in_tangent}
a^\top[\EIF(z)-\cEIF(x)] \in \tangentconditional.
\end{equation}
We prove this by contradiction. Suppose there exists $a \in \bR^p$ such that $a^\top[\EIF(z)-\cEIF(x)] \notin \tangentconditional$. 
Consider the decomposition
$$a^\top\EIF(z) = \underbrace{a^\top\cEIF(x)}_{f_a(x)} + \underbrace{a^\top[\EIF(z)-\cEIF(x)]}_{g_a(x)}.$$
As $a^\top\EIF(z) \in \lpone = \lpxone \oplus \lpyxone$, 
$f_a(x) \in \lpxone$, and $g_a(x) \in \lpyxone$, therefore by the property of direct sum, 
$f_a(x)$ and $g_a(z)$ is the unique decomposition of $a^\top\EIF(z) = f_a(x) + g_a(z)$ such that $f_a(x) \in \lpxone$, and $g_a(x) \in \lpyxone$. Further, by definition of the efficient influence function, $a^\top\EIF(z) \in \tangent$. Recall the model $\jointmodel$ defined in \eqref{eq:separable_model}. By Lemma~\ref{lem:apdx_tangent_separable}, the tangent space relative to $\jointmodel$ is 
$$\tangent = \tangentmarginal \oplus \tangentconditional.$$
Similarly, by the property of direct sum, there exists a unique decomposition $a^\top\EIF(z) = f^\prime_a(x) + g^\prime_a(z)$ such that $f^\prime_a(x) \in \tangentmarginal$ and $g^\prime_a(z) \in \tangentconditional$. However, 
$g_a(z) \notin \tangentconditional$, and hence $g_a(z) \ne g_a^\prime(z)$. By definition, $\tangentmarginal \subseteq \lpxone$,  and $\tangentconditional \subseteq \lpyxone$. Therefore, $a^\top\EIF(z) = f^\prime_a(x) + g^\prime_a(z)$
is another decomposition of $f^\prime_a(x)$ such that $f^\prime_a(x) \in \lpxone$, and $g^\prime_a(x) \in \lpyxone$. This contradicts the uniqueness of direct sum. Therefore, we have established \eqref{eq:apdx_conditional_in_tangent} for all $a \in \bR^p$.

In the third step, we show that $\EIF(z)-\cEIF(x)$ is a gradient relative to the model $\reducedmodel$. Consider any one-dimensional regular parametric sub-model $\sP_T$ of $\reducedmodel$. Since the marginal model is a singleton set $\set{\marginal}$, this maps one-to-one to a one-dimensional regular parametric sub-model $\sP_{T,Y\mid X}$ of $\conditionalmodel$ such that $\sP_T = \set{\marginal} \otimes \sP_{T,Y\mid X}$. Suppose that $\sP_{T,Y\mid X} = \set{p_{t,Y\mid X}(z), t \in T} \subset \conditionalmodel$ such that $p_{t^*,Y\mid X}(z)$ is the conditional density of $\conditional$ for some $t^* \in T$. Denote $s_{t^*}(y\mid x)$ as the score function relative to $\sP_{T,Y\mid X}$ at $t^*$, which then satisfies $s_{t^*}(y\mid x) \in \lpyxone$. Note that the score function relative to $\sP_T$ at $t^*$ remains $s_{t^*}(y\mid x)$, as the marginal distribution is known. Consider the function $\theta: \sP_T \to \Theta$ as a function of $t$, $\theta(t)$. By definition, the efficient influence function $\EIF(z)$ is a gradient relative to $\jointmodel$, hence
\begin{equation}
\notag
\begin{aligned}
\frac{\partial \theta(t)}{\partial t} &= \langle \EIF(z),  s_{t^*}(y\mid x)\rangle_\lp\\
&= \langle \EIF(z)-\cEIF(x)+\cEIF(x),  s_{t^*}(y\mid x)\rangle_\lp\\
&= \langle \EIF(z)-\cEIF(x),  s(y\mid x)\rangle_\lp + \langle \cEIF(x),  s_{t^*}(y\mid x)\rangle_\lp\\
&=\langle \EIF(z)-\cEIF(x),  s_{t^*}(y\mid x)\rangle_\lp,
\end{aligned}
\end{equation}
where the last equality used the fact that $s_{t^*}(y\mid x) \in \lpyxone$ and  $$\bE[\cEIF(X)s_{t^*}(Y\mid X)] = \bE\set{\cEIF(X)\bE[s_{t^*}(Y\mid X)\mid X]} = 0.$$ As the above holds for an arbitrary one-dimensional regular parametric sub-model of $\reducedmodel$, $\EIF(z)-\cEIF(x)$ is a gradient relative to $\reducedmodel$ at $\joint$. 

Finally, combining step 1 to step 3 above, $\EIF(z)-\cEIF(x)$ is a gradient relative to $\reducedmodel$ at $\joint$ and satisfies $a^\top[\EIF(z)-\cEIF(x)] \in \sT_{\joint}(\reducedmodel)$ for all $a \in \bR^p$. Therefore, by definition, $\EIF(z)-\cEIF(x)$ is the efficient influence function of $\thetafunctional$ at $\theta^*$ relative to $\reducedmodel$. By Lemma~\ref{lem:apdx_decomp_var:}, the variance, $\Var[\EIF(Z)]$, can be represented as
\begin{equation}
\notag
\begin{aligned}
\Var[\EIF(Z)] &= \Var[\EIF(Z)-\cEIF(X)] +  \Var[\cEIF(X)],
\end{aligned}
\end{equation}
which proves the final claim of the theorem.
\end{proof}

\begin{proof}[Proof of Theorem~\ref{thm:no_improvement_efficiency}] To prove the first part of the theorem, consider any element $s(x) \in \tangentmarginal$. Without loss of generality, suppose $\sP_{T,X} = \set{p_{t,X}(x):t \in T}$ is a regular parametric sub-model of $\marginalmodel$ with score function $s(x)$ at $\marginal$. (Otherwise, because the tangent space is a closed linear space by definition, we can always find a sequence of functions $\set{s_r(x)}_{r=1}^\infty$ that are score functions of regular parametric sub-models of $\jointmodel$, and the following arguments hold by the continuity of the inner product.)
Then, the model  $\sP_T = \sP_{T,X} \otimes \set{\conditional}$ is a regular parametric sub-model of $\jointmodel$  with score function $s(x)$ at $\joint$. Since $\sP_T$ is parameterized by $t \in T$, we can write $\theta: \sP_T\to \Theta$ as a function of $t$, $\theta(t)$. Because $\thetafunctional$ is well-specified at $\conditional$ relative to $\sP_{T,X} \subset \marginalmodel$, $\theta(t)$ is a constant function, $\theta(t) \equiv \theta^*$, and hence $\frac{\partial \theta(t)}{\partial t} = 0$. By pathwise-differentiability, $$<\gradient(z), s(x)>_{\lp} = \frac{\partial \theta(t)}{\partial t} = 0$$ 
 for any gradient $\gradient(z)$ of $\thetafunctional$ at $\joint$ relative to $\jointmodel$. Since this holds true for any $s(x) \in \tangentmarginal$, we see that any gradient $\gradient(z)$ satisfies
$$\gradient(z) \perp \tangentmarginal.$$
Consider the the efficient influence function $\EIF(z)$ of $\thetafunctional$ relative to $\jointmodel$ at $\joint$, which is a gradient of $\thetafunctional$ relative to $\jointmodel$ at $\joint$ by definition. Further, by the definition of the efficient influence function, for all $a \in \bR^p$ it holds that
$$a^\top\EIF(z) \in \tangent = \tangentmarginal\oplus\tangentconditional.$$
As $\EIF(z) \perp \tangentmarginal$, we see that $a^\top\EIF(z) \perp \tangentmarginal$, and hence
$$a^\top\EIF(z) \in \tangentconditional \subset \lpyxone,$$
for all $a \in \bR^p$.
Therefore $$\cEIF(x) = \bE[\EIF(Z)\mid X = x] = 0$$ $\marginal$-almost surely. 

To prove the second part of the theorem, recall that $\EIF(z)$ is a gradient by definition. Therefore, by Lemma~\ref{lem:apdx_gradient_set}, the set of gradients relative to $\jointmodel$ can be expressed as
\begin{equation}
\label{eq:apdx_gradient_well_specification}
\set{\EIF(z) + h(z): a^\top h(z) \in \left[\tangentconditional \oplus \tangentmarginal\right]^\perp, \forall a \in \bR^p},
\end{equation}
where $\perp$ represents the orthogonal complement of a subspace. If $\tangentmarginal = \lpxone$, then $$\left[\tangentconditional \oplus \tangentmarginal\right]^\perp \subseteq \lpxone^\perp = \lpyxone.$$
In proving the first part of the theorem, we have shown that $\EIF(z) \in \lpyx$ under well-specification. Therefore, for any gradient $\gradient(z)$ relative to $\jointmodel$, by \eqref{eq:apdx_gradient_well_specification},
we have $$a^\top\gradient(z) =  a^\top\EIF(z) + a^\top h(z) \in \lpyxone,$$
for all $a \in \bR^p$, which implies that $\gradient(z) \in \lpyx$. For a regular and asymptotically linear estimator $\thetaest$ with influence function $\IF(z)$, Lemma~\ref{lem:apdx_gradients_IF} implies that $\IF(z)$ is a gradient of $\thetafunctional$ at $\joint$ relative to $\jointmodel$. Therefore $\IF(z)\in \lpyx$, which implies that
$\bE[\IF(Z)\mid X = x] = 0$ $\marginal$-almost surely. 
\end{proof}

\begin{proof}[Proof of Corollary~\ref{cor:well_specified_ISS}]
Suppose the efficient influence function for $\thetafunctional$ at $\joint$ relative to $\jointmodel$ is $\EIF(z)$. Since $\thetafunctional$ is well-specified, Theorem~\ref{thm:no_improvement_efficiency} implies that the conditional efficient influence function $\cEIF(X) = 0, \marginal-\text{almost surely}$. Therefore, by Theorem~\ref{thm:efficiency_ISS}, the efficient influence function for $\thetafunctional$ at $\joint$ relative to $\jointmodel$ as \eqref{eq:separable_model} remains $\EIF(z)$. We have proved that the efficient influence relative to the reduced model $\reducedmodel$ is $\EIF(z)-\cEIF(x) = \EIF(z)$ in the proof of Theorem~\ref{thm:efficiency_ISS}.
\end{proof}

\begin{proof}[Proof of Proposition~\ref{prop:equi_assumption1}]
We first show that (i) and (ii) of Proposition~\ref{prop:equi_assumption1} implies Assumption~\ref{asu:IF_Lipschitz} (b). Since $\sO$ is a bounded Euclidean subset and $\vp_\eta(z)$ is $L(z)$-Lipshitz in $\eta$ over $\sO$, by example 19.6 of \cite{van2000asymptotic}, the class $\set{\vp_\eta(z):\eta \in \sO}$ is $\joint$-Donsker. To see that  Assumption~\ref{asu:IF_Lipschitz} (c) holds, note that by consistency of $\cetaest$, $\joint\set{\cetaest \in \sO} \to 1$ for any open set $\sO$ such that $\eta^* \in \sO$.\\
\end{proof}

\begin{proof}[Proof of Theorem~\ref{thm:ISS_estimator}]
Denote $\*\Sigma = \Cov[g(X)]$. First, we will show that
\begin{equation}
\label{eq:apdx_convergence_iOLS}
\norm{\iOLS - \OLS}_2 = \littleO_p(1),
\end{equation}
where $\iOLS$ is defined as in \eqref{eq:OLS_ISS}, and $\OLS$ is defined as in \eqref{eq:OLS}. To this end, we write
\begin{equation}
\notag
\begin{aligned}
\norm{\iOLS - \OLS}_2 
&\le \norm{\empiricalPn\left[\IFplugin (\centeredg)^\top\right]-\joint\left[\IF(\centeredg)^\top\right]}_2\norm{\*\Sigma^{-1}}_2\\
&\le \norm{\empiricalPn\left[\IFplugin (\centeredg)^\top\right]-\joint\left[\IFplugin (\centeredg)^\top\right]}_2\norm{\*\Sigma^{-1}}_2 + \\
&\quad\norm{\joint\left[\IFplugin (\centeredg)^\top\right]-\joint\left[\IF(\centeredg)^\top\right]}_2\norm{\*\Sigma^{-1}}_2.
\end{aligned}
\end{equation}
By Assumption~\ref{asu:IF_Lipschitz} and Lemma~\ref{lem:apdx_1st_lem}, we have, 
\begin{equation}
\label{eq:apdx_ULLN}
\norm{\empiricalPn\left[\IFplugin (\centeredg)^\top\right]-\joint\left[\IFplugin (\centeredg)^\top\right]}_2 = \littleO_p(1).
\end{equation}
Next,  
\begin{equation}
\label{eq:apdx_convergence_expected_IF}
\begin{aligned}
\norm{\joint\left[\IFplugin (\centeredg)^\top\right]-\joint\left[\IF(\centeredg)^\top\right]}_2
&=\norm{\joint\left[(\IFplugin -\IF)(\centeredg)^\top\right]}_2\\
&\le \joint\norm{(\IFplugin -\IF)(\centeredg)^\top}_2\\
&\le \norm{\IFplugin -\IF}_{\sL^2(\joint)}\norm{\centeredg}_{\sL^2(\joint)}. 
\end{aligned}
\end{equation}
By Assumption~\ref{asu:IF_Lipschitz}, when $\cetaest \in \sO$, we have $\norm{\IFplugin-\IF}_{\sL^2(\joint)} \le \norm{L(z)}_{\sL^2(\joint)}\rho(\cetaest,\eta^*)$, and it follows that
$$\norm{\joint\left[\IFplugin (\centeredg)^\top\right]-\joint\left[\IF(\centeredg)^\top\right]}_2 \le \norm{L(z)}_{\sL^2(\joint)}\rho(\cetaest, \eta^*)\norm{\centeredg}_{\sL^2(\joint)}.$$
Since $\norm{L(z)}_{\sL^2(\joint)} < \infty$ and $\rho(\cetaest, \eta^*) = \littleO_p(1)$ by Assumption~\ref{asu:IF_Lipschitz}, and $\norm{\centeredg}_{\sL^2(\joint)} < \infty$, we have that
$$\norm{L(z)}_{\sL^2(\joint)}\rho(\cetaest, \eta^*)\norm{\centeredg}_{\sL^2(\joint)} = \littleO_p(1).$$
Further, by Assumption~\ref{asu:IF_Lipschitz}, $\joint\set{\cetaest \in \sO} \to 1$, and hence
\begin{equation}
\label{eq:apdx_convergence_expected_IF_2}
\norm{\joint\left[\IFplugin (\centeredg)^\top\right]-\joint\left[\IF(\centeredg)^\top\right]}_2 = \littleO_p(1).   
\end{equation}
Combining \eqref{eq:apdx_ULLN} and  \eqref{eq:apdx_convergence_expected_IF_2}, we have $\norm{\iOLS - \OLS}_2 = \littleO_p(1)$ as \eqref{eq:apdx_convergence_iOLS}.

Since $\sqrt{n}\empiricalPn \left(\centeredg\right) = \bigO_p(1)$, the fact that $\norm{\iOLS - \OLS}_2 = \littleO_p(1)$ implies,
\begin{equation}
\notag
\begin{aligned}
\sqrt{n}\left(\issthetaest - \theta^*\right) &= \sqrt{n}\left(\thetaest - \theta^*\right) - \iOLS\sqrt{n}\empiricalPn \left(\centeredg\right) \\
&= \sqrt{n}\empiricalPn (\IF) - \iOLS\sqrt{n}\empiricalPn \left(\centeredg\right) + \littleO_p(1) \\
&= \sqrt{n}\empiricalPn (\IF) - \OLS\sqrt{n}\empiricalPn \left(\centeredg\right) + \left(\OLS-\iOLS\right)\empiricalPn \left(\centeredg\right)+\littleO_p(1)\\
&= \sqrt{n}\empiricalPn (\IF) - \OLS\sqrt{n}\empiricalPn \left(\centeredg\right) +\littleO_p(1)\\
&= \sqrt{n}\empiricalPn\left(\IF-\OLS\centeredg\right)+\littleO_p(1).
\end{aligned}
\end{equation}
Since $\OLS\centeredg(x)$ is the projection of $\cIF$ onto the linear subspace spanned by $\centeredg(x)$, by Lemma~\ref{lem:apdx_decomp_var:} and~\ref{lem:apdx_pythagorean} we have the decomposition
\begin{equation}
\notag
\begin{aligned}
&\Var\left[\IF(Z)-\OLS\centeredg(X)\right] \\
&= \Var \set{\IF(Z) - \cIF(X) + \cIF(X) -\OLS\centeredg(X)}\\    
&= \Var[\IF(Z) - \cIF(X)] + \Var\set{\cIF(X) -\OLS\centeredg(X)}\\
&= \Var[\IF(Z)]-\Var\set{\cIF(X)} + \Var\set{\cIF(X)} -\Var[\OLS\centeredg(X)]\\
&= \Var[\IF(Z)]-\Var[\OLS\centeredg(X)].
\end{aligned}
\end{equation}
\end{proof}

\begin{proof}[Proof of Theorem~\ref{thm:ISS_efficient_estimator}]
It suffices to prove the case of $p=1$, i.e., we only one parameter. For $p>1$, the analysis follows by separately analyzing each of the $p$ components. This does not affect the convergence rate since we treat $p$ as fixed (rather than increasing with the sample size $n$).

When $p = 1$, 
by Assumption~\ref{asu:IF_holder}, $\cIF \in \sC^{\alpha_1}_{M_1}(\sX)$, where $\sC^{\alpha_1}_{M_1}(\sX)$ is the Hölder class \eqref{eq:holder_space} with parameter $\alpha_1$ and $M_1$. By Theorem 2.7.1 of \cite{wellner2013weak}, $\sC^{\alpha_1}_{M_1}(\sX)$ is $\marginal$-Donsker when $\alpha_1 > \dim(\sX)$. Let $\set{g_k(x)}_{k=1}^\infty$ be a basis of $\sC^{\alpha_1}_{M_1}(\sX)$ that satisfies \eqref{eq:approx_error} for all $f \in \sC^{\alpha_1}_{M_1}(\sX)$. 
Since $\GKn(x) \in \sC^{\alpha_1}_{M_1}(\sX)$, clearly we also have $\cIFest = \npiOLS\centeredGKn(x) \in \sC^{\alpha_1}_{M_1}(\sX)$ as this is only a linear transformation of $\GKn(x)$. If we can show that
$$\norm{\cIF - \cIFest}_\lpxone = \littleO_p(1),$$
then by Lemma~\ref{lem:19.24 of vdv} and the fact that $\joint(\cIF) = \joint(\cIFest) = 0$,  
$$\bG_n(\cIF-\cIFest) = \sqrt{n}\empiricalPn(\cIF-\cIFest) = \littleO_p(1).$$ 
Then
\begin{equation}
\notag
\begin{aligned}
\sqrt{n}\left(\issthetanpest - \theta^*\right) &= \sqrt{n}\empiricalPn(\IF-\cIFest) + \littleO_p(1)\\
&= \sqrt{n}\empiricalPn(\IF-\cIF) + \littleO_p(1),
\end{aligned}
\end{equation}
and the results of Theorem~\ref{thm:ISS_efficient_estimator} follow.

It remains to show that 
$$\norm{\cIF - \cIFest}_\lpxone = \littleO_p(1).$$
Denoting $\*\Sigma = \Var\left[\centeredGKn(X)\right] \succ 0$ and $\npOLS = \joint\left[\IF(\centeredGKn)^\top\right]\*\Sigma^{-1}$, we have:
\begin{equation}
\notag
\begin{aligned}
\norm{\cIF - \cIFest}_\lpxone &= \norm{\cIF-\npiOLS\centeredGKn}_\lpxone \\
&= \norm{\cIF-\npOLS\centeredGKn + \npOLS\centeredGKn - \npiOLS\centeredGKn}_\lpxone \\
&\le \underbrace{\norm{\npiOLS\centeredGKn - \npOLS\centeredGKn}_\lpxone}_{\text{I}} + \underbrace{\norm{\cIF - \npOLS\centeredGKn}_\lpxone}_{\text{II}}
\end{aligned}
\end{equation}
where $\npiOLS$ is defined as \eqref{eq:B_np_OLS_ISS}.
We first look at I. By the fact that
\begin{equation}
\label{eq:apdx_var_normalized}
\bE\left[\*\Sigma^{-\frac{1}{2}}\centeredGKn(X)\centeredGKn(X)^\top\*\Sigma^{-\frac{1}{2}}\right] = \*I_{K_n},
\end{equation}
we have:
\begin{equation}
\notag
\begin{aligned}
\text{I}^2 &=\norm{\left(\npiOLS - \npOLS\right)\centeredGKn(x)}^2_\lpxone \\
&= \bE\left[(\centeredGKn(X))^\top\left(\npiOLS - \npOLS\right)^\top\left(\npiOLS - \npOLS\right)\centeredGKn(X)\right]\\
&= \bE\set{\left[\*\Sigma^{-\frac{1}{2}}\centeredGKn(X)\right]^\top\empiricalPn\left[(\IF - \IFplugin)\*\Sigma^{-\frac{1}{2}}\centeredGKn \right]^\top \empiricalPn\left[(\IF - \IFplugin)\*\Sigma^{-\frac{1}{2}}\centeredGKn \right]\*\Sigma^{-\frac{1}{2}}\GKn(X)}\\
&= \bE\left[\*\Sigma^{-\frac{1}{2}}\centeredGKn(X)\centeredGKn(X)^\top\*\Sigma^{-\frac{1}{2}}\right]\empiricalPn\left[(\IF - \IFplugin)\*\Sigma^{-\frac{1}{2}}\centeredGKn \right]^\top \empiricalPn\left[(\IF - \IFplugin)\*\Sigma^{-\frac{1}{2}}\centeredGKn \right]\\
&=\norm{\empiricalPn\left[\*\Sigma^{-\frac{1}{2}}\centeredGKn(\IF - \IFplugin)\right]}^2.
\end{aligned}
\end{equation}
By \eqref{eq:apdx_var_normalized}, without loss of generality, we can suppose that $\centeredGKn$ has identity covariance, i.e., $\*\Sigma = \*I_{K_n}$. For I, when $\cetaest \in \sO$, it follows that
\begin{equation}
\notag
\begin{aligned}
\text{I}^2 &= \norm{\empiricalPn\left[\centeredGKn(\IF - \IFplugin)\right]}^2\\
&\le \norm{\empiricalPn\left[L\centeredGKn\right]}^2\rho^2(\cetaest, \eta^*)\\
&\le \empiricalPn(L^2)\empiricalPn\left[(\centeredGKn)^\top\centeredGKn\right] \rho^2(\cetaest, \eta^*)\\
\end{aligned}
\end{equation}
Since $\joint\set{\cetaest \in \sO} \to 1$ by Assumption~\ref{asu:IF_Lipschitz}, we have
$$\text{I}^2 \le \empiricalPn(L^2)\empiricalPn\left[(\centeredGKn)^\top\centeredGKn\right] \rho^2(\cetaest, \eta^*) + \littleO_p(1).$$
Now, we consider the term $\empiricalPn\left[(\centeredGKn)^\top\centeredGKn\right]$. By Theorem 12.16.1 of \cite{hansen2022econometrics}, 
$$\norm{\empiricalPn\left[\centeredGKn(\centeredGKn)^\top\right]-\*I_{K_n}}_F = \littleO_p(1).$$
Therefore,
\begin{equation}
\notag
\begin{aligned}
\empiricalPn\left[(\centeredGKn)^\top\centeredGKn\right] &= K_n + \tr\left(\empiricalPn\left[\centeredGKn(\centeredGKn)^\top\right]-\*I_{K_n}\right)\\
&= K_n + \littleO_p(1)\\
&= \bigO_p(K_n),
\end{aligned}
\end{equation}
where we used continuous mapping for the $\tr(\cdot)$ function. Since $L(z)$ is square integrable, we then have $$\empiricalPn(L^2)\empiricalPn\left[(\centeredGKn)^\top\centeredGKn\right] \rho^2(\cetaest, \eta^*) = \bigO_p(\rho^2(\cetaest, \eta^*)K_n),$$
which implies that
$$\text{I} = \bigO_p\left(\sqrt{K_n}\rho(\cetaest, \eta^*)\right) = \littleO_p(1).$$
For II, by the bias-variance decomposition, 
$$\norm{\cIF - \npOLS\GKn(x)}^2_\lpxone = \bE[\npOLS\GKn(x)]^2 + \norm{\cIF - \npOLS\centeredGKn(x)}^2_\lpxone,$$
where we used the fact that $\bE[\cIF(X)] = \bE[\centeredGKn(X)] = 0$. Therefore,
\begin{equation}
\label{eq:apdx_approx_error}
\begin{aligned}
\text{II}^2 
&\le \norm{\cIF - \npOLS\GKn(x)}^2_\lpxone\\
&= \bigO\left(K_n^{-\frac{2\alpha_1}{\dim(\sX)}}\right),
\end{aligned}
\end{equation}
by \eqref{eq:approx_error}. This shows
$$\text{II} = \bigO\left(K_n^{-\frac{\alpha_1}{\dim(\sX)}}\right) = \littleO_p(1).$$
Combining the results of I and II, we have shown that $\norm{\cIF - \cIFest}_\lpxone = \littleO_p(1)$, which finishes the proof.
\end{proof}

\begin{proof}[Proof of Theorem~\ref{thm:efficiency_OSS}.] 
The proof of this theorem utilizes notation and results from  Section~\ref{sec:apdx_efficiency_noniid}, as well as Lemma~\ref{lem:apdx_tangent_space_OSS}.

By Lemma~\ref{lem:apdx_path}, pathwise differentiablity at $\joint$ implies pathwise differentiability at $\ossjoint$. Since the efficient influence function at $\joint$ relative to $\jointmodel$ is $\EIF(z)$, $\EIF(z)$ is a gradient at $\joint$ relative to $\jointmodel$, and by Lemma~\ref{lem:apdx_path}, the function
$$\EIF(z_1)-\cEIF(x_1)+\cEIF(x_2)$$
is a gradient at $\ossjoint$ relative to $\jointmodel$. 

First we consider the case of $p = 1$. We first derive the $\norm{\cdot}_\sH$-projection of $\EIF(z_1)-\cEIF(x_1)+\cEIF(x_2)$ onto $\osstangent$, where the Hilbert space $\sH$ and its norm $\norm{\cdot}_\sH$ are defined in Section~\ref{sec:apdx_efficiency_noniid}. The form of $\osstangent$ is characterized by Lemma~\ref{lem:apdx_tangent_space_OSS}. Recall that $\cEIF(x)$ is the conditional expectation of $\EIF(z)$ on $X$. The projection problem can be expressed as
\begin{equation}
\label{eq:apdx_projection}
\begin{aligned}
&\min_{\substack{f \in \tangentmarginal, \\g \in \tangentconditional}}\norm{\EIF(z_1)-\cEIF(x_1)-f(x_1)-g(z_1)+\cEIF(x_2)-\frac{\gamma}{1-\gamma}f(x_2)}^2_\sH\\
&= \min_{\substack{f \in \tangentmarginal, \\g \in \tangentconditional}}\left(\norm{\EIF(z)-\cEIF(x)-f(x)-g(z)}^2_\lpone+\frac{\gamma}{1-\gamma}\norm{\cEIF(x)-\frac{\gamma}{1-\gamma}f(x)}^2_\lpxone\right)\\
&=\min_{g \in \tangentconditional}\norm{\EIF(z)-\cEIF(x) - g(z)}^2_{\lpyxone}\\
&\quad + \min_{f \in \tangentmarginal}\set{\norm{f(x)}^2_{\lpxone} + \frac{\gamma}{1-\gamma} \norm{f(x)-\frac{1-\gamma}{\gamma}\cEIF(x)}^2_\lpxone},
\end{aligned}
\end{equation}
where we used the definition of $\norm{\cdot}_\sH$ and the fact that 
$\lpone = \lpxone \oplus \lpyxone$. Since $\EIF(z)$ is the efficient influence function at $\joint$ and $p=1$, we have $\EIF(z)-\cEIF(x) \in \tangentconditional$, which we proved in the proof of Theorem~\ref{thm:efficiency_ISS}. Therefore, for the first term in \eqref{eq:apdx_projection}, the minimizer takes $g(z) = \EIF(z)-\cEIF(x)$. Similarly, we have $\cEIF(x) \in \tangentmarginal$. Then, it can be straightforwardly verified that $f(x) = (1-\gamma)\cEIF(x)$ minimizes the second term in \eqref{eq:apdx_projection}. Therefore, the $\norm{\cdot}_\sH$-projection of $\EIF(z_i)$ onto $\osstangent$ is
\begin{equation}
\notag
\begin{aligned}
\EIF(z_1,x_2) &=\EIF(z_1)-\cEIF(x_1) +(1-\gamma)\cEIF(x_1)+\gamma\cEIF(x_2)\\
&= \EIF(z_1)-\gamma\cEIF(x_1)+\gamma\cEIF(x_2).
\end{aligned}
\end{equation}

We now consider the general case of $p$ and validate that $\EIF(z_1,x_2)$ is indeed the efficient influence function by definition. Clearly, since $\osstangent$ is a linear space, $a^\top\EIF(z_1,x_2) \in \osstangent$ for all $a \in \bR^p$. Denote $D_\ossjoint(z_1,x_2) = \EIF(z_1)-\cEIF(x_1)+\cEIF(x_2)$. For any element $f(x_1) + g(z_1) + \frac{\gamma}{1-\gamma}f(x_2) \in \osstangent$, it can be straightforwardly validated by the definition of $\left\langle \cdot, \cdot \right\rangle_\sH$ that
\begin{equation}
\notag
\begin{aligned}
\langle \EIF(z_1,x_2)- D_\ossjoint(z_1,x_2), f(x_1) + g(z_1) + \frac{\gamma}{1-\gamma}f(x_2) \rangle_\sH = 0.
\end{aligned}
\end{equation}
Consider any one-dimensional regular parametric sub-model $\sP_T = \set{\bP_t: t \in T}$ of $\jointmodel$ such that $\bP_{t^*} = \joint$ and its score function at $\ossjoint$ is $l_{t^*}(z_1,x_2)$. Then, $l_{t^*}(z_1,x_2) \in \osstangent$. Write $\theta: \sP_T \to \Theta$ as a function of $t \in T$. By pathwise differentiability and the fact that $D_\ossjoint(z_1,x_2)$,is a gradient at $\ossjoint$ relative to $\jointmodel$,
\begin{equation}
\notag
\begin{aligned}
\frac{d \theta(t^*)}{dt} &= \langle l_{t^*}(z_1,x_2), D_\ossjoint(z_1,x_2)\rangle_\sH \\
&= \langle l_{t^*}(z_1,x_2), D_\ossjoint(z_1,x_2)-\EIF(z_1,x_2)+\EIF(z_1,x_2)\rangle_\sH\\
&= \langle l_{t^*}(z_1,x_2), \EIF(z_1,x_2)\rangle_\sH + \langle l_{t^*}(z_1,x_2), D_\ossjoint(z_1,x_2)-\EIF(z_1,x_2)\rangle_\sH\\
&= \langle l_{t^*}(z_1,x_2), \EIF(z_1,x_2)\rangle_\sH,
\end{aligned}
\end{equation}
which shows that $\EIF(z_1,x_2)$ is a gradient. Since further $a^\top\EIF(z_1,x_2) \in \osstangent$ for all $a \in \bR^p$, by definition $\EIF(z_1,x_2)$ is the efficient influence function at $\ossjoint$ relative to $\jointmodel$.

By Lemma~\ref{lem:oss_efficiency}, the semiparametric efficiency lower bound is the outer product of the efficient influence function with itself in $\langle\cdot, \cdot\rangle_\sH$. By the definition of $\langle\cdot, \cdot\rangle_\sH$, this can be written as
\begin{equation}
\notag
\begin{aligned}
&\langle\EIF(z)-\gamma\cEIF(x), \EIF(z)^\top-\gamma\cEIF(x)^\top \rangle_\lp \\
&\quad+  \frac{\gamma}{1-\gamma}\langle (1-\gamma)\cEIF(x),  (1-\gamma)\cEIF(x)^\top \rangle_\lpx\\
&= \Var\left[\EIF(Z)-\cEIF(X)\right] + (1-\gamma)^2\Var\left[\cEIF(X)\right] + \gamma(1-\gamma)\Var\left[\cEIF(X)\right]\\
&=  \Var\left[\EIF(Z)-\cEIF(X)\right] + (1-\gamma)\Var\left[\cEIF(X)\right]\\
&=  \Var\left[\EIF(Z)\right]-\gamma\Var\left[\cEIF(X)\right],
\end{aligned}
\end{equation}
where we used Lemma~\ref{lem:apdx_decomp_var:} in the last equality.
\end{proof}

\begin{proof}[Proof of Theorem~\ref{thm:OSS_estimator}.]
Recall that
\begin{equation}
\notag
\begin{aligned}
\ecenteredg(X_i) &= g(X_i) - \empiricalPnN (g)   \\
&= g(X_i) - \joint(g) - \empiricalPnN \left[g-\joint(g)\right]\\
&= \centeredg(X_i) - \empiricalPnN (\centeredg).
\end{aligned}
\end{equation}
Then, for the estimator $\ossthetaest$ defined in \eqref{eq:est_OSS},
\begin{equation}
\label{eq:apdx_OSS_est}
\begin{aligned}
\sqrt{n}(\ossthetaest-\theta^*) &= \sqrt{n}(\thetaest-\theta^*) - \oOLS\sqrt{n}\empiricalPn (\ecenteredg)\\
&= \sqrt{n}\empiricalPn(\IF) - \oOLS\sqrt{n}\empiricalPn(\centeredg) + \oOLS\sqrt{n}\empiricalPnN(\centeredg)+ \littleO_p(1)\\
&= \sqrt{n}\empiricalPn(\IF) - \oOLS\sqrt{n}\empiricalPn(\centeredg) + \oOLS\frac{\sqrt{n}n}{n+N}\empiricalPn(\centeredg) + \oOLS\frac{\sqrt{n}N}{n+N}\empiricalPN(\centeredg)+ \littleO_p(1)\\
&= \sqrt{n}\empiricalPn(\IF) - \gamma\oOLS\sqrt{n}\empiricalPn(\centeredg) + \sqrt{\gamma(1-\gamma)}\oOLS\sqrt{N}\empiricalPN(\centeredg) + \littleO_p(1),
\end{aligned}
\end{equation}
as $\frac{N}{n+N} \to \gamma \in (0,1]$. By CLT, we have $\sqrt{n}\empiricalPn\left(\centeredg\right) = \bigO_p(1)$, and $\sqrt{N}\empiricalPN\left(\centeredg\right) = \bigO_p(1)$. Further, by the fact that $\labeled \indep \unlabeled$, for any function $f$, it holds that $\empiricalPn(f) \indep \empiricalPN(f)$.

Next, we show that
$$\norm{\iOLS - \oOLS}_2 = \littleO_p(1),$$
where $\iOLS$ is defined in \eqref{eq:OLS_ISS} and $\oOLS$ is defined in \eqref{eq:OLS_OSS}. Recall that we define $\*\Sigma = \Cov[g(X)]$. Then, 
\begin{equation}
\notag
\begin{aligned}
\norm{\iOLS - \oOLS}_2 
&= \norm{\empiricalPn\left[\IFplugin (\centeredg)^\top\right]\*\Sigma^{-1}-\empiricalPn\left[\IFplugin (\ecenteredg)^\top\right]\empiricalPnN\left[\ecenteredg(\ecenteredg)^\top\right]^{-1}}_2\\
&\le \underbrace{\norm{\empiricalPn\left[\IFplugin (\centeredg)^\top\right]-\empiricalPn\left[\IFplugin (\ecenteredg)^\top\right]}_2\norm{\*\Sigma^{-1}}_2}_{\text{I}}+ \\
&\quad\underbrace{\norm{\empiricalPn\left[\IFplugin (\ecenteredg)^\top\right]}_2\norm{\*\Sigma^{-1} - \set{\empiricalPnN\left[\ecenteredg(\ecenteredg)^\top\right]}^{-1}}_2}_{\text{II}}.
\end{aligned}
\end{equation}

First we show that I=$\littleO_p(1)$.By Assumption~\ref{asu:IF_Lipschitz} and Lemma~\ref{lem:apdx_1st_lem}, $\norm{\empiricalPn(\IFplugin )} = \bigO_p(1)$, and since $\joint\left(\centeredg\right) = 0$, $\empiricalPnN\left(\centeredg\right) = \littleO_p(1)$, we have:
\begin{equation}
\notag
\begin{aligned}
&\norm{\empiricalPn\left[\IFplugin (\centeredg)^\top\right]-\empiricalPn\left[\IFplugin (\ecenteredg)^\top\right]}_2\\
&= \norm{\empiricalPn(\IFplugin )\left[\empiricalPnN(\centeredg)\right]^\top}_2\\
&\le \norm{\empiricalPn(\IFplugin )}_2\norm{\empiricalPnN(\centeredg)}_2\\
&= \littleO_p(1).
\end{aligned}
\end{equation}
Since $\*\Sigma \succ 0$, it follows that
$$\text{I} = \norm{\empiricalPn\left[\IFplugin (\centeredg)^\top\right]-\empiricalPn\left[\IFplugin (\ecenteredg)^\top\right]}_2\norm{\*\Sigma^{-1}}_2 = \littleO_p(1).$$

Next we show that II=$\littleO_p(1)$. By \eqref{eq:apdx_ULLN},
$$\norm{\empiricalPn\left[\IFplugin (\centeredg)^\top\right] - \joint\left[\IFplugin (\centeredg)^\top\right]}_2 = \littleO_p(1),$$ 
and by \eqref{eq:apdx_convergence_expected_IF},
$$\norm{\joint\left[\IFplugin (\centeredg)^\top\right] - \joint\left[\IF (\centeredg)^\top\right]}_2 = \littleO_p(1).$$
Therefore,
\begin{equation}
\notag
\begin{aligned}
&\norm{\empiricalPn\left[\IFplugin (\ecenteredg)^\top\right]}_2\\
&\le \norm{\empiricalPn\left[\IFplugin (\centeredg)^\top\right]}_2 + \norm{\empiricalPn(\IFplugin )\left[\empiricalPnN(\centeredg)\right]^\top}_2\\
&= \norm{\empiricalPn\left[\IFplugin (\centeredg)^\top\right]}_2 + \littleO_p(1)\\
&\le \norm{\empiricalPn\left[\IFplugin (\centeredg)^\top\right] - \joint\left[\IFplugin (\centeredg)^\top\right]}_2 + \norm{\joint\left[\IFplugin (\centeredg)^\top\right] - \joint\left[\IF (\centeredg)^\top\right]}_2\\
&\quad+ \norm{\joint\left[\IF (\centeredg)^\top\right]}_2 + \littleO_p(1)\\
&\le \underbrace{\norm{\empiricalPn\left[\IFplugin (\centeredg)^\top\right] - \joint\left[\IFplugin (\centeredg)^\top\right]}_2}_{\littleO_p(1)} + \underbrace{\norm{\joint\left[\IFplugin (\centeredg)^\top\right] - \joint\left[\IF (\centeredg)^\top\right]}_2}_{\littleO_p(1)}\\
&\quad+ \underbrace{\norm{\joint\left[\IF (\centeredg)^\top\right]}_2}_{\bigO_p(1)} + \littleO_p(1)\\
&= \bigO_p(1).
\end{aligned}
\end{equation}
Further, by the law of large numbers, $\norm{\*\Sigma - \empiricalPnN\left[\centeredg\left(\centeredg\right)^\top\right]}_2 = \littleO_p(1)$ and $\empiricalPnN(\centeredg) = \littleO_p(1)$. Then,
\begin{equation}
\notag
\begin{aligned}
&\norm{\*\Sigma - \empiricalPnN\left[\ecenteredg(\ecenteredg)^\top\right]}_2\\
&\le \norm{\*\Sigma - \empiricalPnN\left[\centeredg(\centeredg)^\top\right]}_2 + \norm{\empiricalPnN(\centeredg)\empiricalPnN\left[(\centeredg)^\top\right]}_2\\
&= \littleO_p(1).
\end{aligned}
\end{equation}
Hence, $\norm{\*\Sigma^{-1} - \set{\empiricalPnN\left[\ecenteredg(\ecenteredg)^\top\right]}^{-1}}_2 = \littleO_p(1)$ by continuous mapping. As a result, the second term also satisfies
$$\text{II} = \norm{\empiricalPn\left[\IFplugin (\ecenteredg)^\top\right]}_2\norm{\*\Sigma^{-1} - \set{\empiricalPnN\left[\ecenteredg(\ecenteredg)^\top\right]}^{-1}}_2 = \littleO_p(1).$$
Combining I and II, we have shown that $\norm{\iOLS - \oOLS}_2 = \littleO_p(1)$, and consequently $\norm{\OLS - \oOLS}_2 = \littleO_p(1)$ given \eqref{eq:apdx_convergence_iOLS}.

Now, by \eqref{eq:apdx_OSS_est},
$$\sqrt{n}\left(\ossthetaest-\theta^*\right) = \sqrt{n}\empiricalPn(\IF - \gamma\OLS\centeredg) + \sqrt{N}\empiricalPN(\sqrt{\gamma(1-\gamma)}\OLS\centeredg) + \littleO_p(1).$$
Therefore, $\ossthetaest$ is a regular and asymptotically linear estimator with influence function
$$\IF(z_1) - \gamma\OLS\centeredg(x_1)+ \sqrt{\gamma(1-\gamma)}\OLS\centeredg(x_2).$$
By Lemmas~\ref{lem:apdx_decomp_var:} and~\ref{lem:apdx_pythagorean}, the asymptotic variance of $\ossthetaest$ can be represented as
\begin{equation}
\notag
\begin{aligned}
&\Var\left[\IF(Z) - \gamma\OLS\centeredg(X)\right] + \Var\left[\sqrt{\gamma(1-\gamma)}\OLS\centeredg(X)\right]\\
&= \Var\left[\IF(Z) - \gamma\OLS\centeredg(X)\right] + \gamma(1-\gamma)\Var\left[\OLS\centeredg(X)\right]\\
&= \Var\left[\IF(Z)-\cIF(X)\right] + \Var\left[\cIF(X)-\OLS\centeredg(X) + (1-\gamma)\OLS\centeredg(X)\right]\\
&\quad+\gamma(1-\gamma)\Var\left[\OLS\centeredg(X)\right]\\
&=\Var\left[\IF(Z)-\cIF(X)\right] + \Var\left[\cIF(X)-\OLS\centeredg(X)\right] + (1-\gamma)^2\Var\left[\OLS\centeredg(X)\right]\\
&\quad+ \gamma(1-\gamma)\Var\left[\OLS\centeredg(X)\right]\\
&=\Var\left[\IF(Z)-\cIF(X)\right] + \Var\left[\cIF(X)-\OLS\centeredg(X)\right] + (1-\gamma)\Var\left[\OLS\centeredg(X)\right]\\
&=\Var\left[\IF(Z)\right]-\Var\left[\cIF(X)\right] + \Var\left[\cIF(X)-\OLS\centeredg(X)\right] + (1-\gamma)\Var\left[\OLS\centeredg(X)\right]\\
&=\Var\left[\IF(Z)\right]- \Var\left[\OLS\centeredg(X)\right] + (1-\gamma)\Var\left[\OLS\centeredg(X)\right]\\
&= \Var\left[\IF(Z)\right]-\gamma\Var\left[\OLS\centeredg(X)\right],
\end{aligned}
\end{equation}
which proves the claim of the theorem.
\end{proof}

\begin{proof}[Proof of Theorem~\ref{thm:OSS_efficient_estimator}.]
As $\ecenteredGKn(x) = \centeredGKn - \empiricalPnN\left(\centeredGKn\right)$, we have
\begin{equation}
\notag
\begin{aligned}
\sqrt{n}\left(\ossthetanpest-\theta^*\right)
&= \sqrt{n}\empiricalPn(\IF) - \sqrt{n}\empiricalPn\left(\npoOLS\centeredGKn\right) + \sqrt{n}\empiricalPnN\left(\npoOLS\centeredGKn\right)\\
&= \sqrt{n}\empiricalPn(\IF)- \sqrt{n}\empiricalPn\left(\frac{N}{n+N}\npoOLS\centeredGKn\right) + \sqrt{N}\empiricalPN\left(\frac{\sqrt{nN}}{n+N}\npoOLS\centeredGKn\right).
\end{aligned}
\end{equation}
Similar to the proof of Theorem~\ref{thm:ISS_efficient_estimator}, we only prove the case of $p=1$. The general case proceeds by analyzing each coordinate individually. By Assumption~\ref{asu:IF_holder}, $\cIF \in \sC^{\alpha_1}_{M_1}(\sX)$. As $\set{g_k(x)}_{k=1}^\infty$ is a basis of $\sC^{\alpha_1}_{M_1}(\sX)$, we have that $\npoOLS\centeredGKn(x) \in \sC^{\alpha_1}_{M_1}$. Therefore, if we can show that
\begin{equation}
\label{eq:apdx_convergence_np_est}
\norm{\cIF-\npoOLS\centeredGKn(x)}_{\lpxone} = \littleO_p(1),
\end{equation}
then by Lemma~\ref{lem:19.24 of vdv}, $\sqrt{n}\empiricalPn\left(\cIF - \npoOLS\centeredGKn(x)\right) = \littleO_p(1)$. By definition, $\bE[\centeredGKn] = 0$. Therefore, using the fact that $\frac
{N}{n+N} \to \gamma \in (0,1]$, 
\begin{equation}
\notag
\begin{aligned}
&\sqrt{n}\left(\ossthetanpest-\theta^*\right)\\
&= \sqrt{n}\empiricalPn(\IF)- \sqrt{n}\empiricalPn\left(\frac{N}{n+N}\npoOLS\centeredGKn\right) + \sqrt{N}\empiricalPN\left(\frac{\sqrt{nN}}{n+N}\npoOLS\centeredGKn\right)\\
&= \sqrt{n}\empiricalPn(\IF-\gamma\cIF) + \sqrt{N}\empiricalPN\left(\sqrt{\gamma(1-\gamma)}\cIF\right) + \littleO_p(1).
\end{aligned}
\end{equation}
It then follows that $\ossthetanpest$ is a regular and asymptotically linear estimator with influence function
$$\IF(z_1) - \gamma\cIF(x_1)+ \sqrt{\gamma(1-\gamma)}\cIF(x_2).$$
By Lemma~\ref{lem:apdx_decomp_var:}, the asymptotic variance of $\ossthetanpest$ can be represented as
\begin{equation}
\notag
\begin{aligned}
&\Var\left[\IF(Z) - \gamma\cIF(X)\right] + \Var\left[\sqrt{\gamma(1-\gamma)}\cIF(X)\right]\\
&= \Var\left[\IF(Z) - \cIF(X) + (1-\gamma)\cIF(x)\right] + \gamma(1-\gamma)\Var\left[\cIF(X)\right]\\
&= \Var\left[\IF(Z) - \cIF(X)\right] + (1-\gamma)^2\Var\left[\cIF(x)\right] + \gamma(1-\gamma)\Var\left[\cIF(X)\right]\\
&= \Var\left[\IF(Z) - \cIF(X)\right] + (1-\gamma)\Var\left[\cIF(X)\right]\\
&= \Var\left[\IF(Z)\right] - \Var\left[\cIF(X)\right] + (1-\gamma)\Var\left[\cIF(X)\right]\\
&= \Var\left[\IF(Z)\right] - \gamma\Var\left[\cIF(X)\right],
\end{aligned}
\end{equation}
which would prove Theorem~\ref{thm:OSS_efficient_estimator}. 

To complete the argument, we next prove \eqref{eq:apdx_convergence_np_est}.
Denoting $\*\Sigma = \bE\left[\centeredGKn(\centeredGKn)^\top\right]$, and $\npOLS = \joint\left[\IF(\centeredGKn)^\top\right]\*\Sigma^{-1} \in \bR^{K_n \times 1}$, we have:
\begin{equation}
\notag
\begin{aligned}
\norm{\cIF - \npoOLS\centeredGKn}_\lpxone
&= \norm{\cIF - \npOLS\centeredGKn + \npOLS\centeredGKn-\npoOLS\centeredGKn}_\lpxone\\
&\le \underbrace{\norm{\npoOLS\centeredGKn(x) - \npOLS\centeredGKn(x)}_\lpxone}_{\text{I}} \\
&\quad+ \underbrace{\norm{\cIF - \npOLS\centeredGKn(x)}_\lpxone}_{\text{II}}.
\end{aligned}
\end{equation}
From \eqref{eq:apdx_approx_error} in the proof of Theorem~\ref{thm:ISS_efficient_estimator}, we already have
$$\text{II} = \bigO\left(K_n^{-\frac{\alpha_1}{2\dim(\sX)}}\right) = \littleO_p(1).$$
Therefore, it remains to prove that $\text{I} = \littleO_p(1)$. Notice that
\begin{equation}
\notag
\begin{aligned}
\text{I} &= \norm{\npoOLS\centeredGKn - \npOLS\centeredGKn}^2_{\lpx} \\
&= \bE\set{\centeredGKn^\top\left(\npoOLS - \npOLS\right)^\top\left(\npoOLS - \npOLS\right)\centeredGKn}\\
&= \tr\set{\left(\npoOLS - \npOLS\right)^\top\left(\npoOLS - \npOLS\right)\*\Sigma}.
\end{aligned}
\end{equation}
Moreover,
\begin{equation}
\notag
\begin{aligned}
&\npOLS = \joint\left[\IF\left(\*\Sigma^{-\frac{1}{2}}\centeredGKn\right)^\top\right]\*\Sigma^{-\frac{1}{2}},\\
&\npoOLS = \empiricalPn\left[\IFplugin\left(\*\Sigma^{-\frac{1}{2}}\ecenteredGKn\right)^\top\right]\empiricalPn\left[\left(\*\Sigma^{-\frac{1}{2}}\ecenteredGKn\right)\left(\*\Sigma^{-\frac{1}{2}}\ecenteredGKn\right)^\top\right]^{-1}\*\Sigma^{-\frac{1}{2}},\\
&\*\Sigma^{-\frac{1}{2}}\ecenteredGKn = \*\Sigma^{-\frac{1}{2}}\centeredGKn - \empiricalPnN \left(\*\Sigma^{-\frac{1}{2}}\centeredGKn\right).
\end{aligned}
\end{equation}
Therefore, without loss of generality, we can consider normalizing $\centeredGKn$ with $\*\Sigma^{-\frac{1}{2}}\centeredGKn$ and let $\*\Sigma = \*I_{K_n}$. Then, we can write
$$\text{I} = \tr\set{\left(\npoOLS - \npOLS\right)^\top\left(\npoOLS - \npOLS\right)} = \norm{\npoOLS - \npOLS}^2_F.$$

Denoting $r_{K_n}(x) = \cIF(x) - \npOLS\centeredGKn(x)$ as the approximation error, and $\epsilon(z) = \IF(z)-\cIF(x)$, $\IFplugin(z)$ can be decomposed as:
\begin{equation}
\label{eq:apdx_IF_decomp}
\begin{aligned}
\IFplugin(z) &= \cIF(x) + \epsilon(z) + [\IFplugin(z)-\IF(z)],\\
&= \npOLS\centeredGKn(x) + r_{K_n}(x) + \epsilon(z) + [\IFplugin(z)-\IF(z)],\\
&= \npOLS\ecenteredGKn(x) + \npOLS\empiricalPnN(\centeredGKn) + r_{K_n}(x) + \epsilon(z)+ [\IFplugin(z)-\IF(z)].
\end{aligned}
\end{equation}
Therefore, denoting $\hat{\*\Sigma} = \empiricalPn\left[\ecenteredGKn(\ecenteredGKn)^\top\right]$, it follows that
\begin{equation}
\notag
\begin{aligned}
\npoOLS &= \npOLS + \underbrace{\npOLS\empiricalPnN(\centeredGKn)\empiricalPn(\ecenteredGKn)^\top\hat{\*\Sigma}^{-1}}_{\text{III}} + \underbrace{\empiricalPn\left[r_{K_n}(\ecenteredGKn)^\top\right]\hat{\*\Sigma}^{-1}}_{\text{IV}}\\
&\quad + \underbrace{\empiricalPn\left[\epsilon(\ecenteredGKn)^\top\right]\hat{\*\Sigma}^{-1}}_{\text{V}} + \underbrace{\empiricalPn\left[(\IFplugin(z)-\IF(z))(\ecenteredGKn)^\top\right]\hat{\*\Sigma}^{-1}}_{\text{VI}}.
\end{aligned}
\end{equation}
We first consider $\hat{\*\Sigma}$:
\begin{equation}
\notag
\begin{aligned}
&\empiricalPn\left[\ecenteredGKn\left(\ecenteredGKn\right)^\top\right] \\
&= \empiricalPn\left[\centeredGKn\left(\centeredGKn\right)^\top\right] + \empiricalPn\left(\centeredGKn\right)\empiricalPnN\left[\left(\centeredGKn\right)^\top\right] + \empiricalPnN\left(\centeredGKn\right)\empiricalPn\left[\left(\centeredGKn\right)^\top\right] \\
&\quad+\empiricalPnN\left(\centeredGKn\right)\empiricalPnN\left[\left(\centeredGKn\right)^\top\right],
\end{aligned}
\end{equation}
where, by Lemma~\ref{lem:apdx_order_Kn}, we have that
\begin{equation}
\notag
\begin{aligned}
&\norm{\empiricalPn\left(\centeredGKn\right)\empiricalPnN\left[\left(\centeredGKn\right)^\top\right]}_F = \bigO_p\left(\frac{K_n}{n}\right) = \littleO_p(1)\\
&\norm{\empiricalPnN\left(\centeredGKn\right)\empiricalPn\left[\left(\centeredGKn\right)^\top\right]}_F = \bigO_p\left(\frac{K_n}{n}\right)= \littleO_p(1),\\
&\norm{\empiricalPnN\left(\centeredGKn\right)\empiricalPnN\left[\left(\centeredGKn\right)^\top\right]}_F = \bigO_p\left(\frac{K_n}{n}\right)= \littleO_p(1).
\end{aligned}
\end{equation}
Further, by Theorem 12.16.1 of \cite{hansen2022econometrics},$$\norm{\empiricalPn\left[\centeredGKn\left(\centeredGKn\right)^\top\right]-\*I_{K_n}}_F = \littleO_p(1),\text{ and } \lambda_{\min}\left(\empiricalPn\left[\centeredGKn\left(\centeredGKn\right)^\top\right]\right) = \bigO_p(1).$$
Therefore it follows that
$$\norm{\empiricalPn\left[\ecenteredGKn\left(\ecenteredGKn\right)^\top\right]-\*I_{K_n}}_F = \littleO_p(1),\text{ and } \lambda_{\min}\left(\empiricalPn\left[\ecenteredGKn\left(\ecenteredGKn\right)^\top\right]\right) = \bigO_p(1),$$
which implies $\norm{\hat{\*\Sigma}^{-1}}_2 = \bigO_p(1)$. Now, for the term III,
\begin{equation}
\notag
\begin{aligned}
\norm{\text{III}} &\le \vert\npOLS\empiricalPnN\left(\centeredGKn\right)\rvert\norm{\empiricalPn\left(\ecenteredGKn\right)^\top}\norm{\hat{\*\Sigma}^{-1}}_2\\
&=\bigO_p\left(\sqrt{\frac{K_n}{n}}\right)\times\bigO_p\left(\sqrt{\frac{K_n}{n}}\right) \times \bigO_p(1)\\
&= \bigO_p\left(\frac{K_n}{n}\right)\\
&= \littleO_p(1),
\end{aligned}
\end{equation}
where we used Lemma~\ref{lem:apdx_order_Kn} to show that $\vert\npOLS\empiricalPnN\left(\centeredGKn\right)\rvert$ and $\norm{\empiricalPn\left(\ecenteredGKn\right)^\top}$ are $\bigO_p\left(\sqrt{\frac{K_n}{n}}\right)$. For the term IV, first note that by the property of projection and \eqref{eq:approx_error}, we have
$$\joint(r_{K_n}) = 0, \quad \joint(r_{K_n}\centeredGKn) = 0,\text{ and } \joint\left(r_{K_n}^2\right) = \bigO\left(K_n^{-\frac{2\alpha}{\dim(\sX)}}\right).$$
Therefore,
\begin{equation}
\notag
\begin{aligned}
\norm{\text{IV}} &\le \norm{\empiricalPn\left[r_{K_n}\left(\ecenteredGKn\right)^\top\right]}\norm{\hat{\*\Sigma}}_2\\
&\le \norm{\empiricalPn\left[r_{K_n}\left(\centeredGKn\right)^\top\right]}\norm{\hat{\*\Sigma}}_2 +  \lvert\empiricalPn(r_{K_n})\rvert\norm{\empiricalPnN\left(\ecenteredGKn\right)^\top}\norm{\hat{\*\Sigma}}_2\\
&\le \norm{\empiricalPn\left[r_{K_n}\left(\centeredGKn\right)^\top\right]}\times \bigO_p(1) + \bigO_p\left(\frac{1}{\sqrt{n}}\right)\times\bigO_p\left(\sqrt{\frac{K_n}{n}}\right)\times\bigO_p(1)\\
&= \bigO_p\left(\norm{\empiricalPn\left[r_{K_n}\left(\centeredGKn\right)^\top\right]}\right) + \littleO_p(1),
\end{aligned}
\end{equation}
where we again used Lemma~\ref{lem:apdx_order_Kn}. 
Further,
\begin{equation}
\notag
\begin{aligned}
\bE\set{\empiricalPn\left[r_{K_n}\left(\centeredGKn\right)^\top\right]\empiricalPn\left[r_{K_n}\left(\centeredGKn\right)\right]}
&=\frac{1}{n}\bE\left[r_{K_n}^2(\centeredGKn)^\top\centeredGKn\right]\\
&\le \frac{\sup_{\sX}\set{\centeredGKn(x)^\top\centeredGKn(x)}}{n} \bE[r^2_{K_n}]\\
&= \bigO\left(\frac{\zeta_n^2K_n^{-\frac{2\alpha_1}{\dim(\sX)}}}{n}\right) \\
&= \littleO(1).
\end{aligned}
\end{equation}
Therefore, by Markov's inequality $\norm{\empiricalPn\left[r_{K_n}(\centeredGKn)^\top\right]} = \littleO_p(1)$ and thus so is IV. Next, consider the term V. Recall that $\epsilon(z) = \IF-\cIF$, and hence
$$\joint(\epsilon)=0,\text{ and }\joint(\epsilon g) = 0,\quad \joint\left(\epsilon^2\right) < \infty,$$
for any measurable function $g(x)$ of $x$. Thus,
\begin{equation}
\notag
\begin{aligned}
\norm{\text{V}} &\le \norm{\empiricalPn\left[\epsilon\left(\ecenteredGKn\right)^\top\right]}\norm{\hat{\*\Sigma}}_2\\
&\le \norm{\empiricalPn\left[\epsilon\left(\centeredGKn\right)^\top\right]}\norm{\hat{\*\Sigma}}_2 +  \lvert\empiricalPn(\epsilon)\rvert\norm{\empiricalPnN\left(\ecenteredGKn\right)^\top}\norm{\hat{\*\Sigma}}_2\\
&\le \norm{\empiricalPn\left[\epsilon\left(\centeredGKn\right)^\top\right]}\times \bigO_p(1) + \bigO_p\left(\frac{1}{\sqrt{n}}\right)\times\bigO_p\left(\sqrt{\frac{K_n}{n}}\right)\times\bigO_p(1)\\
&= \bigO_p\left(\norm{\empiricalPn\left[\epsilon(\centeredGKn)^\top\right]}\right) + \littleO_p(1).
\end{aligned}
\end{equation}
Moreover,
\begin{equation}
\notag
\begin{aligned}
\bE\set{\empiricalPn\left[\epsilon\left(\centeredGKn\right)^\top\right]\empiricalPn\left[\epsilon\left(\centeredGKn\right)\right]}
&=\frac{1}{n}\bE\left[\epsilon^2\left(\centeredGKn\right)^\top\centeredGKn\right]\\
&\le \frac{\sup_{\sX}\set{\centeredGKn(x)^\top\centeredGKn(x)}}{n} \bE[\epsilon^2]\\
&= \bigO\left(\frac{\zeta_n^2}{n}\right)\\
&= \littleO(1).
\end{aligned}
\end{equation}
Therefore, by Markov's inequality $\norm{\empiricalPn\left[\epsilon(\centeredGKn)^\top\right]} = \littleO_p(1)$ and so is V. Finally, for VI, by the fact that $\norm{\empiricalPn\left[\centeredGKn(\centeredGKn)^\top\right]-\*I_{K_n}}_F = \littleO_p(1)$ and continuous mapping, we have:
\begin{equation}
\notag
\begin{aligned}
\empiricalPn\left[(\centeredGKn)^\top\centeredGKn\right] &= K_n + \tr\left(\empiricalPn\left[\centeredGKn(\centeredGKn)^\top\right]-\*I_{K_n}\right)\\
&= K_n + \littleO_p(1)\\
&= \bigO_p(K_n).
\end{aligned}
\end{equation}
Then,
\begin{equation}
\notag
\begin{aligned}
\norm{\text{IV}} &\le \norm{\empiricalPn\left[(\IFplugin-\IF)(\centeredGKn)^\top\right]}\norm{\hat{\*\Sigma}^{-1}}_2 + \norm{\empiricalPn(\IFplugin-\IF)}\norm{\empiricalPnN(\centeredGKn)^\top}\norm{\hat{\*\Sigma}^{-1}}_2 \\
&\le \norm{\empiricalPn\left[(\IFplugin-\IF)(\centeredGKn)^\top\right]} \times \bigO_p(1) + \bigO_p(\rho(\cetaest, \eta^*))\times\bigO_p\left(\sqrt{\frac{K}{n}}\right)\times \bigO_p(1)\\
&= \norm{\empiricalPn\left[(\IFplugin-\IF)(\centeredGKn)^\top\right]} \times \bigO_p(1) + \littleO_p(1)\\
&\le \norm{\empiricalPn\left[L(\centeredGKn)^\top\right]} \times \bigO_p(\rho(\cetaest, \eta^*)) + \littleO_p(1)\\
&\le \sqrt{\empiricalPn(L^2)\empiricalPn\left[(\centeredGKn)^\top\centeredGKn\right]} \times \bigO_p(\rho(\cetaest, \eta^*))+ \littleO_p(1)\\
&= \bigO_p\left(\rho(\cetaest, \eta^*)\sqrt{K_n}\right) + \littleO_p(1)\\
&= \littleO_p(1).
\end{aligned}
\end{equation}
Putting these results together, we see that $\text{I} = \norm{\npOLS - \npoOLS}_F = \littleO_p(1)$, which then finishes the proof.
\end{proof}

\begin{proof}[Proof of Proposition~\ref{prop:ppi_OSS}]
By the definition of $\cecenteredg(x)$, we have:
\begin{equation}
\notag
\begin{aligned}
\cecenteredg(X_i) &= \cg(X_i) - \empiricalPnN (\cg)   \\
&= \cg(X_i)-\joint(\cg) - \empiricalPnN \left[\cg-\joint(\cg)\right]   \\
&= \ccenteredg(X_i) - \empiricalPnN (\ccenteredg).
\end{aligned}
\end{equation}
For the estimator $\ossthetaest$ defined in \eqref{eq:ppi_est} and $\oOLS$ defined in \eqref{eq:ppi_oOLS},
\begin{equation}
\label{eq:apdx_ppi_est}
\begin{aligned}
\sqrt{n}\left(\ossthetaest-\theta^*\right) &= \sqrt{n}\left(\thetaest-\theta^*\right) - \oOLS\sqrt{n}\empiricalPn\left(\cecenteredg\right)\\
&= \sqrt{n}\empiricalPn(\IF) - \oOLS\sqrt{n}\empiricalPn\left(\ccenteredg\right) + \oOLS\sqrt{n}\empiricalPnN\left(\ccenteredg\right)+ \littleO_p(1)\\
&= \sqrt{n}\empiricalPn(\IF) - \gamma\oOLS\sqrt{n}\empiricalPn\left(\ccenteredg\right) + \sqrt{\gamma(1-\gamma)}\oOLS\sqrt{N}\empiricalPN\left(\ccenteredg\right) + \littleO_p(1),
\end{aligned}
\end{equation}
as $\frac{N}{n+N} \to \gamma \in (0,1]$. 

By Assumption~\ref{asu:g_Lipschitz}, the restricted class of functions $\set{g_{\eta}: \eta \in \sO}$ is $\joint$-Donsker, and it  satisfies $\norm{\cg(x)-\tg(x)}_{\sL^2(\marginal)} = \littleO_p(1)$
by Lemma~\ref{lem:apdx_1st_lem}.
Therefore,  by Lemma~\ref{lem:apdx_3rd_lem}, the centered class  $\set{g^0_\eta: \eta \in \sO}$ satisfies the conditions of Lemma~\ref{lem:apdx_1st_lem}, and it then follows that $\norm{\ccenteredg(x)-\tcenteredg(x)}_{\sL^2(\marginal)} = \littleO_p(1)$, and 
\begin{equation}
\label{eq:apdx_ep_term}
\bG_n(\ccenteredg-\tcenteredg) = \littleO_p(1). 
\end{equation}
As a result,
\begin{equation}
\notag
\begin{aligned}
\sqrt{n}\empiricalPn\left(\ccenteredg\right) &= \bG_n\left(\ccenteredg-\tcenteredg\right) + \sqrt{n}\joint\left(\ccenteredg\right) + \sqrt{n}\empiricalPn\left(\tcenteredg\right)\\
&= \bG_n\left(\ccenteredg-\tcenteredg\right) + \sqrt{n}\empiricalPn\left(\tcenteredg\right)\\
&= \littleO_p(1) + \bigO_p(1)\\
&= \bigO_p(1),
\end{aligned}
\end{equation}
where we used the fact that $\joint\left(\ccenteredg\right) = 0$ by centering and $\sqrt{n}\empiricalPn\left(\tcenteredg\right) = \bigO_p(1)$ by centering and CLT. Following a similar argument,  $\sqrt{N}\empiricalPN\left(\ccenteredg\right) = \bigO_p(1)$. Further, by \eqref{eq:apdx_ep_term},
\begin{equation}
\notag
\begin{aligned}
&\sqrt{n}\empiricalPn(\ccenteredg) - \sqrt{n}\empiricalPn(\tcenteredg)\\
&=\bG_n(\ccenteredg-\tcenteredg) - \sqrt{n}\joint(\ccenteredg) + \sqrt{n}\joint(\tcenteredg)\\
&= \bG_n(\ccenteredg-\tcenteredg) = \littleO_p(1),
\end{aligned}
\end{equation}
and similarly 
$$\sqrt{n}\empiricalPN(\ccenteredg) - \sqrt{n}\empiricalPN(\tcenteredg) = \littleO_p(1).$$
Therefore, \eqref{eq:apdx_ppi_est} can be further modified as:
\begin{equation}
\label{eq:apdx_ppi_est_1}
\begin{aligned}
\sqrt{n}(\ossthetaest-\theta^*) &= \sqrt{n}\empiricalPn(\IF) - \gamma\oOLS\sqrt{n}\empiricalPn(\tcenteredg) + \sqrt{\gamma(1-\gamma)}\oOLS\sqrt{N}\empiricalPN(\tcenteredg) + \littleO_p(1).
\end{aligned}
\end{equation}

For convenience, let
$$\hat{\*B}_1 = \oOLS = \empiricalPn\left[\IFplugin\left(\cecenteredg\right)^\top\right]\set{\empiricalPnN\left[\cecenteredg\left(\cecenteredg\right)^\top\right]}^{-1},$$
$$\hat{\*B}_2 = \empiricalPn\left[\IFplugin\left(\tecenteredg\right)^\top\right]\set{\empiricalPnN\left[\tecenteredg\left(\tecenteredg\right)^\top\right]}^{-1}.$$
We now show that
$\norm{\hat{\*B}_1-\hat{\*B}_2}_2 = \littleO_p(1)$. First,
\begin{equation}
\notag
\begin{aligned}
\norm{\hat{\*B}_1-\hat{\*B}_2}_2
&\le\underbrace{\norm{\empiricalPn\left[\IFplugin(\cecenteredg-\tecenteredg)^\top\right]}_2}_{\text{I}}\underbrace{\norm{\set{\empiricalPnN\left[\cecenteredg(\cecenteredg)^\top\right]}^{-1}}_2}_{\text{II}}\\
&\quad+\underbrace{\norm{\empiricalPn\left[\IFplugin(\tecenteredg)^\top\right]}_2}_{\text{III}}\underbrace{\norm{\set{\empiricalPnN\left[\cecenteredg(\cecenteredg)^\top\right]}^{-1} - \set{\empiricalPnN\left[\tecenteredg(\tecenteredg)^\top \right]}^{-1}}_2}_{\text{IV}}.
\end{aligned}
\end{equation}

For $\text{I}$, note that by Lemma~\ref{lem:apdx_2nd_lem}, we have $
\norm{\empiricalPn\left[\IFplugin(\ccenteredg-\tcenteredg)^\top\right]}_2 = \littleO_p(1)$.
Therefore, 
\begin{equation}
\notag
\begin{aligned}
&\text{I} \le \norm{\empiricalPn\left[\IFplugin\left(\ccenteredg-\tcenteredg\right)^\top\right]}_2 + \norm{\empiricalPn\left(\IFplugin \right)\left[\empiricalPnN\left(\tcenteredg\right)\right]^\top}_2 + \norm{\empiricalPn\left(\IFplugin \right)\left[\empiricalPnN\left(\ccenteredg\right)\right]^\top}_2\\
&= \norm{\empiricalPn\left(\IFplugin \right)\left[\empiricalPnN\left(\tcenteredg\right)\right]^\top}_2 + \norm{\empiricalPn(\IFplugin )\left[\empiricalPnN\left(\ccenteredg\right)\right]^\top}_2 + \littleO_p(1)\\
&\le \norm{\empiricalPn\left(\IFplugin \right)}_2\norm{\empiricalPnN\left(\tcenteredg\right)}_2 + \norm{\empiricalPn\left(\IFplugin \right)}_2\norm{\empiricalPnN\left(\ccenteredg\right)}_2 + \littleO_p(1)\\
& = \littleO_p(1), 
\end{aligned}
\end{equation}
where we used the fact that $\norm{\empiricalPn(\IFplugin )}_2 = \littleO_p(1)$ and $\norm{\empiricalPnN\left(\ccenteredg\right)}_2 = \littleO_p(1)$ by Lemma~\ref{lem:apdx_1st_lem}. 

For II, note that by Lemma~\ref{lem:apdx_1st_lem} we have $\norm{\empiricalPnN(\ccenteredg)} = \littleO_p(1)$. Thus,
\begin{equation}
\begin{aligned}
\label{eq:adpx_convergence_cov_1}
\norm{\empiricalPnN\left[\cecenteredg\left(\cecenteredg\right)^\top - \ccenteredg\left(\ccenteredg\right)^\top\right]}_2 &= \norm{\empiricalPnN\left(\ccenteredg\right)\empiricalPnN\left(\ccenteredg\right)^\top}_2\\
&\le \norm{\empiricalPnN\left(\ccenteredg\right)}^2_2\\
&= \littleO_p(1).
\end{aligned}
\end{equation}
Next, consider the term $\norm{\empiricalPnN\left[\ccenteredg\left(\ccenteredg\right)^\top - \tcenteredg\left(\tcenteredg\right)^\top \right]}_2$. By Lemma~\ref{lem:apdx_3rd_lem}, the centered class $\set{g^0_{\eta}: \eta \in \Omega}$ satisfies the conditions of Lemma~\ref{lem:apdx_2nd_lem}. Therefore applying Lemma~\ref{lem:apdx_2nd_lem}, the two terms $\norm{\empiricalPnN\left[\ccenteredg\left(\ccenteredg-\tcenteredg\right)^\top \right]}_2$ and $ \norm{\empiricalPnN\left[\tcenteredg\left(\ccenteredg-\tcenteredg\right)^\top \right]}_2$ are both $\littleO_p(1)$. Then,
\begin{equation}
\notag
\begin{aligned}
&\norm{\empiricalPnN\left[\ccenteredg\left(\ccenteredg\right)^\top - \tcenteredg\left(\tcenteredg\right)^\top \right]}_2\\
&\le \norm{\empiricalPnN\left[\ccenteredg\left(\ccenteredg-\tcenteredg\right)^\top \right]}_2 + \norm{\empiricalPnN\left[\tcenteredg\left(\ccenteredg-\tcenteredg\right)^\top \right]}_2\\
&= \littleO_p(1).
\end{aligned}
\end{equation}
Combining these two results, we have:
\begin{equation}
\label{eq:adpx_convergence_cov_2}
\begin{aligned}
&\norm{\empiricalPnN\left[\cecenteredg(\cecenteredg)^\top - \tcenteredg(\tcenteredg)^\top \right]}_2\\
&\le \norm{\empiricalPnN\left[\cecenteredg(\cecenteredg)^\top - \ccenteredg(\ccenteredg)^\top \right]}_2 + \norm{\empiricalPnN\left[\ccenteredg(\ccenteredg)^\top - \tcenteredg(\tcenteredg)^\top \right]}_2\\
&= \littleO_p(1).
\end{aligned}
\end{equation}
Finally, since $\empiricalPnN\left[\tcenteredg(\tcenteredg)^\top\right] \inprobability \joint\left[\tcenteredg(\tcenteredg)^\top \right] \succ 0$, by continuous mapping, 
$$\text{II} = \norm{\set{\empiricalPnN\left[\cecenteredg(\cecenteredg)^\top \right]}^{-1}}_2 = \norm{\set{\joint\left[\tcenteredg(\tcenteredg)^\top \right]}^{-1}}_2 + \littleO_p(1) = \bigO_p(1).$$

We showed that III$=\bigO_p(1)$ in the proof of Theorem~\ref{thm:OSS_estimator}. 

For IV, by Lemma~\ref{lem:apdx_1st_lem}, both $\norm{\empiricalPnN(\ccenteredg)}$ and $\norm{\empiricalPnN(\tcenteredg)}$ are $\littleO_p(1)$. Consequently, 
\begin{equation}
\notag
\begin{aligned}
&\norm{\empiricalPnN\left[\cecenteredg(\cecenteredg)^\top - \tecenteredg(\tecenteredg)^\top \right]}_2\\
&= \norm{\empiricalPnN\left[\ccenteredg(\ccenteredg)^\top - \tcenteredg(\tcenteredg)^\top \right] - \empiricalPnN(\ccenteredg)\empiricalPnN(\ccenteredg)^\top + \empiricalPnN(\tcenteredg)\empiricalPnN(\tcenteredg)^\top}_2\\
&\le \norm{\empiricalPnN\left[\ccenteredg(\ccenteredg)^\top - \tcenteredg(\tcenteredg)^\top \right]}_2 + \norm{\empiricalPnN(\ccenteredg)\empiricalPnN(\ccenteredg)^\top}_2 + \norm{\empiricalPnN(\tcenteredg)\empiricalPnN(\tcenteredg)^\top}_2\\
&\le \norm{\empiricalPnN\left[\ccenteredg(\ccenteredg)^\top - \tcenteredg(\tcenteredg)^\top \right]}_2 + \norm{\empiricalPnN(\ccenteredg)}^2_2 + \norm{\empiricalPnN(\tcenteredg)}^2_2\\
&=\littleO_p(1),
\end{aligned}
\end{equation}
where $\norm{\empiricalPnN\left[\ccenteredg(\ccenteredg)^\top - \tcenteredg(\tcenteredg)^\top \right]}_2 = \littleO_p(1)$ is given by \eqref{eq:adpx_convergence_cov_1}. Further, we have shown that $\text{II} = \norm{\set{\empiricalPnN\left[\cecenteredg(\cecenteredg)^\top \right]}^{-1}}_2 = \bigO_p(1)$. Similarly, it can be shown that $\norm{\set{\empiricalPnN\left[\tecenteredg(\tecenteredg)^\top \right]}^{-1}}_2 = \bigO_p(1)$.
Moreover, we have
\begin{equation}
\begin{aligned}
\norm{\empiricalPnN\left[\tecenteredg(\tecenteredg)^\top - \tcenteredg(\tcenteredg)^\top\right]}_2 &= \norm{\empiricalPnN\left(\tcenteredg\right)\empiricalPnN\left(\tcenteredg\right)^\top}_2\\
&\le \norm{\empiricalPnN\left(\tcenteredg\right)}^2_2\\
&= \littleO_p(1),
\end{aligned}
\end{equation}
which, combined with \eqref{eq:adpx_convergence_cov_2}, gives
$$\norm{\empiricalPnN\left[\cecenteredg(\cecenteredg)^\top - \tecenteredg(\tecenteredg)^\top \right]}_2 = \littleO_p(1).$$
Finally, from the fact that 
$\*A^{-1}-\*B^{-1} = \*B^{-1}(\*B-\*A)\*A^{-1}$,
it follows that:
\begin{equation}
\notag
\begin{aligned}
&\text{IV} = \norm{\set{\empiricalPnN\left[\cecenteredg(\cecenteredg)^\top\right]}^{-1} - \set{\empiricalPnN\left[\tecenteredg(\tecenteredg)^\top \right]}^{-1}}_2\\
&\le\norm{\empiricalPnN\left[\cecenteredg(\cecenteredg)^\top - \tecenteredg(\tecenteredg)^\top \right]}_2\norm{\set{\empiricalPnN\left[\cecenteredg(\cecenteredg)^\top\right]}^{-1}}_2\norm{\set{\empiricalPnN\left[\tecenteredg(\tecenteredg)^\top \right]}^{-1}}_2\\
&= \littleO_p(1). 
\end{aligned}
\end{equation}

Combining the above results, we have $\norm{\hat{\*B}_1-\hat{\*B}_2}_2 = \littleO_p(1)$. As we have shown that $\sqrt{n}\empiricalPn(\ccenteredg) = \bigO_p(1)$ and $\sqrt{N}\empiricalPN(\ccenteredg) = \bigO_p(1)$, \eqref{eq:apdx_ppi_est_1} can be expressed as
\begin{equation}
\notag
\begin{aligned}
\sqrt{n}\left(\ossthetaest-\theta^*\right) &= \sqrt{n}\empiricalPn(\IF) - \gamma\hat{\*B}_2\sqrt{n}\empiricalPn(\tcenteredg) \\
&\quad+ \sqrt{\gamma(1-\gamma)}\hat{\*B}_2\sqrt{N}\empiricalPN(\tcenteredg) + \littleO_p(1).
\end{aligned}
\end{equation}
The rest of the proof follows  Theorem~\ref{thm:OSS_estimator}.
\end{proof}

\begin{proof}[Proof of Proposition~\ref{prop:efficiency_missing}.]
First, we show that the tangent space relative to $\sQ$ at $\joint \times \Pdelta$, denoted as $\tangentmissing$, is 
$$\sM = \set{s_X(x)+w s_{Y\mid X}(z): s_X(x) \in \tangentmarginal, s_{Y\mid X}(z) \in \tangentconditional},$$
where $\tangentmarginal$ is the tangent space at $\marginal$ relative to $\marginalmodel$ and $\tangentconditional$ is the tangent space at $\conditional$ relative to $\conditionalmodel$. 

We first show that $\tangentmissing \subseteq \sM$. Consider any one-dimensional regular parametric sub-model of $\sQ$, which can be represented as
$$\set{\bP \times \Pdelta: \bP \in \sP_T},$$
where $\sP_T$ is a one-dimensional regular parametric sub-model of $\jointmodel$ such that
$$\sP_T = \set{p_t(z) = p_{t,X}(x)p_{t,Y\mid X}(z): t \in T},$$ 
where $p_{t^*}(z)$ corresponds to the density of $\joint$ for some $t^* \in T$. Denote the score function of $p_t(z)$ at $t^*$ as $s(z) = s_X(x) + s_{Y\mid X}(z)$, where $s_X(x)$ is score function of $p_{t,X}(x)$ at $t^*$, and $s_{Y\mid X}(z)$ is the score function of $p_{t,Y\mid X}(y\mid x)$ at $t^*$. The density of $(Z, W)$ is thus
$$p_{t,X}(x)\left[(1-\gamma)p_{t,Y\mid X}(z)\right]^{w}\gamma^{1-w}, $$
and its score function at $t^*$ is $s_X(x) + w s_{Y\mid X}(z)$. Because $\sP_T$ is a one-dimensional regular parametric sub-model of $\jointmodel$, which has the form \eqref{eq:separable_model}, it must be true that $\set{p_{t,X}(x): t \in T}$ is a one-dimensional regular parametric sub-model of $\marginalmodel$, and $\set{p_{t,Y\mid X}(z): t \in T}$ is a one-dimensional regular parametric sub-model of $\conditional$. Therefore, the corresponding score function satisfies that $s_X(x) \in \tangentmarginal$ and $s_{Y\mid X}(z) \in \tangentconditional$, and hence $s_X(x) + w s_{Y\mid X}(z) \in \sM$. We have $\tangentmissing \subseteq \sM$ by the fact that $\sM$ is a closed linear space. 

Now we show the other direction, i.e., $\sM \subseteq \tangentmissing$. Consider an arbitrary element $s_X(x)+w s_{Y\mid X}(z)$ of $\sM$ where $s_X(x) \in \tangentmarginal$ and $s_{Y\mid X}(z) \in \conditional$. Suppose $s_X(x)$ is the score function of some parametric sub-model $\set{p_{t,X}: t \in T \subset \bR}$ of $\marginalmodel$ at $\marginal$ and $s_{Y\mid X}(z)$ is the score function of some parametric sub-model $\set{p_{t,X\mid Y}: t \in T \subset \bR}$ of $\conditionalmodel$ at $\conditional$. Then $s_X(x)+w s_{Y\mid X}(z)$ is the score function at $\joint \times \Pdelta$ of
$$\set{\bP \times \Pdelta: \bP \in \sP_T},$$
where 
$\sP_T = \set{p_{t,X}(x)p_{t,Y\mid X}(z): t \in T \subset R}$, which proves $s_X(x)+w s_{Y\mid X}(z) \in \tangentmissing$ and hence $\sM \subseteq \tangentmissing$.

Pathwise differentiability of $\thetafunctional$ at $\joint$ relative to $\joint$ implies, for any one-dimensional parametric sub-model $\sP_T \subset \jointmodel$ such that $p_{t^*}(z)$ corresponds to the density of $\joint$,
$$\frac{d\theta(p_{t^*})}{dt} = \langle \EIF(z), s_X(x) + s_{Y\mid X}(z) \rangle_\lp,$$
where $s(z) = s_X(x) + s_{Y\mid X}(z)$ is the score function of $\sP_T$ at $\joint$, $s_X(x) \in \tangentmarginal$ and $s_{Y\mid X}(z) \in \tangentconditional$. This maps one-to-one to a parametric sub-model $\set{\bP \times \Pdelta: \bP \in \sP_T}$ of $\sQ$ with score function $s_X(x) + w s_{Y\mid X}(z) \in \tangentmissing$. We thus have:
\begin{equation}
\notag
\begin{aligned}
\frac{d\theta(p_{t^*})}{dt} &= \langle \EIF(z), s_X(x) + s_{Y\mid X}(z) \rangle_\lp\\
&= \langle \EIF(z)-\cEIF(x), s_{Y\mid X}(z) \rangle_\lpyx + \langle \cEIF(x), s_{X}(z) \rangle_\lpx\\
&= \left\langle \frac{w}{1-\gamma}[\EIF(z)-\cEIF(x)], \delta s_{Y\mid X}(z) \right\rangle_\lpyx + \langle \cEIF(x), s_{X}(z) \rangle_\lpx\\
&= \left\langle \frac{w}{1-\gamma}[\EIF(z)-\cEIF(x)] + \cEIF(x),  s_{X}(z) + \delta s_{Y\mid X}(z) \right\rangle_\lp.
\end{aligned}
\end{equation}
This implies that the function
$$\frac{w}{1-\gamma}[\EIF(z)-\cEIF(x)] + \cEIF(x)$$
is a gradient of $\thetafunctional$ at $\joint \times \Pdelta$ relative to $\sQ$. As shown in the proof of Theorem~\ref{thm:efficiency_ISS}, $a^\top[\EIF(z)-\cEIF(x)] \in \tangentconditional$ and $a^\top\cEIF(x) \in \marginal$ for $a \in \bR^p$. Therefore, we see that 
$$a^\top\set{\frac{w}{1-\gamma}[\EIF(z)-\cEIF(x)] + \cEIF(x)} \in \tangentmissing,$$
for all $a \in \bR^p$. By definition, $\frac{w}{1-\gamma}[\EIF(z)-\cEIF(x)] + \cEIF(x)$ is the efficient influence function of $\thetafunctional$ at $\joint \times \Pdelta$ relative to $\sQ$. Note the above derivations correspond to the sample size $n+N$. Therefore we need to additionally multiply a factor of $\lim_{n\to\infty}\sqrt{\frac{n}{n+N}} = \sqrt{1-\gamma}$ to obtain the efficient influence function corresponding to the sample size $n$. By independence of $W$ and $Z$, the semiparametric efficiency lower bound can be expressed as:
\begin{equation}
\notag
\begin{aligned}
&\Var\set{\frac{\delta}{\sqrt{1-\gamma}}[\EIF(Z)-\cEIF(X)] + \sqrt{1-\gamma}\cEIF(X)}\\
&= \Var\set{\frac{\delta}{\sqrt{1-\gamma}}[\EIF(Z)-\cEIF(X)]} + \Var[\sqrt{1-\gamma}\cEIF(X)]\\
&\quad+\Cov\set{\frac{\delta}{\sqrt{1-\gamma}}[\EIF(Z)-\cEIF(X)],\sqrt{1-\gamma}\cEIF(X)}\\
&\quad+ \Cov\set{\sqrt{1-\gamma}\cEIF(X),\frac{\delta}{\sqrt{1-\gamma}}[\EIF(Z)-\cEIF(X)]}\\
&= \frac{\bE[\delta^2]}{1-\gamma}\Var\set{[\EIF(Z)-\cEIF(X)]} + (1-\gamma)\Var[\cEIF(X)]\\
&\quad + \bE[\delta]\Cov\set{\EIF(Z)-\cEIF(X),\cEIF(X)}+ \bE[\delta]\Cov\set{\cEIF(X),\EIF(Z)-\cEIF(X)}\\
&=\Var\left[\EIF(Z)-\cEIF(X)\right] + (1-\gamma)\Var[\cEIF(X)],
\end{aligned}
\end{equation}
where in the last step we use Lemma~\ref{lem:apdx_decomp_var:} to show that $\Cov\set{\cEIF(X),\EIF(Z)-\cEIF(X)} = 0$.
\end{proof}

\section{Connection to prediction-powered inference}
\label{sec:apdx_ppi}
In this section, we analyze existing PPI estimators through the lens of the proposed framework. We will show that (i) many existing PPI estimators can be analyzed in a unified manner, including our proposed safe PPI estimator \eqref{eq:ppi_est}; (ii) the proposed safe PPI estimator is optimally efficient among the PPI estimators; (iii) none of the PPI estimators achieves the efficiency bound of Theorem~\ref{thm:efficiency_OSS}, without strong assumptions on the machine learning prediction model. 

We consider the setting of M-estimation as described in Section 7.1 of the main text. (Our results can be extended to Z-estimation.) Let $f:\sX \to \sY$ denote a machine learning prediction model trained on independent data. Let $\mdf(x)$ denote $\mdf(x) = \md(x,f(x))$, and let $\cmd(x) = \bE[\md(Z)\mid X = x]$. Further, denote $\thetaest$ as the supervised M-estimator \eqref{eq:m-estimator}, and suppose that suitable conditions hold such that $\thetaest$ is regular and asymptotically linear with influence function $\IF(z)$ in  \eqref{eq:m-estimator_IF}. Finally, suppose that the conditions of Proposition~\ref{prop:ppi_OSS} hold. 

First, we present a unified approach to analyze existing PPI estimators, as well as the proposed safe PPI estimator \eqref{eq:ppi_est}. Consider semi-supervised estimators of $\theta^*$ that are regular and asymptotically linear in the sense of Definitions~\ref{def:regular_OSS} and~\ref{def:apdx_AL_OSS}, with influence function 
\begin{equation}
\label{eq:apdx_ppi_IF}
\vp^f_{\*A}(z_1, x_2) = -\msecondderivativeinv\left[\md(z_1)-\*A\mdf(x_1) + \*A\mdf(x_2)\right],
\end{equation}
where $\*A \in \bR^{p \times p}$. The above class includes a number of existing PPI estimators, such as the proposal of  \cite{angelopoulos2023prediction, angelopoulos2023ppi++, miao2023assumption,gan2023prediction}, and the proposed safe PPI estimator \eqref{eq:ppi_est}, with  different estimators corresponding to different choices of $\*A$. Specifically, the original PPI estimator \citep{angelopoulos2023prediction} considers $\*A = \*I_p$. \cite{angelopoulos2023ppi++} improve the original PPI estimator by introducing a tuning weight $\omega \in \bR$ and considering $\*A = \omega\*I_p$. They show that the value of $\omega$ that minimizes the trace of the asymptotic variance of their estimator is
$$\omega = \frac{\gamma\tr\left(\msecondderivativeinv\set{\Cov[\mdf(X), \md(Z)] + \Cov[\md(Z), \mdf(X)]}\msecondderivativeinv\right)}{2\tr\set{\msecondderivativeinv\Var[\mdf(X)]\msecondderivativeinv}}.$$
\cite{miao2023assumption} instead lets $\*A = \text{diag}(\pmb{\omega})$, which is a diagonal matrix with tuning weights $\pmb{\omega} \in \bR^p$. They set the weights to minimize the element-wise asymptotic variance:
$$\omega_j = \frac{\left[\msecondderivativeinv \Cov[\md(Z), \mdf(X)]\msecondderivativeinv\right]_{jj}}{\left[\msecondderivativeinv \Var[ \mdf(X)]\msecondderivativeinv\right]_{jj}},$$
for each $j \in [p]$, where $\omega_j$ is the $j$-th element of $\pmb{\omega}$. Finally, 
 \cite{gan2023prediction} considers
$$\*A = \omega\Cov[\md(Z),\mdf(X)]\Var[\mdf(X)]^{-1},$$
for a tuning weight $\omega \in \bR$ and show that the optimal weight is $\omega = \gamma$. 

Now, consider the proposed safe PPI estimator \eqref{eq:ppi_est} with $\tg(x) = \vp_{\eta^*}(x,f(x))$ and with $\sG = \set{\vp_\eta(x,f(x)):\eta \in \Omega}$, where $\eta(\bP) = \left(\*V^{-1}_{\theta(\bP)}(\bP), \theta(\bP)\right)$. Then, \eqref{eq:ppi_OLS} can be written as
\begin{equation}
\notag
\begin{aligned}
\OLSppi &=  \bE\left[\IF(Z)\tcenteredg(X)^\top\right]\bE\left[\tcenteredg(X)\tcenteredg(X)^\top\right]^{-1}\\
&= \msecondderivativeinv \Cov\left[\md(Z),\mdf(X)\right]\Var[\mdf(X)]^{-1}\msecondderivative.
\end{aligned}
\end{equation}
By Proposition~\ref{prop:ppi_OSS}, \eqref{eq:ppi_est} is regular and asymptotically linear with influence function
\begin{equation}
\label{eq:apdx_safe_ppi_IF}
\begin{aligned}
&\IF(x_1)-\gamma\OLSppi\tcenteredg(x_1)+\gamma\OLSppi\tcenteredg(x_2)\\
&=-\msecondderivativeinv\left[\md(z_1)-\*A^*\mdf(x_1) +  \*A^*\mdf(x_2)\right]\\
&=\vp^f_{\*A^*}(z_1, x_2),
\end{aligned}
\end{equation}
where 
$$\*A^* = \gamma \Cov\left[\md(Z),\mdf(X)\right]\Var[\mdf(X)]^{-1}.$$
Therefore, \eqref{eq:ppi_OLS} has influence function in the form of \eqref{eq:apdx_ppi_IF} with $\*A = \*A^*$. In fact, this is the same influence function as \cite{gan2023prediction} with the optimal weight $\omega = \gamma$. 

Next, we show that \eqref{eq:apdx_safe_ppi_IF} has the smallest asymptotic variance among all PPI estimators with influence function of the form \eqref{eq:apdx_ppi_IF}. That is, the proposed safe estimator \eqref{eq:ppi_est} is at least as efficient as existing PPI estimators.  
Recall that under the OSS setting, the asymptotic variance of an estimator with influence function $\IF(z_1,x_2)$ is $\langle\IF(z_1,x_2), \IF(z_1,x_2)^\top\rangle_\sH$.
\begin{proposition}
\label{prop:apdx_ppi_optimal}
For all $\*A \in \bR^{p \times p}$, the following holds:
$$\left\langle\vp^f_{\*A}(z_1, x_2), \vp^f_{\*A}(z_1, x_2)^\top\right\rangle_\sH \succeq \left\langle\vp^f_{\*A^*}(z_1, x_2), \vp^f_{\*A^*}(z_1, x_2)^\top\right\rangle_\sH.$$
\end{proposition}
\begin{proof}
By the linearity of inner product, 
$$\left\langle h_1,h_1^\top   \right\rangle_\sH = \left\langle h_1-h_2,(h_1-h_2)^\top   \right\rangle_\sH+\left\langle h_2,h_2^\top \right\rangle_\sH + \left\langle h_1-h_2,h_2^\top   \right\rangle_\sH + \left\langle h_2,(h_2-h_1)^\top   \right\rangle_\sH,$$
for arbitrary $h_1$ and $h_2$. Therefore, it suffices to prove that the cross term satisfies
$$\left\langle\vp^f_{\*A}(z_1, x_2)-\vp^f_{\*A^*}(z_1, x_2), \vp^f_{\*A^*}(z_1, x_2)^\top\right\rangle_\sH = 0$$
for all $\*A \in \bR^{p \times p}$, which we now prove.

By definition, we have
\begin{equation}
\label{eq:apdx_IF_PPI_inner}
\begin{aligned}
&\left\langle\vp^f_{\*A}(z_1, x_2)-\vp^f_{\*A^*}(z_1, x_2), \vp^f_{\*A^*}(z_1, x_2)^\top\right\rangle_\sH\\
&= \msecondderivativeinv\left\langle-(\*A-\*A^*)\mdf(x_1)+(\*A-\*A^*)\mdf(x_2), \left[\md(x_1)-\*A^*\mdf(x_1)+\*A^*\mdf(x_2)\right]^\top\right\rangle_\sH\msecondderivativeinv\\
&= -\msecondderivativeinv(\*A-\*A^*)\Cov\left[\mdf(X), \md(X)-\*A^*\mdf(X)\right]\msecondderivativeinv\\
&\quad+ \frac{1-\gamma}{\gamma}\msecondderivativeinv(\*A-\*A^*)\Var[\mdf(X)](\*A^*)^\top\msecondderivativeinv.
\end{aligned}
\end{equation}
Recall that $\*A^* = \gamma \Cov\left[\md(Z),\mdf(X)\right]\Var[\mdf(X)]^{-1}$. Therefore, we have
$$\Cov\left[\mdf(X), \md(X)-\*A^*\mdf(X)\right] = \frac{1-\gamma}{\gamma}\Var[\mdf(X)](\*A^*)^\top.$$
Plugging this into \eqref{eq:apdx_IF_PPI_inner}, it then follows that
$$\left\langle\vp^f_{\*A}(z_1, x_2)-\vp^f_{\*A^*}(z_1, x_2), \vp^f_{\*A^*}(z_1, x_2)^\top\right\rangle_\sH = 0,$$
which then finishes the proof.
\end{proof}
Proposition~\ref{prop:apdx_ppi_optimal} shows that \eqref{eq:ppi_est} provides optimal efficiency among estimators with influence function of the form \eqref{eq:apdx_ppi_IF}. We note that, for fair comparison, we consider the case where there is only one machine learning model $f(\cdot)$ for all estimators. However, the proposed safe PPI estimator can incorporate multiple machine learning models in a principled way, whereas it is unclear how to do this for most existing PPI estimators. Moreover, our estimator can deal with general inferential problems beyond M-estimation, such as U-statistics.

Our next result shows that estimators with influence function of the form \eqref{eq:apdx_ppi_IF} cannot achieve the semiparametric efficiency bound in the OSS setting unless strong assumptions are imposed on the machine learning prediction model $f$. To discuss efficiency in M-estimation (or Z-estimation), we informally consider a nonparametric model in the form of \eqref{eq:separable_model} such that the supervised M-estimator \eqref{eq:m-estimator} is efficient (in the supervised setting). Proposition~\ref{prop:PPIconnection2} provides a necessary condition for estimators with influence function \eqref{eq:apdx_ppi_IF} to be efficient in the OSS setting. 
\begin{proposition}
\label{prop:PPIconnection2}
Let $\sS_f = \set{\*A\mdf(x)-\*A\bE[\mdf(X)]:\*A \in \bR^{p \times p}}$ denote the linear space generated by $\mdf(x)-\bE[\mdf(X)]$ in $\lpx$. If
\begin{equation}
\label{eq:apdx_condition_ppi_efficient}
 \cmd(x) \notin \sS_f,   
\end{equation}
then \eqref{eq:apdx_ppi_IF} cannot be the efficient influence function of $\thetafunctional$ at $\ossjoint$ relative to model \eqref{eq:separable_model}.
\end{proposition}
\begin{proof}
Specializing to the case of M-estimation, by Theorem~\ref{thm:efficiency_OSS},
the efficient influence function is
$$\IF(z_1, x_2) = -\msecondderivativeinv[\md(z_1)-\gamma\cmd(x_1)+\gamma\cmd(x_2)].$$
Therefore \eqref{eq:apdx_ppi_IF} is the efficient influence function if and only if
$$\*A\mdf(X)-\*A\bE[\mdf(X)] = \gamma\cmd(X)$$
$\marginal$-almost surely. This cannot happen if $\cmd(x) \notin \sS_f$. 
\end{proof}
To better interpret condition \eqref{eq:apdx_condition_ppi_efficient}, we consider a simple case where our target of inference is the expectation of $Y$: $\theta^* = \bE[Y]$, and the loss function is $m_\theta(z) = \frac{1}{2}(y-\theta)^2$. Therefore, it follows that $\nabla m_\theta(z) = y - \theta$, $\msecondderivative = \*I_p$, $\cmd(x) = \bE[Y\mid X=x] - \theta^*$, and $\mdf(x) = f(x) - \theta^*$. We see that, in this case, condition \eqref{eq:apdx_condition_ppi_efficient} is equivalent to
$$f(x) = \*A\bE[Y\mid X =x] + b,$$
for some $\*A \in \bR^{p \times p}$ and $b \in \bR^p$. In other words, a necessary condition for these estimators to be efficient is that the machine learning prediction model is a linear transformation of the true regression function. Of course, this is unlikely to hold for black-box machine learning predictions in practice. 

We end this section by noting that although the safe PPI estimator \eqref{eq:ppi_est} may not be efficient, the proposed efficient estimator \eqref{eq:est_NP_OSS} can always achieve the semiparametric efficiency bound under regularity conditions. Further, as we show in numerical experiments in Section~\ref{sec:simu}, the safe PPI estimator has good performance when using a good prediction model.


\section{Proof of claims in Section~\ref{sec:app}}
\label{sec:apdx_app} 
\subsection{Proof of claim in Example~\ref{exmp:glm}}
Recall that the GLM estimator
\begin{equation}
\label{eq:apdx_GLM}
\hat{\theta}_n = \underset{\theta}{\arg\max}\set{\frac{1}{n}\sum_{i=1}^n\left[Y_i\theta^\top X_i - b\left(X_i^\top\theta\right)\right]}
\end{equation}
which is regular and asymptotically linear with influence function $$\IF(z) = \bE\left[b^{(2)}(X^\top\theta^*)XX^\top\right]^{-1}\left[yx-b^{(1)}(x^\top\theta^*)x\right]$$
under regularity conditions, where $\eta^* = \left(\bE\left[b^{(2)}(X^\top\theta^*)XX^\top\right]^{-1}, \theta^*\right)$.
\begin{proposition}
Suppose that both $\sX$ and $\Theta$ are compact subsets. Then, the GLM estimator $\thetaest$ satisfies the conditions of Proposition~\ref{prop:equi_assumption1} with estimator $$\cetaest = \left(\left[\frac{1}{n}\sum_{i=1}^nb^{(2)}(X_i^\top\hat{\theta}_n)X_iX_i^\top\right]^{-1}, \hat{\theta}_n\right).$$
\end{proposition}
\begin{proof}
We first prove that $\thetaest$ satisfies the local Lipshitz condition of Proposition~\ref{prop:equi_assumption1}. Denote $\eta^{(1)}(\bP) = \bE_\bP\left[b^{(2)}(X^\top\theta(\bP))XX^\top\right]^{-1} \in \bR^{p \times p}$ and $\eta^{(2)}(\bP) = \theta(\bP)$. Consider the metric $\rho(\eta_1, \eta_2) = \norm{\eta_1^{(1)}-\eta_2^{(1)}}_2 + \norm{\eta_1^{(2)}-\eta_2^{(2)}}$. Assuming that both $\sX$ and $\Theta$ are compact, which is a mild assumption, we have: 
\begin{equation}
\notag
\begin{aligned}
&\norm{\vp_{\eta_1}(z)-\vp_{\eta_2}(z)}\\ &= \norm{\left(\eta_1^{(1)}-\eta_2^{(1)}\right)yx-\eta_1^{(1)} b^{(1)}\left(x^\top\eta_1^{(2)}\right)+\eta_2^{(1)} b^{(1)}\left(x^\top\eta_2^{(2)}\right)}\\
&=\norm{\left(\eta_1^{(1)}-\eta_2^{(1)}\right)yx-\eta_1^{(1)} \left[b^{(1)}\left(x^\top\eta_1^{(2)}\right)-b^{(1)}\left(x^\top\eta_2^{(2)}\right)\right]+\left(\eta_2^{(1)}-\eta_1^{(1)}\right) b^{(1)}\left(x^\top\eta_2^{(2)}\right)}\\
&\le \norm{\eta_1^{(1)}-\eta_2^{(1)}}_2\norm{yx}+\norm{\eta_1^{(1)}}_2 \left\lvert b^{(1)}\left(x^\top\eta_1^{(2)}\right)-b^{(1)}\left(x^\top\eta_2^{(2)}\right)\right\rvert+\norm{\eta_2^{(1)}-\eta_1^{(1)}}_2 \left\lvert b^{(1)}\left(x^\top\eta_2^{(2)}\right) \right\rvert. 
\end{aligned}
\end{equation}
Since $\bE\left[b^{(2)}\left(X^\top\theta^*\right)XX^\top\right] \succ 0$, we can find a sufficiently small neighborhood $\sO$ containing $\eta^*$ such that $\sup_{\eta \in \sO}\norm{\eta^{(1)}}_2 > 0$. Since $b^{(1)}\left(x^\top\theta\right)$ is differentiable, 
\begin{equation}
\notag
\begin{aligned}
&\left\lvert b^{(1)}\left(x^\top\eta_1^{(2)}\right)-b^{(1)}\left(x^\top\eta_2^{(2)}\right)\right\rvert \le \sup_{x \in \sX, \eta \in \sO}\left\lvert 
b^{(2)}(x^\top\eta^{(2)})\right\rvert \norm{x}\norm{\eta_1^{(2)}-\eta_2^{(2)}},
\end{aligned}
\end{equation}
where $\sup_{x \in \sX, \eta \in \sO}\left\lvert 
b^{(2)}\left(x^\top\eta^{(2)}\right)\right\rvert < \infty$ as $b^{(2)}(\cdot)$ is continuous, $\sX$ is compact, and $\sO$ is bounded. Additionally, we have 
$$\left\lvert 
b^{(1)}\left(x^\top\eta_2^{(2)}\right)\right\rvert \le \sup_{\eta \in \sO}\left\lvert 
b^{(1)}\left(x^\top\eta^{(2)}\right)\right\rvert.$$
Similarly, we can always find a sufficiently small neighborhood $\sO$ containing $\eta^*$ such that $$\sup_{\eta \in \sO}\left\lvert 
b^{(1)}\left(x^\top\eta^{(2)}\right)\right\rvert \in \sL^2_1(\joint).$$
Putting these results together, we obtain
\begin{equation}
\notag
\begin{aligned}
&\norm{\vp_{\eta_1}(z)-\vp_{\eta_2}(z)}\\
&\le \norm{\eta_1^{(1)}-\eta_2^{(1)}}_2\norm{yx}+\sup_{\eta \in \sO}\norm{\eta^{(1)}}_2\sup_{x \in \sX, \eta \in \sO}\left\lvert 
b^{(2)}\left(x^\top\eta^{(2)}\right)\right\rvert \norm{x}\norm{\eta_1^{(2)}-\eta_2^{(2)}} \\
&\quad+ \sup_{\eta \in \sO}\left\lvert 
b^{(1)}\left(x^\top\eta^{(2)}\right)\right\rvert\norm{\eta_2^{(1)}-\eta_1^{(1)}}_2\\
&\le L(z)\rho(\eta_1, \eta_2),
\end{aligned}
\end{equation}
where
$$L(z) = \max\set{\norm{yx}+\sup_{\eta \in \sO}\left\lvert 
b^{(1)}\left(x^\top\eta^{(2)}\right)\right\rvert,\quad \left(\sup_{\eta \in \sO}\norm{\eta^{(1)}}_2\sup_{x \in \sX, \eta \in \sO}\left\lvert 
b^{(2)}\left(x^\top\eta^{(2)}\right)\right\rvert\right) \norm{x}} \in  \sL_1^2(\joint).$$
Next, we show that a consistent estimator of $\eta^*$ is $\cetaest = \left(\left[\frac{1}{n}\sum_{i=1}^nb^{(2)}\left(X_i^\top\hat{\theta}_n\right)X_iX_i^\top\right]^{-1}, \hat{\theta}_n\right)$. Since $b^{(2)}(\cdot)$ is continuous, $b^{(2)}\left(x^\top\theta\right)xx^\top$ is a continuous (matrix-valued) function of $\theta$ for every $x$. Then, $b^{(2)}\left(x^\top\theta\right)xx^\top$ is a Gilvenko-Cantelli class (which means that it is Gilvenko-Cantelli component-wise) by Example 19.9 of \cite{van2000asymptotic} and compactness of $\Theta$. Therefore,
$$\norm{\frac{1}{n}\sum_{i=1}^nb^{(2)}\left(X_i^\top\hat{\theta}_n\right)X_iX_i^\top-\bE\left[b^{(2)}\left(X^\top\thetaest\right)XX^\top\right]}_2 = \littleO_p(1).$$
Similarly,
\begin{equation}
\notag
\begin{aligned}
&\left\lvert b^{(2)}\left(x^\top\theta_1\right)-b^{(2)}\left(x^\top\theta_2\right)\right\rvert \le \sup_{x \in \sX, \theta \in \Theta}\left\lvert 
b^{(3)}\left(x^\top\theta\right)\right\rvert \norm{x}\norm{\theta_1-\theta_2},
\end{aligned}
\end{equation}
where $\sup_{x \in \sX, \theta \in \Theta}\left\lvert 
b^{(3)}\left(x^\top\theta\right)\right\rvert < \infty$ as $b^{(3)}(\cdot)$ is continuous and both $\sX$ and $\Theta$ are compact. Therefore, it follows that $$\bE\set{\left[b^{(2)}\left(X^\top\thetaest\right)-b^{(2)}\left(X^\top\theta^*\right)\right]XX^\top} \le  \bE\left[\norm{X}^3\right] \sup_{x \in \sX, \theta \in \Theta}\left\lvert 
b^{(3)}\left(x^\top\theta\right)\right\rvert \norm{\thetaest-\theta^*} =  \littleO_p(1)$$ as $\thetaest$ is consistent. Putting these results together, we have
$$\norm{\frac{1}{n}\sum_{i=1}^nb^{(2)}\left(X_i^\top\hat{\theta}_n\right)X_iX_i^\top-\bE\left[b^{(2)}\left(X^\top\theta^*\right)XX^\top\right]}_2 = \littleO_p(1),$$
and by continuous mapping,
$$\norm{\left[\frac{1}{n}\sum_{i=1}^nb^{(2)}\left(X_i^\top\hat{\theta}_n\right)X_iX_i^\top\right]^{-1}-\set{\bE\left[b^{(2)}\left(X^\top\theta^*\right)XX^\top\right]}^{-1}}_2 = \littleO_p(1).$$
\end{proof}

\subsection{Proof of claim in Example~\ref{exmp:kendall}}

Here, we validate the remaining parts of Assumption~\ref{asu:IF_Lipschitz} for Example~\ref{exmp:kendall} under additional assumptions. Following the notation of Example~\ref{exmp:kendall},  recall that the supervised estimator is Kendall's $\tau$:
\begin{equation}
\label{eq:apdx_kendall}
\thetaest = \frac{2}{n(n-1)}\sum_{i<j}I\set{(U_i-U_j)(V_i-V_j)>0}
\end{equation} 
which is regular and asymptotically linear with influence function 
$$\vp_{\eta^*}(z) = 2\bP_Y^*\set{(U-u)(V-v)>0}-2\theta^*,$$
where $\eta^* = \left(\bP_Y^*\set{(U-u)(V-v)>0}, \theta^*\right)$. For simplicity, denote $h^*(y) = \bP_Y^*\set{(U-u)(V-v)>0}$. By imposing additional assumptions on $h^*(y)$, the next proposition guarantees that \eqref{eq:apdx_kendall} satisfies Assumption~\ref{asu:IF_Lipschitz}. 
\begin{proposition}
Suppose that $\norm{h^*(y)}_V^* \le M$ for some $M > 1$, where $\norm{\cdot}_V^*$ represents the uniform sectional variation norm of a function \citep{van1995efficient, gill1995inefficient} . Further suppose that $\Theta$ is a bounded subset of $\bR^2$ such that $\theta^* \in \text{int}(\Theta)$. Then, the Kendall's $\tau$ \eqref{eq:apdx_kendall} satisfies Assumption~\ref{asu:IF_Lipschitz} with $$\sO = \set{h(y)-\theta:\bR^2 \to \bR, \norm{h}_V^*\le M, \theta \in \Theta},$$ and the estimator $$\cetaest = \left(\frac{1}{n}\sum_{i=1}^nI\set{(U_i-u)(V_i-v)>0}, \thetaest\right).$$
\end{proposition}
\begin{proof}

We first validate part (b) of Assumption~\ref{asu:IF_Lipschitz}. Clearly, $h^*(y)-\theta^* \in \sO$, and we have shown that the influence function of U-statistics of kernel $R$ is $R$-Lipshitz in Section~\ref{subsec:u_statistics}. Therefore, it remains to prove that $\set{\vp_\eta: \eta \in \sO} = \sO$ is $\joint$-Donsker. Re-write $\sO$ as
$$\sO = \set{h(y) - h^*(y) + h^*(y)-\theta:\bR^2 \to \bR, \norm{h}_V^*\le M, \theta \in \Theta},$$
which is the pairwise sum of two sets: $\sO = \sO_1 + \sO_2$ where $\sO_1 = \set{h(y) - h^*(y): \norm{h}_V^*\le M}$ and $\sO_2 = \set{h^*(y)-\theta: \theta \in \Theta}$. By Example 2.10.9 of \cite{wellner2013weak}, if both  $\sO_1$ and $\sO_2$ are $\joint$-Donsker, then their pairwise sum $\sO$ is also $\joint$-Donsker. For $\sO_1$, by triangle inequality of the norm $\norm{h - h^*}_V^* \le \norm{h}_V^* + \norm{h^*}_V^* \le 2M$, therefore $\sO_1$ is a subset of functions with uniform sectional variation norm bounded by $2M$. By Example 1.2 of \cite{van1995efficient}, this implies that $\sO_1$ is $\joint$-Donsker. For $\sO_2$, it is a set indexed by a finite-dimensional parameter $\theta \in \Theta$ where $\Theta$ is bounded and $h(y)^* - \theta$ is $1$-Lipschitz in $\theta$. Hence, by Example 19.6 of \cite{van2000asymptotic}, $\sO_2$ is $\joint$-Donsker. 

We now validate part (c) of Assumption~\ref{asu:IF_Lipschitz} with $\cetaest$. Since any indicator function on $\bR^2$ has uniform sectional variation norm bounded by 1, we have
$$\norm{\frac{1}{n}\sum_{i=1}^nI\set{(U_i-u)(V_i-v)>0}}_V^* \le \frac{1}{n}\sum_{i=1}^n\norm{I\set{(U_i-u)(V_i-v)>0}}_V^* \le 1.$$
Further, by asymptotic linearity of the U-statistics $\thetaest$, we have $\norm{\thetaest-\theta^*}_2 = \littleO_p(1)$. Therefore, $$\joint\set{\cetaest \in \sO} = \joint\set{\thetaest \in \Theta} \to 1$$
by the definition of $\sO$. Further, since the class of indicator functions is $\joint$-Gilvenko Cantelli, it follows that
$$\norm{\frac{1}{n}\sum_{i=1}^nI\set{(U_i-u)(V_i-v)>0}-h^*(y)}_\infty = \littleO_p(1).$$
Combined with the consistency of $\cthetaest$, we have
$\rho(\cetaest, \eta^*) = \littleO_p(1)$, which completes the proof.
\end{proof}

\subsection{Proof of claim in Remark~\ref{rmk:efficiency_ATE}}
We prove a general result. 
\begin{proposition}
Let $f(Z) \in \lp$ and let $V= g(X)$ be a measurable function of $X$. Denote $f_X(x) = \bE[f(Z)\mid X=x]$ and $f_V(v) = \bE[f(Z) \mid V=v]$. Then
\begin{equation}
\label{eq:apdx_condition_on_more}
\begin{aligned}
&\left\langle f(z_1)-\gamma f_V(v_1)+\gamma f_V(v_2) ,\left[f(z_1)-\gamma f_V(v_1)+\gamma f_V(v_2)\right]^\top   \right\rangle_\sH \\
&\succeq \left\langle f(z_1)-\gamma f_X(x_1)+\gamma f_X(x_2) ,[f(z_1)-\gamma f_X(x_1)+\gamma f_X(x_2)]^\top   \right\rangle_\sH,
\end{aligned}
\end{equation}
where $\langle \cdot, \cdot \rangle_\sH$ is defined in \eqref{eq:apdx_sH_norm}.
\end{proposition}
\begin{proof}
By the linearity of inner product, 
$$\left\langle h_1,h_1^\top   \right\rangle_\sH = \left\langle h_1-h_2,(h_1-h_2)^\top   \right\rangle_\sH+\left\langle h_2,h_2^\top \right\rangle_\sH + \left\langle h_1-h_2,h_2^\top   \right\rangle_\sH + \left\langle h_2,(h_2-h_1)^\top   \right\rangle_\sH,$$
for arbitrary $h_1$ and $h_2$. Therefore, it suffices to prove that the cross term satisfies
$$\left\langle \gamma f_X(x_1)-\gamma f_V(v_1) - \gamma f_X(x_2) + \gamma f_V(v_2) ,[f(z_1)-\gamma f_X(x_1)+\gamma f_X(x_2)]^\top   \right\rangle_\sH = 0.$$
By the definition of $\langle \cdot, \cdot \rangle_\sH$,
\begin{equation}
\notag
\begin{aligned}
&\left\langle \gamma f_X(x_1)-\gamma f_V(v_1) - \gamma f_X(x_2) + \gamma f_V(v_2) ,[f(z_1)-\gamma f_X(x_1)+\gamma f_X(x_2)]^\top   \right\rangle_\sH \\
&= \gamma\left\langle f_X(x_1)- f_V(v_1) - f_X(x_2) + f_V(v_2) ,[f(z_1)- f_X(x_1) + (1-\gamma)f_X(x_1) + \gamma f_X(x_2)]^\top   \right\rangle_\sH\\
&= \gamma\set{\left\langle f_X- f_V, f- f_X + (1-\gamma)f_X\right\rangle_\lp +  \frac{1-\gamma}{\gamma}\left\langle f_V- f_X, \gamma f_X\right\rangle_\lpx}\\
&= \gamma\set{\left\langle f_X- f_V, (1-\gamma)f_X\right\rangle_\lpx +  \frac{1-\gamma}{\gamma}\left\langle f_V- f_X, \gamma f_X\right\rangle_\lpx}\\
&= 0,
\end{aligned}
\end{equation}
where in the second-to-last equality we used the fact that
$$\left\langle f_X- f_V, f- f_X\right\rangle_\lp = 0$$
by taking the conditional expectation on $X$ and noting that $V = g(X)$. 
\end{proof}

\section{Additional results}
\subsection{Technical lemmas}
\begin{lemma}[Decomposition of conditional variances]
\label{lem:apdx_decomp_var:}
For a random variables $Z = (X,Y) \sim \joint$ and $X \sim \marginal$. Let $f,g$ be any functions such that $f(z) \in \sL^2_p(\joint)$ and $g(x) \in \sL^2_p(\marginal)$. Then,
$$\Var[f(Z)-g(X)] = \Var\set{f(Z)-\bE[f(Z)\mid X]} + \Var\set{\bE[f(Z)\mid X]-g(X)}.$$
\end{lemma}
\begin{proof}
By the conditional variance formula, 
$$\Var\set{f(Z)-\bE[f(Z)\mid X]} = \bE\set{\Var\set{f(Z)-\bE[f(Z)\mid X]}},$$
and
$$\Var\set{\bE[f(Z)\mid X] - g(X)} = \Var\set{\bE\set{\bE[f(Z)\mid X] - g(X)}}.$$
It follows that
\begin{equation}
\notag
\begin{aligned}
\Var[f(Z)-g(X)] &= \Var\set{f(Z)-\bE[f(Z)\mid X] + \bE[f(Z)\mid X] - g(X)}\\
&= \Var\set{\bE\set{\bE[f(Z)\mid X] - g(X)}} + \bE\set{\Var\set{f(Z)-\bE[f(Z)\mid X]}}\\
&=\Var\set{\bE[f(Z)\mid X] - g(X)} + \Var\set{f(Z)-\bE[f(Z)\mid X]}.
\end{aligned}
\end{equation}
\end{proof}

\begin{lemma}[Multivariate pythagorean theorem]
\label{lem:apdx_pythagorean}
For any function $g(x) \in \sL^2_{d}(\marginal)$, let 
$$\sA = \set{\*Bg(x): \*B \in \bR^{p \times d}}$$ denote its linear span in $\sL^2_{p}(\marginal)$. For any $f(x) \in \sL^2_{p}(\marginal)$,
$$\Var\left[f(X)\right] = \Var\left[f(X) - \Pi(f(X)\mid \sA)\right] + \Var\left[\Pi(f(X)\mid \sA)\right],$$
where $\Pi(a\mid \sA)$ denote the projection operator that project $a$ onto a linear space $\sA$.
\end{lemma}
\begin{proof}
See Theorem 3.3 of \cite{tsiatis2006semiparametric}.
\end{proof}

\begin{lemma}
\label{lem:matrix_CS}
For a random variable $Z \sim \joint$, and let $f,g$ be any functions such that $f \in \sL^2_p(\joint)$ and $g \in \sL^2_q(\joint)$. Then,
$$\norm{\joint\left(fg^\top\right)}_2 \le \norm{f}_{\sL^2(\joint)}\norm{g}_{\sL^2(\joint)}.$$
\end{lemma}
\begin{proof}
Consider any non-random vector $a \in \bR^q$ such that $\norm{a}_2 = 1$. By the definition of operator norm,
\begin{equation}
\notag
\begin{aligned}
\norm{\joint\left(fg^\top\right)a}_2 &= \norm{\joint\left(fg^\top a\right)}_2 \\
&\le \joint\norm{fg^\top a}_2\\
&\le  \joint\norm{f}\norm{g}\\
&\le \norm{f}_{\sL^2(\joint)}\norm{g}_{\sL^2(\joint)},
\end{aligned}
\end{equation}
by Cauchy-Schwartz inequality.
\end{proof}

\begin{lemma}[Lemma 19.24 of \cite{van2000asymptotic}.]
\label{lem:19.24 of vdv}
Let $\sS$ denote a $\joint$-Donsker class of measurable functions over $(\sZ, \sF, \joint)$, and let $\hat{s}_n$ denote a sequence of random functions that takes value in $\sS$ and satisfies
$$\norm{\hat{s}_n-s}_{\lp} = \littleO_p(1),$$
for some element $s \in \sS$. Then,
$$\bG_n(\hat{s}_n-s) = \littleO_p(1).$$
\end{lemma}

\begin{lemma}
\label{lem:apdx_1st_lem}
Let $\sS = \set{s_\eta(z): \sZ \to \bR^p,\eta \in \Omega}$ be a class of measurable functions on $(\sZ, \sF, \joint)$ indexed by $\eta$, where  $\Omega$ is a metric space with metric $\rho$. Let $\eta^* \in \text{int}(\Omega)$ such that $s_{\eta^*}(z) \in \sL^2_p(\joint)$. Suppose that there exists a set $\sO \subset \Omega$ such that $\eta^* \in \sO$, the class of functions $\set{s_\eta(z): \eta \in \sO}$ is $\joint$-Donsker, and for all $\set{\eta_1,\eta_2} \subset \sO$ there exists a square-integrable function $S:\sZ \to \bR^+$ such that
$\norm{s_{\eta_1}(Z)-s_{\eta_2}(Z)} \le S(Z)\rho(\eta_1, \eta_2)$
almost surely. In addition, suppose there exists an estimator $\cetaest$ of $\eta^*$ such that $\joint\set{\cetaest \in \sO} \to 1$ and $\rho(\cetaest, \eta^*) = \littleO_p(1)$. Then, the following hold:
\begin{itemize}
    \item $\sup_{\eta \in \sO}\norm{\left(\empiricalPn-\joint\right)s_\eta} = \littleO_p(1)$,
    \item $\norm{\cs-\ts}^2_{\sL^2(\joint)} = \littleO_p(1)$,
    \item $\norm{\bG_n(\cs-\ts)} = \littleO_p(1)$,
    \item $\norm{\empiricalPn(\cs)} = \bigO_p(1)$; and if additionally $\joint(\ts) = 0$, then $$\norm{\empiricalPn(\cs)} = \bigO_p\left(\rho(\cetaest, \eta^*) \wedge \frac{1}{\sqrt{n}}\right) = \littleO_p(1).$$
\end{itemize}
\end{lemma}
\begin{proof}
Since the class $\sS = \set{s_\eta(z): \eta \in \sO}$ is $\joint$-Donsker, it is also $\joint$-Gilvenko-Cantelli, hence $$\sup_{\eta \in \sO}\norm{\left(\empiricalPn-\joint\right)s_\eta} = \littleO_p(1).$$ 
By assumption $\norm{s_{\eta_1}(Z)-s_{\eta_2}(Z)} \le S(Z)\rho(\eta_1, \eta_2)$ for all $\set{\eta_1,\eta_2} \subset \sO$, therefore when $\cetaest \in \sO$, we have 
$$\norm{\cs-\ts}^2_{\sL^2(\joint)} \le \rho^2(\cetaest, \eta^*)\norm{S(z)}_{\sL^2(\joint)}^2 = \bigO_p(\rho^2(\cetaest, \eta^*)) = \littleO_p(1).$$
Further, by assumption $\joint\set{\cetaest \notin \sO} \to 0$, it then follows that $\norm{\cs-\ts}^2_{\sL^2(\joint)} = \littleO_p(1)$. Therefore, by Lemma~\ref{lem:19.24 of vdv}, $\norm{\bG_n(\cs-\ts)} = \littleO_p(1)$. Then,
\begin{equation}
\begin{aligned}
\notag
\norm{\empiricalPn(\cs)}&\le \frac{1}{\sqrt{n}}\norm{\bG_n(\cs-\ts)} + \norm{\joint(\cs-\ts)} + \norm{\empiricalPn(\ts)}\\
&= \littleO_p\left(\frac{1}{\sqrt{n}}\right) + \norm{\joint(\cs-\ts)} + \norm{\empiricalPn(\ts)}\\\
&\le \bigO_p\left(\frac{1}{\sqrt{n}}\right) + \norm{\cs-\ts}_{\sL^2(\joint)} + \norm{\joint(\ts)}\\
&= \bigO_p\left(\frac{1}{\sqrt{n}}\right) + \bigO_p(\rho(\cetaest, \eta^*)) +\norm{\joint(\ts)}\\
&= \bigO_p\left(\rho(\cetaest, \eta^*) \wedge \frac{1}{\sqrt{n}}\right) +\norm{\joint(\ts)}.
\end{aligned}
\end{equation}
Therefore, $\norm{\empiricalPn(\cs)} = \bigO_p(1)$. Further, if $\joint(\ts) = 0$, $\norm{\empiricalPn(\cs)} = \left(\rho(\cetaest, \eta^*) \wedge \frac{1}{\sqrt{n}}\right)$.
\end{proof}

\begin{lemma}
\label{lem:apdx_2nd_lem}
Under the conditions of Lemma~\ref{lem:apdx_1st_lem}. Let $\sG = \set{g_\eta(z): \sZ \to \bR^p, \eta \in \Omega}$ denote another class of measurable functions on $(\sZ, \sF, \joint)$ such that (i) $g_{\eta^*}(z) \in \sL^2(\joint)$ (ii) for all $\set{\eta_1,\eta_2} \subset \sO$ where $\sO$ follows from Lemma~\ref{lem:apdx_1st_lem} there exists a square-integrable function $G:\sZ \to \bR^+$ such that
$\norm{g_{\eta_1}(Z)-g_{\eta_2}(Z)} \le G(Z)\rho(\eta_1, \eta_2)$
almost surely. Then
$$\norm{\empiricalPn\left[g_{\cetaest}(\cs-\ts)^\top\right]}_2 = \bigO_p(\rho(\cetaest, \eta^*)) = \littleO_p(1),$$
where $\cetaest$ is defined in  Lemma~\ref{lem:apdx_1st_lem}.
\end{lemma}
\begin{proof}
When $\cetaest \in \sO$, 
\begin{equation}
\notag
\begin{aligned}
&\norm{\empiricalPn\left[g_{\cetaest}(\cs-\ts)^\top\right]}_2 \\
&\le \empiricalPn\set{\norm{g_{\cetaest}(\cs-\ts)^\top}_2}\\
&\le \empiricalPn\set{\norm{g_{\cetaest}}\norm{\cs-\ts}}\\
&\le \empiricalPn\set{(\norm{g_{\cetaest}-g_{\eta^*}} + \norm{g_{\eta^*}})\norm{\cs-\ts}}\\
&\le \empiricalPn\set{\norm{g_{\cetaest}-g_{\eta^*}}\norm{\cs-\ts}} + \empiricalPn\set{\norm{g_{\eta^*}}\norm{\cs-\ts}}\\
&\overset{\text{a.s.}}{\le} \empiricalPn\set{SG}\rho(\cetaest, \eta^*)^2 + \empiricalPn\set{S\norm{g_{\eta^*}}}\rho(\cetaest, \eta^*)\\
&=  \joint\set{SG}\rho(\cetaest, \eta^*)^2 + \joint\set{S\norm{g_{\eta^*}}}\rho(\cetaest, \eta^*) + \bigO_p\left(\frac{1}{\sqrt{n}}\rho(\cetaest, \eta^*)\right)\\
&\le \norm{S}_{\sL^2(\joint)}\norm{G}_{\sL^2(\joint)}\rho(\cetaest, \eta^*)^2 + \norm{S}_{\sL^2(\joint)}\norm{g_{\eta^*}}_{\sL^2(\joint)}\rho(\cetaest, \eta^*) + \bigO_p\left(\frac{1}{\sqrt{n}}\rho(\cetaest, \eta^*)\right)\\
&= \bigO_p\left(\rho(\cetaest, \eta^*)^2\right) + \bigO_p\left(\rho(\cetaest, \eta^*)\right)+ \bigO_p\left(\frac{1}{\sqrt{n}}\rho(\cetaest, \eta^*)\right)\\
&=\bigO_p(\rho(\cetaest, \eta^*))\\
&=\littleO_p(1).
\end{aligned}
\end{equation}
Therefore the statement follows from the fact $\joint\set{\cetaest \notin \sO} \to 0$.
\end{proof}

\begin{lemma}
\label{lem:apdx_3rd_lem}
Under the conditions of Lemma~\ref{lem:apdx_1st_lem}, consider the centered class
$$\sS^* = \set{s_\eta^0(z) = s_\eta(z)-\joint(s_\eta): \eta \in \Omega}.$$
Then, for all $\set{\eta_1,\eta_2} \subset \sO$ where $\sO$ follows from Lemma~\ref{lem:apdx_1st_lem}, it holds that
$$\norm{s^0_{\eta_1}(Z)-s^0_{\eta_2}(Z)} \le \left[S(Z)+\norm{S}_{\sL^2(\joint)}\right]\rho(\eta_1, \eta_2),$$
almost surely. Further, the class $\sS^* = \set{s_\eta^0(z) = s_\eta(z)-\joint(s_\eta): \eta \in \sO}$ is $\joint$-Donsker.
\end{lemma}
\begin{proof}
\begin{equation}
\notag
\begin{aligned}
\norm{s^0_{\eta_1}(Z)-s^0_{\eta_2}(Z)} &\le  \norm{s_{\eta_1}(Z)-s_{\eta_2}(Z)} + \norm{\joint(s_{\eta_1}-s_{\eta_2})}\\
&\le \norm{s_{\eta_1}(Z)-s_{\eta_2}(Z)} + \joint\norm{s_{\eta_1}-s_{\eta_2}}\\
&\le \norm{s_{\eta_1}(Z)-s_{\eta_2}(Z)} + \norm{s_{\eta_1}-s_{\eta_2}}_{\sL^2(\joint)}\\
&\le\left[ S(Z) + \norm{S}_{\sL^2(\joint)}\right]\rho(\eta_1, \eta_2).
\end{aligned}
\end{equation}
Since $\bG_n(s_\eta) = \bG_n(s^0_\eta)$ for all $\eta \in \Omega$, $\joint$-Donsker follows immediately.
\end{proof}

\begin{lemma}
\label{lem:apdx_order_Kn}
Let $\set{X_i}_{i=1}^n$ be i.i.d. random vectors with mean zero and variance $\*I_{K_n}$, where $K_n$ satisfies $K_n \to \infty$ and $\frac{K_n}{n} \to 0$. Let $a \in \bR^{K_n}$ be any fixed vector such that $\norm{a} = \bigO(\sqrt{K_n})$.
Then
\begin{itemize}
    \item $\frac{1}{n}\sum_{i=1}^n a^\top X_i  =  \bigO_p\left(\sqrt{\frac{K_n}{n}}\right)$
    \item $\norm{\frac{1}{n}\sum_{i=1}^nX_i} = \bigO_p\left(\sqrt{\frac{K_n}{n}}\right)$
\end{itemize}
\end{lemma}
\begin{proof}
We have:
$$\bE\left[\left(\frac{1}{n}\sum_{i=1}^n a^\top X_i\right)^2\right] = \frac{\norm{a}^2}{n} = \bigO\left(\frac{K_n}{n}\right).$$
$$\bE\left[\left(\frac{1}{n}\sum_{i=1}^nX_i\right)^\top\left(\frac{1}{n}\sum_{i=1}^nX_i\right)\right] = \frac{K_n}{n}.$$
Therefore, by Markov's inequality, for any $\epsilon>0$,
\begin{equation}
\notag
\begin{aligned}
\bP\set{\left\lvert\frac{1}{n}\sum_{i=1}^n a^\top X_i \right\rvert> \epsilon}
&= \bP\set{\left(\frac{1}{n}\sum_{i=1}^n a^\top X_i \right)^2> \epsilon^2}\\
&\le \frac{\bE\left[\left(\frac{1}{n}\sum_{i=1}^n a^\top X_i\right)^2\right]}{\epsilon^2}\\
&= \frac{\norm{a}^2}{n\epsilon^2},
\end{aligned}
\end{equation}
which we see that let $\epsilon = C^{-1}\sqrt{\frac{K_n}{n}}$ leads to the result. The seond equation can be proved in a similar manner.
\end{proof}

\subsection{Additional definitions and results in semiparametric efficiency theory}
\label{subsec:apdx_efficiency_iid}
Following the notation in Section~\ref{sec:efficiency_theory}, we provide additional definitions and results from classical semiparametric efficiency theory. Recall that for a model $\jointmodel$ that contains $\joint$, the \emph{tangent set} at $\joint$ relative to $\jointmodel$ is
the set of score functions at $\joint$ from all one-dimensional regular parametric sub-models of $\jointmodel$, and the corresponding \emph{tangent space} is the closed linear span of the tangent set. A Euclidean functional $\theta: \jointmodel \to \bR^p$ is \emph{pathwise differentiable} at $\joint$ relative to $\jointmodel$ if, for an arbitrary smooth one-dimensional parametric sub-model $\sP_T = \set{\bP_t: t \in T \subset \bR}$ of $\jointmodel$ such that $\joint = \bP_{t^*}$ for some $t^* \in T$, it holds that $\frac{d\theta(\bP_{t^*})}{dt} = \bE[\gradient(Z) s_{t^*}(Z)]$, where $\gradient(z) \in \lp$ is a \emph{gradient} of $\thetafunctional$ at $\joint$ relative to $\jointmodel$, and $s_{t^*}$ is the score function of $\sP_T$ at $t^*$. The next lemma relates the set of gradients to the set of influence functions of regular asymptotically linear estimators.
\begin{lemma}
\label{lem:apdx_gradients_IF}
Let $\thetaest$ be an asymptotically linear estimator of $\theta^*$ with influence function $\IF(z)$, then $\thetafunctional$ is pathwise differentiable at $\joint$ relative to $\jointmodel$ with a gradient $\IF(z)$ if and only if $\thetaest$ is a regular estimator of $\theta^*$ over $\jointmodel$.
\end{lemma}
For a proof, see \citep{pfanzagl1990estimation, van1991differentiable}. Lemma~\ref{lem:apdx_gradients_IF} implies the equivalence between the set of gradients and the set of influence functions of regular asymptotically linear estimators. If a gradient $\cgradient$ satisfies $a^\top\cgradient \in \tangent$ for all $a \in \bR^p$, then it is the \emph{canonical gradient} or \emph{efficient influence function} of $\thetafunctional$ at $\joint$ relative to $\jointmodel$. The next lemma characterizes the set of gradients for a pathwise differentiable parameter. 
\begin{lemma}[Characterization of the set of gradients]
\label{lem:apdx_gradient_set}
Let $\thetafunctional: \jointmodel \to \bR^p$ be a functional that is pathwise differentiable at $\joint \in \jointmodel$ relative to $\jointmodel$. Then the set of gradients of $\thetafunctional$ relative to $\jointmodel$ can be represented as $\set{\gradient(z) + g(z): a^\top g(z) \perp \tangent,  \forall a \in \bR^p}$, where $\gradient$ is any gradient of $\thetafunctional$.
\end{lemma}
\begin{proof}
See Theorem 3.4 of \cite{tsiatis2006semiparametric}.
\end{proof}
Recall the model \eqref{eq:separable_model}:
\begin{equation}
\notag
\jointmodel = \set{\bP = \bP_X\bP_{Y\mid X}:\bP_X \in \marginalmodel, \bP_{Y\mid X} \in \conditionalmodel},
\end{equation} 
which models the marginal distribution and the conditional distribution separately. We next characterize the tangent space relative to  \eqref{eq:separable_model}. 
\begin{lemma}
\label{lem:apdx_tangent_separable}
Let $\jointmodel$ be defined as in \eqref{eq:separable_model}. Then, the tangent space $\tangent$ can be written as
$$\tangent = \tangentmarginal \oplus \tangentconditional,$$
where $\tangentmarginal$ is the tangent space at $\marginal$ relative to $\marginalmodel$, and $\tangentconditional$ is the tangent space at $\conditional$ relative to $\conditionalmodel$.
\end{lemma}
\begin{proof}
See Lemma 1.6 of \cite{van2003unified}.
\end{proof}

\subsection{Semiparametric efficiency theory under the OSS setting}
\label{sec:apdx_efficiency_noniid}
In this section, we extend the classical semiparametric efficiency theory to the OSS setting, building on ideas from \cite{bickel2001inference}. 
We first review some notation. Let $Z = (X, Y)$ be a random variable that takes value in the probability space $(\sZ, \sF, \joint)$, where $\sF$ is a sigma-algebra on $\sZ$. Let $(\sX, \sF_X, \marginal)$ be the corresponding probability space of $(\sZ, \sF, \joint)$ over $\sX$.
Consider the product probability space 
$$(\sZ \times \sX, \sF \otimes \sF_X, \ossjoint),$$
where $\sF \otimes \sF_X$ represents the product sigma-algebra of $\sF$ and $\sF_X$, and $\ossjoint = \joint \times \marginal$ represents the product measure of $\joint$ and $\marginal$. Note that the probability measure here depends only on $\joint$ as $\marginal$ is the corresponding marginal distribution of $\joint$ over $\sX$. Suppose $\joint$ belong to a model $\jointmodel$. As usual, we assume that there is some common dominating measure $\mu$ for $\jointmodel$ such that $\jointmodel$ can be equivalently represented by its density functions with respect to $\mu$. Our goal here is to estimate a Euclidean functional $\theta(\cdot): \jointmodel \to \Theta \subseteq \bR^p$ from the data $\labeled \cup \unlabeled$, where $\labeled \iidsim \joint$ and $\unlabeled \iidsim \marginal$. The full data is thus generated by 
$$\bQ^{n,N} = \left(\prod_{i=1}^n\joint\right)\times\left(\prod_{i=n+1}^{n+N}\marginal\right).$$
Let $n \to \infty$ and $\frac{N}{n+N} \to \gamma \in (0,1)$.  The empirical distribution of $\labeled \cup \unlabeled$ is $\ossempirical = \empiricalPn \times \empiricalPXN$, where $\empiricalPn = \frac{1}{n}\sum_{i=1}^n\delta_{Z_i}$ is the empirical distribution of $\labeled$ and $\empiricalPXN = \frac{1}{N}\sum_{i=n+1}^{n+N}\delta_{X_i}$ is the empirical distribution of $\unlabeled$. 

We first establish the local asymptotic normality property of $\ossempirical$. 
\begin{lemma}
\label{lem:apdx_LAN}
For any regular parametric sub-model 
$\sP_T = \set{\bP_t:t \in T \subset \bR^d} \subset \jointmodel$
such that $\bP_{t^*} = \joint$ for some $t^* \in T$, let $\bP_{t,X}$  be the marginal distribution of $\bP_t$ over $\sX$. Denote $\bQ_t^{n,N} = \left(\prod_{i=1}^n\bP_t\right)\times\left(\prod_{i=n+1}^{n+N}\bP_{t,X}\right).$ Then, for any $h \in \bR^d$, 
\begin{equation}
\label{eq:apdx_LAN}
\begin{aligned}
\log\left(\frac{d\bQ_{t^*+\frac{h}{\sqrt{n}}}^{n,N}}{d\bQ_{t^*}^{n,N}}\right)
&= h^\top\sqrt{n}\left[\ossempirical-\ossjoint\right]\left(s_{t^*}+\frac{\gamma}{1-\gamma}s_{t^*,X}\right)- \frac{1}{2}h^\top\*I_{t^*}h +\littleO_p(1)\\
&\indistribution N\left(- \frac{1}{2}h^\top\*I_{t^*}h, h^\top\*I_{t^*}h\right),
\end{aligned}
\end{equation} 
where $s_{t^*}(z)$ is the score function of $\bP_t$ at $t^*$, $s_{t^*,X}(x)$ is the score function of $\bP_{t,X}$ at $t^*$, and $\*I_{t^*} = \Var\left[s_{t^*}(Z)\right] + \frac{\gamma}{1-\gamma}\Var\left[s_{t^*,X}(X)\right]$. Moreover, $\bQ_{t^*+\frac{h}{\sqrt{n}}}^{n,N}$ and $\bQ_{t^*}^{n,N}$ are mutually contiguous.
\end{lemma}
Here and in the following text, a regular parametric model is interpreted in the sense of quadratic-mean differentiability of the density function at $t^*$. See Chapter 7 of \cite{van2000asymptotic} for details. 
\begin{proof}
By Theorem 7.2 of \cite{van2000asymptotic}, we have
\begin{equation}
\notag
\begin{aligned}
&\log\left(\prod_{i=1}^n\frac{d\bP_{t^*+\frac{h}{\sqrt{n}}}}{d\bP_{t^*}}\right) = h^\top\sqrt{n}(\empiricalPn-\joint)(s_{t^*}) - \frac{1}{2}h^\top \Var[s_{t^*}(Z)] h + \littleO_p(1),
\end{aligned}
\end{equation}
and
\begin{equation}
\notag
\begin{aligned}
&\log\left(\prod_{i=n+1}^{n+N}\frac{d\bP_{t^*+\frac{h}{\sqrt{n}},X}}{d\bP_{t^*,X}}\right) = \log\left(\prod_{i=n+1}^{n+N}\frac{d\bP_{t^*+\sqrt{\frac{N}{n}}\frac{h}{\sqrt{N}},X}}{d\bP_{t^*,X}}\right)\\
&= h^\top\sqrt{\frac{N}{n}}\sqrt{N}(\empiricalPXN-\marginal)(s_{t^*,X}) - \frac{N}{2n}h^\top \Var[s_{t^*,X}(X)] h + \littleO_p(1)\\
&= h^\top\sqrt{\frac{\gamma}{1-\gamma}}\sqrt{N}(\empiricalPXN-\marginal)(s_{t^*,X}) - \frac{\gamma}{2(1-\gamma)}h^\top \Var[s_{t^*,X}(X)] h + \littleO_p(1).
\end{aligned}
\end{equation}
Then, it follows that
\begin{equation}
\notag
\begin{aligned}
\log\left(\frac{d\bQ_{t^*+\frac{h}{\sqrt{n}}}^{n,N}}{d\bQ_{t^*}^{n,N}}\right)
&=\log\left(\prod_{i=1}^n\frac{d\bP_{t^*+\frac{h}{\sqrt{n}}}}{d\bP_{t^*}}\right) + \log\left(\prod_{i=n+1}^{n+N}\frac{d\bP_{t^*+\frac{h}{\sqrt{n}},X}}{d\bP_{t^*,X}}\right) \\
&= h^\top\sqrt{n}(\empiricalPn-\joint)(s_{t^*}) + h^\top\sqrt{\frac{\gamma}{1-\gamma}}\sqrt{N}(\empiricalPXN - \marginal)\left(s_{t^*,X}\right)\\
&\quad - \frac{1}{2}h^\top\left[\Var(s_{t^*}) + \frac{\gamma}{1-\gamma}\Var(s_{t^*,X})\right]h + \littleO_p(1).\\
\end{aligned}
\end{equation}
The convergence in distribution then follows from CLT and independence, and contiguity follows from Le Cam's first lemma.
\end{proof}
Define $l_{t^*}(z_1, x_2): \sZ \times \sX \to \bR^p$ as 
\begin{equation}
\label{eq:OSS_score}
l_{t^*}(z_1, x_2) = s_{t^*}(z_1)+\frac{\gamma}{1-\gamma}s_{t^*,X}(x_2).
\end{equation}
From Lemma~\ref{lem:apdx_LAN}, $l_{t^*}(z_1, x_2)$ can be viewed as the score function of the parametric model $\sP_T$ at $\ossjoint$ with information matrix $\*I_{t^*}$. It is clear that $l_{t^*}(z_1, x_2)$ is $\sF\otimes\sF_X$ measurable and is an element of the space
\begin{equation}
\label{eq:apdx_sH}
\sH = \set{f(z_1)+ \frac{\gamma}{1-\gamma}g(x_2), f(z) \in \lp, g(x):\in \lpx}.
\end{equation}
Consider any vector-valued functions $f_1(z), f_2(z) \in \lp$ and $g_1(x), g_2(z)\in \lpx$. Therefore, $f_1(z_1)+ \frac{\gamma}{1-\gamma}g_1(x_2) \in \sH$ and $f_2(z_1)+ \frac{\gamma}{1-\gamma}g_2(x_2) \in \sH$. Define
\begin{equation}
\label{eq:apdx_sH_norm}
\begin{aligned}
&\left\langle f_1(z_1)+\frac{\gamma}{1-\gamma}g_1(x_2),f_2(z_1)^\top+\frac{\gamma}{1-\gamma}g_2(x_2)^\top \right\rangle_\sH \\
&:= \left\langle f_1(z),f_2(z)^\top\right\rangle_{\lp} + \frac{\gamma}{1-\gamma}\left\langle g_1(x),g_2(x)^\top\right\rangle_{\lpx}\\
&= \bE\left[f_1(Z)f_2(Z)^\top\right] + \frac{\gamma}{1-\gamma}\bE\left[g_1(X)g_2(X)^\top\right].
\end{aligned}
\end{equation}
It is straightforward to verify that this is indeed an inner product, and $\sH$ is a Hilbert space. The information matrix $\*I_t$ can then be represented as
$$\*I_{t^*} = \left\langle l_{t^*}(z_1, x_2), l_{t^*}(z_1, x_2)^\top\right\rangle_\sH.$$

Consider an arbitrary model $\jointmodel$ of $\joint$. We now extend the notion of a tangent space to the OSS setting. Similar to the i.i.d. setting, the tangent set relative to $\jointmodel$ at $\ossjoint$ is defined as the set of all score functions at $\ossjoint$ of one-dimensional regular parametric sub-models of $\jointmodel$. The tangent space at $\ossjoint$ relative to $\jointmodel$, denoted as $\osstangent$, is then the closed linear span of the tangent set in $\sH$. Note that $\osstangent \subseteq \sH$. Next, we characterizes the tangent space at $\bQ(\joint)$ relative to model \eqref{eq:separable_model}. 

\begin{lemma}
\label{lem:apdx_tangent_space_OSS}
The tangent space $\bQ(\joint)$ relative to $\jointmodel$, denoted as $\osstangent$, can be expressed as
\begin{equation}
\notag
\begin{aligned}
\sM = \left\{f(x_1)+  g(z_1) + \frac{\gamma}{1-\gamma}f(x_2): f(z) \in \tangentmarginal, g(x)\in \tangentconditional\right\}.
\end{aligned}
\end{equation}
\end{lemma}
\begin{proof}
It is straightforward to see that $\sM$ is a closed linear space since both $\tangentconditional$ and $\tangentmarginal$ are closed linear spaces by the definition of tangent space.

First, we show that $\osstangent \subseteq \sM$. Consider any one-dimensional  regular parametric sub-model $\sP_T = \set{p_t(z) = p_{t,X}(x)p_{t,Y\mid X}(z): t \in T}$ of $\jointmodel$ such that $p_{t^*}(z)$ is the density of $\joint$ for some $t^* \in T$. Let $s_{t^*}(z)$, $s_{t^*,Y\mid X}(z)$ and $s_{t^*,X}(x)$ denote the score function of $p_t(z)$, $p_{t,X}(x)$, and $p_{t,Y\mid X}(z)$ at $t^*$, respectively. Then we have $s_{t^*}(z) = s_{t^*,X}(x) + s_{t^*,Y\mid X}(z)$. Since $\sP_T \subset \jointmodel$, and by the definition of $\jointmodel$ \eqref{eq:separable_model}, it holds that $\set{p_{t,X}(x):t\in T}$ is a one-dimensional regular parametric sub-model of $\conditionalmodel$, and $\set{p_{t,Y\mid X}(x):t\in T}$ is a one-dimensional regular parametric sub-model of $\conditionalmodel$. Therefore, $s_{t^*,X}(x) \in \tangentmarginal$ and $s_{t^*,Y\mid X}(z) \in \tangentconditional$. By Lemma~\ref{lem:apdx_LAN}, the score function of $\sQ$ at $\ossjoint$ can be written as
\begin{equation}
\notag
\begin{aligned}
l_{t^*}(z_1, x_2) &= s_{t^*}(z_1) + \frac{\gamma}{1-\gamma}s_{t^*,X}(x_2)\\
&= s_{t^*,X}(x_1) + s_{t^*,Y\mid X}(z_1)  + \frac{\gamma}{1-\gamma}s_{t^*,X}(x_2).
\end{aligned}
\end{equation} 
By Lemma~\ref{lem:apdx_tangent_separable}, $\tangent = \tangentmarginal \oplus \tangentconditional$. Therefore, $l_{t^*}(z_1, x_2) \in \sM$. Then, $\osstangent \subseteq \sM$ follows from the fact that $\sM$ is a closed linear space. 

For the other direction, consider arbitrary elements $g(z) \in \tangentconditional$ and $f(x) \in \tangentmarginal$. Then, without loss of generality, there must exists a one-dimensional regular parametric sub-model $\set{p_{t,X}(x):t \in T}$ of $\marginalmodel$ such that $p_{t^*,X}(x)$ is the density of $\marginal$, and $f(x)$ is the score function of $p_{t,X}(x)$ at $t^*$. Similarly, there must exists a one-dimensional regular parametric sub-model $\set{p_{t,Y\mid X}(z):t \in T}$ of $\conditionalmodel$ such that $p_{t^*,Y\mid X}(z)$ is the density of $\conditional$, and $g(z)$ is the score function of $p_{t,Y\mid X}(z)$ at $t^*$. Consider the model $\sP_T = \set{p_t(z) = p_{t,X}(x)p_{t,Y\mid X}(z): t \in T}$, which is a one-dimensional regular parametric sub-model of $\jointmodel$ by its definition, with score function $f(x)+g(z)$ at $t^*$. By Lemma~\ref{lem:apdx_tangent_separable}, $f(x)+g(z) \in \tangent$. Therefore, by Lemma~\ref{lem:apdx_LAN}, we see that the socre at $\ossjoint$ relative to this parametric sub-model is $f(x_1)+g(z_1)+\frac{\gamma}{1-\gamma}f(x_2) \in \osstangent$, and hence $\sM \subseteq \osstangent$.  
\end{proof}

Next, we consider the extention of pathwise differentiability to the OSS setting.
\begin{definition}[Pathwise differentiability]
\label{def:apdx_path}
A Euclidean functional $\theta:\jointmodel \to \Theta \subseteq \bR^p$ is \emph{pathwise differentiable} relative to $\jointmodel$ at $\ossjoint$ if, any one-dimensional regular parametric sub-model $\sP_T \subset \jointmodel$ where $\sP_T = \set{\bP_t:t \in T \subset \bR}$ such that $\bP_{t^*} = \joint$ for some $t^* \in T$, $\theta:\sP_T \to \Theta$ as a function $\theta(t)$ of $t$ satisfies
$$\frac{d\theta(t^*)}{dt} = \langle l_{t^*}(z_1, x_2), \ossgradient(z_1, x_2)\rangle_\sH,$$
for some $\ossgradient(z_1,x_2) \in \sH$, where $l_{t^*}$ is the score function of $\sP_T$ at $\ossjoint$. Then, $\ossgradient(z_1, x_2)$ is a \emph{gradient} of $\thetafunctional$ relative to $\jointmodel$ at $\ossjoint$. 
\end{definition}
If a gradient $\cossgradient(z_1, x_2)$ satisfies $a^\top\cossgradient(z_1, x_2) \in \osstangent$ for all $a \in \bR^p$, then it is a \emph{canonical gradient} or the \emph{efficient influence function} of $\thetafunctional$ relative to $\jointmodel$ at $\ossjoint$. The next lemma establishes the connection between pathwise differentiability under the OSS setting and pathwise differentiability in the usual sense.
\begin{lemma}
\label{lem:apdx_path}
If $\theta:\jointmodel \to \Theta \subseteq \bR^p$ is a pathwise differentiable relative to $\jointmodel$ at $\joint$, then it is pathwise differentiable relative to $\jointmodel$ at $\ossjoint$. Moreover, if $\gradient$ is a gradient of $\thetafunctional$ relative to $\jointmodel$ at $\joint$, then $$\gradient(z_1)-\Tilde{D}_\jointmodel(\joint)(x_1)+\Tilde{D}_\jointmodel(\joint)(x_2)$$
is a gradient of $\thetafunctional$ relative to $\jointmodel$ at $\ossjoint$, where $$\Tilde{D}_\jointmodel(\joint)(x) = \bE\left[D_\jointmodel(\joint)(Z)\mid X = x\right]$$
is the conditional gradient.
\end{lemma}
\begin{proof}
Let $\gradient(z)$ be a gradient of $\thetafunctional$ relative to $\jointmodel$ at $\joint$. For simplicity, denote $D(z) = \gradient(z)$, and $\Tilde{D}(x) = \bE[D(Z)\mid X = x]$. For any one-dimensional regular parametric sub-model $\sP_T$ of $\jointmodel$, pathwise differentiability at $\joint$ implies
\begin{equation}
\notag
\begin{aligned}
\theta(t) &= \theta(t^*) + t\langle s_{t^*},  D  \rangle_{\lpxone} + \littleO(t)\\
&= \theta(t^*) + t\langle s_{t^*}, D\rangle_{\lpxone} + \littleO(t)\\
&=  \theta(t^*) + t\langle  s_{t^*}, D - \Tilde{D}\rangle_{\lpone} +t\frac{\gamma}{1-\gamma}\left\langle s_{t^*,X},  \frac{1-\gamma}{\gamma}\Tilde{D}\right\rangle_{\lpxone} + \littleO(t)\\
&= \theta(t^*) + t\left\langle s_{t^*}(z_1)+\frac{\gamma}{1-\gamma}s_{t^*,X}(x_2), D(z_1)-\Tilde{D}(x_1)+\Tilde{D}(x_2)\right\rangle_\sH + \littleO(t),
\end{aligned}
\end{equation}
where $D(z_1)-\Tilde{D}(x_1)+\Tilde{D}(x_2) \in \sH$ and $s_{t^*}(z_1)+\frac{\gamma}{1-\gamma}s_{t^*,X}(x_2)$ is the score function of $\sP_T$ at $\joint$.
\end{proof}

We now extend the definitions of regular and asymptotically linear estimators to the OSS setting as in Section~\ref{subsec:OSS_efficiency}. The definition of regularity remains unchanged as Definition~\ref{def:regular_OSS} in the main text. We present an equivalent definition of asymptotic linearity here for convenience. 
\begin{definition}[Asymptotic linearity]
\label{def:apdx_AL_OSS}
An estimator $\thetaestnN$ of $\theta^*$ is \emph{asymptotically linear} with influence function $\IF(z_1, x_2) \in \sH$
if
\begin{equation}
\label{eq:apdx_AL}
\begin{aligned}
&\sqrt{n}\left(\thetaestnN-\theta^*\right) = \sqrt{n}\left(\ossempirical - \ossjoint\right)\left[\IF(z_1, x_2)\right]+\littleO_p(1).
\end{aligned}
\end{equation}
\end{definition}
It is easy to verify, by expanding $\ossempirical - \ossjoint$ and using the fact that $\gamma \in (0,1)$, that definition~\ref{def:apdx_AL_OSS} is equivalent to Definition~\ref{def:AL_OSS}. Clearly, if $\thetaestnN$ is an asymptotically linear estimator of $\theta^*$ with influence function $\IF(z_1, x_2)$, its asymptotic variance is $\langle \IF(z_1, x_2), \IF(z_1, x_2)^\top\rangle_\sH$.

Similar to Lemma~\ref{lem:apdx_gradients_IF}, we first establish the equivalence between the set of gradients and the set of influence functions under the OSS setting. 
\begin{lemma}
\label{lem:apdx_gradients_IF_OSS}
Let $\thetaestnN$ be an asymptotically linear estimator of $\theta^*$ with influence function $\IF(z_1, x_2) = \IFa(z_1) + \frac{\gamma}{1-\gamma}\IFb(x_2)$ in the sense of Definition~\ref{def:apdx_AL_OSS}. Then, $\thetafunctional$ is pathwise differentiable at $\ossjoint$ relative to $\jointmodel$ with gradient $\IF(z_1, x_2)$ if and only if $\thetaest$ is a regular estimator of $\theta^*$ over $\jointmodel$ in the sense of Definition~\ref{def:regular_OSS}.
\end{lemma}
\begin{proof}
Following the notation of Lemma~\ref{lem:apdx_LAN}, consider a one-dimensional regular parametric sub-model $\sP_T$. By asymptotic linearity and Lemma~\ref{lem:apdx_LAN}, we have
$$\m{\sqrt{n}(\thetaestnN-\theta^*) \\ \log\left(\frac{d\bQ_{t^*+\frac{h}{\sqrt{n}}}^{n,N}}{d\bQ_{t^*}^{n,N}}\right)} \indistribution N\left( \m{0 \\ -\frac{1}{2}h^2\*I_{t^*}}, \m{\langle \IF, \IF^\top\rangle_\sH & h \langle l_{t^*}, \IF^\top\rangle_\sH \\ h \langle \IF, l_{t^*}\rangle_\sH & \*I_{t^*}} \right).$$
By contiguity and Le Cam's third lemma, it follows that
$$\sqrt{n}\left(\thetaestnN-\theta^*\right) \indistribution N\left(h \langle \IF, l_{t^*}\rangle_\sH, \langle \IF, \IF^\top\rangle_\sH\right)$$
under $\bQ_{t^*+\frac{h}{\sqrt{n}}}^{n,N}$. By regularity of $\thetaestnN$, it follows that
$$\sqrt{n}\left[\thetaestnN-\theta\left(\bP_{t^*+\frac{h}{\sqrt{n}}}\right)\right] \indistribution N\left(h \langle \IF, l_{t^*}\rangle_\sH, \langle \IF, \IF^\top\rangle_\sH\right)$$
under $\bQ_{t^*+\frac{h}{\sqrt{n}}}^{n,N}$. This implies that
$$\sqrt{n}\left[\theta\left(\bP_{t^*+\frac{h}{\sqrt{n}}}\right)-\theta^*\right] \to h \langle \IF, l_{t^*}\rangle_\sH.$$
Pathwise differentiability then follows by the definition of a derivative. 
\end{proof}

Finally, we extend Lemma~\ref{lem:classic_efficiency} to the OSS setting. 
\begin{lemma}
\label{lem:oss_efficiency}
Suppose $\thetafunctional$ is pathwise differentiable at $\ossjoint$ relative to $\jointmodel$ with efficient influence function $\EIF(z_1, x_2)$. Let $\thetaestnN$ be a regular and asymptotically linear estimator of $\theta^*$ in the sense of Definitions~\ref{def:regular_OSS} and ~\ref{def:apdx_AL_OSS} with influence function $\gradient(z_1,x_2)$. Then,
$$\left\langle \gradient(z_1,x_2), \gradient(z_1,x_2)^\top \right\rangle_\sH \succeq \left\langle \EIF(z_1,x_2), \EIF(z_1,x_2)^\top \right\rangle_\sH.$$
\end{lemma}
Hence, the asymptotic variance of any regular and asymptotically linear estimator of $\theta^*$ is lower bounded by $\left\langle \EIF(z_1,x_2), \EIF(z_1,x_2)^\top \right\rangle_\sH$.
\begin{proof}
By Lemma~\ref{lem:apdx_gradients_IF_OSS}, $\gradient(z_1,x_2)$ is a gradient of $\thetaest$ relative to $\jointmodel$.
For any $s(z_1, x_2) \in \osstangent$, without loss of generality, let $\sP_T$ be any one-dimensional regular parametric sub-model of $\jointmodel$ with score function  $s(z_1, x_2)$ at $\bP_{t^*} = \joint$. (Otherwise, $s(z_1, x_2)$ is the limit of a sequence of score functions of parametric sub-models, and the following holds by the continuity of the inner product.) Write $\theta: \sP_T \to \Theta$ as a function $\theta(t)$ of $t$. By pathwise differentiability, 
$$\frac{d \theta(t^*)}{dt} = \langle s(z_1, x_2), \gradient(z_1, x_2) \rangle_\sH = \langle s(z_1, x_2), \EIF(z_1, x_2)\rangle_\sH,$$
which shows that
$\langle s(z_1, x_2), \EIF(z_1, x_2)-\gradient(z_1, x_2)\rangle_\sH = 0$. Since this holds for arbitrary $s(z_1,x_2) \in \osstangent$, we have $\EIF(z_1, x_2)-\gradient(z_1, x_2) \perp \osstangent$ in $\sH$. By definition, $e_i^\top\EIF(z_1, x_2) \in \osstangent$ for all $i \in [p]$, where $e_i \in \bR^p$ is the vector whose $i$-th coordinate is $1$ and the rest are zeros. Therefore, $\langle\EIF(z_1, x_2), \EIF(z_1, x_2)-\gradient(z_1, x_2) \rangle_\sH = 0$. It follows that
\begin{equation}
\notag
\begin{aligned}
\langle\gradient, \gradient^\top \rangle_\sH &= \langle\gradient-\EIF+\EIF, \gradient^\top-\EIF^\top+\EIF^\top \rangle_\sH\\
&= \langle\gradient-\EIF, \gradient^\top-\EIF^\top\rangle_\sH+\langle\EIF, \EIF^\top \rangle_\sH\\
&\succeq \langle\EIF, \EIF^\top \rangle_\sH.\\
\end{aligned}
\end{equation}
\end{proof}

\section{Additional numerical results}

\subsection{Extended results for mean estimation and generalized linear model}
\label{subsec:apdx_simu_mean_glm}

Tables~\ref{tab:apdx_mean_CI}, ~\ref{tab:apdx_glm1_CI}, and~\ref{tab:apdx_glm2_CI} report the coverage of 95\% confidence intervals for methods of mean estimation and generalized linear models presented in Sections~\ref{subsec:simu_mean} and \ref{subsec:simu_glm}. All methods achieve the nominal coverage. 


\begin{table}[ht]
\centering
\begin{tabular}{c|l|rrrrr}
  \hline
Setting & Method & $\gamma = 0.1$ & $\gamma = 0.3$ & $\gamma = 0.5$ & $\gamma = 0.7$ & $\gamma = 0.9$  \\ 
  \hline
 & $\thetaest$ & 0.967 & 0.944 & 0.946 & 0.938 & 0.954  \\ 
 & $\ossthetaest$ & 0.949 & 0.933 & 0.947 & 0.942 & 0.951 \\ 
 & $\ossthetanpest, K_n = 4$ & 0.944 & 0.944 & 0.944 & 0.953 & 0.942  \\ 
 & $\ossthetanpest, K_n = 9$ & 0.960 & 0.957 & 0.948 & 0.949 & 0.945 \\ 
1 & $\ossthetanpest, K_n = 16$ & 0.950 & 0.939 & 0.941 & 0.951 & 0.948  \\ 
 & $\ossthetaestppi$ (Random Forest) & 0.948 & 0.953 & 0.936 & 0.956 & 0.947  \\ 
 & $\ossthetaestppi$ (Noisy Predictor) & 0.946 & 0.952 & 0.950 & 0.951 & 0.948 \\ 
 & PPI++ (Random Forest) & 0.962 & 0.951 & 0.954 & 0.940 & 0.938  \\ 
 & PPI++ (Noisy Predictor) & 0.950 & 0.949 & 0.946 & 0.942 & 0.946 \\ 
 \hline
    & $\thetaest$ & 0.954 & 0.966 & 0.941 & 0.954 & 0.960  \\ 
    & $\ossthetaest$ & 0.948 & 0.949 & 0.964 & 0.942 & 0.948 \\ 
   & $\ossthetanpest, K_n = 4$ & 0.952 & 0.950 & 0.961 & 0.941 & 0.947 \\ 
   & $\ossthetanpest, K_n = 9$ & 0.946 & 0.955 & 0.955 & 0.938 & 0.942 \\ 
 2  & $\ossthetanpest, K_n = 16$ & 0.944 & 0.945 & 0.957 & 0.941 & 0.937 \\ 
   & $\ossthetaestppi$ (Random Forest)  & 0.956 & 0.968 & 0.953 & 0.945 & 0.944  \\ 
   & $\ossthetaestppi$ (Noisy Predictor) & 0.939 & 0.961 & 0.951 & 0.956 & 0.948 \\ 
   & PPI++ (Random Forest) & 0.953 & 0.950 & 0.971 & 0.948 & 0.940 \\ 
   & PPI++ (Noisy Predictor) & 0.947 & 0.953 & 0.956 & 0.944 & 0.938 \\ 
   \hline
   
    & $\thetaest$ & 0.946 & 0.956 & 0.950 & 0.954 & 0.941   \\ 
& $\ossthetaest$ & 0.948 & 0.952 & 0.958 & 0.932 & 0.949 \\ 
   & $\ossthetanpest, K_n = 4$ & 0.948 & 0.958 & 0.963 & 0.943 & 0.950 \\ 
   & $\ossthetanpest, K_n = 9$ & 0.936 & 0.967 & 0.941 & 0.939 & 0.946 \\ 
 3  & $\ossthetanpest, K_n = 16$ & 0.951 & 0.954 & 0.961 & 0.942 & 0.944 \\ 
   & $\ossthetaestppi$ (Random Forest)  & 0.957 & 0.963 & 0.962 & 0.946 & 0.941 \\ 
   & $\ossthetaestppi$ (Noisy Predictor) & 0.957 & 0.952 & 0.956 & 0.937 & 0.949 \\ 
   & PPI++ (Random Forest) & 0.939 & 0.959 & 0.960 & 0.939 & 0.946 \\ 
   & PPI++ (Noisy Predictor)& 0.940 & 0.948 & 0.948 & 0.960 & 0.934 \\ 
   \hline
\end{tabular}
\caption{95\% coverage of confidence intervals for methods of mean estimation in three settings, averaged over 1,000 simulated datasets. For details of numerical experiments, see Section~\ref{subsec:simu_mean}. }
\label{tab:apdx_mean_CI}
\end{table}

\begin{table}[ht]
\centering
\begin{tabular}{c|l|rrrrr}
  \hline
  Setting & Method & $\gamma = 0.1$ & $\gamma = 0.3$ & $\gamma = 0.5$ & $\gamma = 0.7$ & $\gamma = 0.9$ \\ 
  \hline
 & $\thetaest$ & 0.933 & 0.950 & 0.944 & 0.950 & 0.938  \\ 
 & $\ossthetaest$ & 0.950 & 0.950 & 0.949 & 0.943 & 0.940 \\ 
& $\ossthetanpest, K_n = 4$ & 0.950 & 0.945 & 0.945 & 0.938 & 0.934 \\ 
 & $\ossthetanpest, K_n = 9$ & 0.949 & 0.940 & 0.944 & 0.935 & 0.935 \\ 
 1 & $\ossthetanpest, K_n = 16$ & 0.940 & 0.936 & 0.946 & 0.931 & 0.931 \\ 
 & $\ossthetaestppi$ (Random Forest) & 0.948 & 0.942 & 0.938 & 0.939 & 0.939 \\ 
 & $\ossthetaestppi$ (Noisy Predictor) & 0.942 & 0.941 & 0.949 & 0.934 & 0.939  \\ 
 & PPI++ (Random Forest) & 0.952 & 0.952 & 0.951 & 0.947 & 0.941 \\ 
 & PPI++ (Noisy Predictor) & 0.952 & 0.943 & 0.945 & 0.946 & 0.934 \\ 
 \hline
   & $\thetaest$ & 0.948 & 0.943 & 0.949 & 0.948 & 0.940 \\ 
   & $\ossthetaest$ & 0.944 & 0.951 & 0.944 & 0.953 & 0.945 \\ 
   & $\ossthetanpest, K_n = 4$ & 0.947 & 0.940 & 0.936 & 0.942 & 0.950 \\ 
   & $\ossthetanpest, K_n = 9$ & 0.947 & 0.945 & 0.930 & 0.942 & 0.944 \\ 
   2 & $\ossthetanpest, K_n = 16$ & 0.952 & 0.932 & 0.940 & 0.946 & 0.944 \\ 
   & $\ossthetaestppi$ (Random Forest) & 0.952 & 0.943 & 0.955 & 0.947 & 0.956 \\ 
   & $\ossthetaestppi$ (Noisy Predictor) & 0.954 & 0.952 & 0.923 & 0.937 & 0.953 \\ 
   & PPI++ (Random Forest) & 0.941 & 0.949 & 0.943 & 0.955 & 0.945 \\ 
   & PPI++ (Noisy Predictor) & 0.942 & 0.951 & 0.944 & 0.953 & 0.947 \\ 
   \hline
\end{tabular}
\caption{95\% coverage of confidence intervals of $\theta_1^*$ for methods of Poisson GLM in two settings, averaged over 1,000 simulated datasets. For details of numerical experiments, see Section~\ref{subsec:simu_glm}.}
\label{tab:apdx_glm1_CI}
\end{table}




\begin{table}[ht]
\centering
\begin{tabular}{c|l|rrrrr}
  \hline
  Setting & Method & $\gamma = 0.05$ & $\gamma = 0.25$ & $\gamma = 0.5$ & $\gamma = 0.75$ & $\gamma = 0.95$ \\  
  \hline
  & $\thetaest$ & 0.947 & 0.950 & 0.946 & 0.937 & 0.952  \\ 
  & $\ossthetaest$ & 0.943 & 0.949 & 0.949 & 0.941 & 0.945  \\ 
  & $\ossthetanpest, K_n = 4$ & 0.936 & 0.942 & 0.945 & 0.943 & 0.950 \\ 
  & $\ossthetanpest, K_n = 9$ & 0.944 & 0.944 & 0.956 & 0.936 & 0.953 \\ 
  1 & $\ossthetanpest, K_n = 16$ & 0.938 & 0.937 & 0.950 & 0.937 & 0.939 \\ 
  & $\ossthetaestppi$ (Random Forest) & 0.942 & 0.954 & 0.948 & 0.943 & 0.953 \\ 
  & $\ossthetaestppi$ (Noisy Predictor)& 0.947 & 0.958 & 0.947 & 0.944 & 0.947 \\ 
  & PPI++ (Random Forest) & 0.939 & 0.948 & 0.947 & 0.956 & 0.953 \\ 
  & PPI++ (Noisy Predictor) & 0.940 & 0.946 & 0.950 & 0.946 & 0.940 \\ 
  \hline
   & $\thetaest$ & 0.960 & 0.942 & 0.949 & 0.939 & 0.950 \\ 
   & $\ossthetaest$ & 0.960 & 0.953 & 0.949 & 0.952 & 0.943 \\ 
   & $\ossthetanpest, K_n = 4$ & 0.952 & 0.944 & 0.947 & 0.939 & 0.950 \\ 
   & $\ossthetanpest, K_n = 9$ & 0.955 & 0.945 & 0.942 & 0.946 & 0.943 \\ 
   2 & $\ossthetanpest, K_n = 16$ & 0.950 & 0.945 & 0.943 & 0.935 & 0.941 \\ 
  & $\ossthetaestppi$ (Random Forest) & 0.948 & 0.949 & 0.950 & 0.937 & 0.954  \\ 
  & $\ossthetaestppi$ (Noisy Predictor) & 0.958 & 0.959 & 0.943 & 0.944 & 0.955 \\ 
   & PPI++ (Random Forest) & 0.960 & 0.954 & 0.949 & 0.952 & 0.943 \\ 
   & PPI++ (Noisy Predictor) &  0.960 & 0.952 & 0.947 & 0.951 & 0.943 \\ 
   \hline
\end{tabular}
\caption{95\% coverage of confidence intervals of $\theta_2^*$ for methods of Poisson GLM in two settings, averaged over 1,000 simulated datasets. For details of numerical experiments, see Section~\ref{subsec:simu_glm}.}
\label{tab:apdx_glm2_CI}
\end{table}

In the GLM setting of Section~\ref{subsec:simu_glm}, Figure~\ref{fig:apdx_glm2} displays the standard error of estimators of the second parameter. The results are similar to those described in Section~\ref{subsec:simu_glm}.


\begin{figure}
    \centering
    \includegraphics[width=0.8\linewidth]{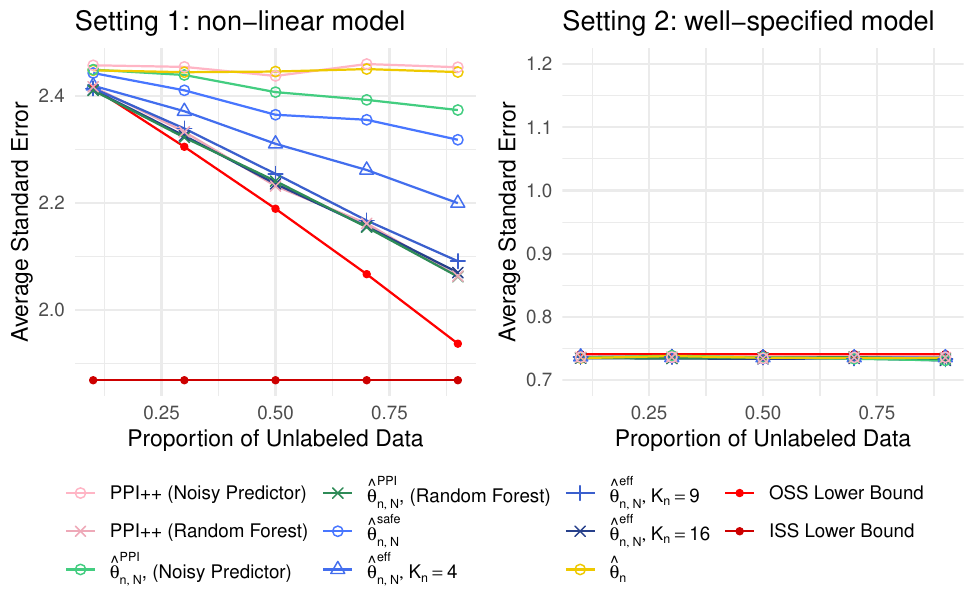}
    \caption{Estimated standard error  for of $\theta_2^*$ in the Poisson GLM setting, detailed in  Section~\ref{subsec:simu_glm}. The results are similar to Figure~\ref{fig:glm}.}
    \label{fig:apdx_glm2}
\end{figure}

\subsection{Variance estimation}
\label{subsec:apdx_simu_var}
We consider Example ~\ref{exmp:var}. The supervised estimator is the sample variance with influence function
$\IF(z) = (y-\bE[Y])^2-\theta^*$.
The conditional influence function is 
$$\cIF(x) = \bE\left[(Y-\bE[Y])^2\mid X=x\right]-\theta^*.$$
We generate the response $Y$ as $Y = \bE[Y\mid X] + \epsilon$, where $\epsilon \sim N(0,1)$. We consider two settings for $\bE[Y\mid X]$:
\begin{enumerate}
    \item \textit{Setting 1 (non-linear model):} $\bE[Y\mid X=x] = -1.70x_1x_2 -6.94x_1x_2^2-1.35x_1^2x_2+2.28x_1^2x_2^2$.
    \item \textit{Setting 2 (well-specified model, in the sense of Definition~\ref{def:well_specification}):} $\bE[Y\mid X=x] = 0$.
\end{enumerate}
For each model, we set $n = 1,000$, and vary the proportion of unlabeled observations, $\gamma =\frac{N}{n+N} \in \set{0.1,0.3,0.5,0.7,0.9}$.
The semiparametric efficiency lower bound in the ISS setting \eqref{eq:EIF_var_ISS} and in the OSS setting \eqref{eq:EIF_var_OSS} are estimated  separately on a sample of  $100,000$ observations. 

Table~\ref{tab:apdx_var_CI}  reports the coverage of 95\% confidence intervals for each method. All methods achieve the nominal coverage.

\begin{table}[ht]
\centering
\begin{tabular}{c|l|rrrrr}
  \hline
  Setting & Method & $\gamma = 0.05$ & $\gamma = 0.25$ & $\gamma = 0.5$ & $\gamma = 0.75$ & $\gamma = 0.95$\\ 
  \hline
   & $\thetaest$ & 0.943 & 0.943 & 0.947 & 0.946 & 0.955 \\ 
  & $\ossthetaest$ & 0.950 & 0.949 & 0.957 & 0.949 & 0.952 \\ 
  & $\ossthetanpest, K_n = 4$ & 0.945 & 0.937 & 0.956 & 0.956 & 0.951 \\ 
  1 & $\ossthetanpest, K_n = 9$ & 0.946 & 0.943 & 0.953 & 0.949 & 0.941 \\ 
  & $\ossthetanpest, K_n = 16$ & 0.945 & 0.941 & 0.944 & 0.959 & 0.948 \\ 
  & $\ossthetaestppi$ (Random Forest) & 0.948 & 0.948 & 0.951 & 0.951 & 0.943 \\ 
  & $\ossthetaestppi$ (Noisy Predictor) & 0.948 & 0.961 & 0.942 & 0.957 & 0.956 \\
  \hline
  & $\thetaest$ & 0.945 & 0.937 & 0.951 & 0.937 & 0.944 \\  
   & $\ossthetaest$ & 0.945 & 0.954 & 0.949 & 0.954 & 0.955 \\ 
   & $\ossthetanpest, K_n = 4$ & 0.940 & 0.947 & 0.958 & 0.944 & 0.942 \\ 
   2 & $\ossthetanpest, K_n = 9$ & 0.945 & 0.964 & 0.940 & 0.948 & 0.938 \\ 
   & $\ossthetanpest, K_n = 16$ & 0.943 & 0.939 & 0.945 & 0.946 & 0.942 \\
   & $\ossthetaestppi$ (Random Forest) & 0.949 & 0.957 & 0.948 & 0.948 & 0.946 \\ 
   & $\ossthetaestppi$ (Noisy Predictor) & 0.950 & 0.946 & 0.940 & 0.950 & 0.954 \\ 
   \hline
\end{tabular}
\caption{95\% coverage of confidence intervals for methods of variance estimation in two settings, averaged over 1,000 simulated datasets. For details of numerical experiments, see Section~\ref{subsec:apdx_simu_var}.}
\label{tab:apdx_var_CI}
\end{table}

\begin{figure}
    \centering
    \includegraphics[width=0.8\linewidth]{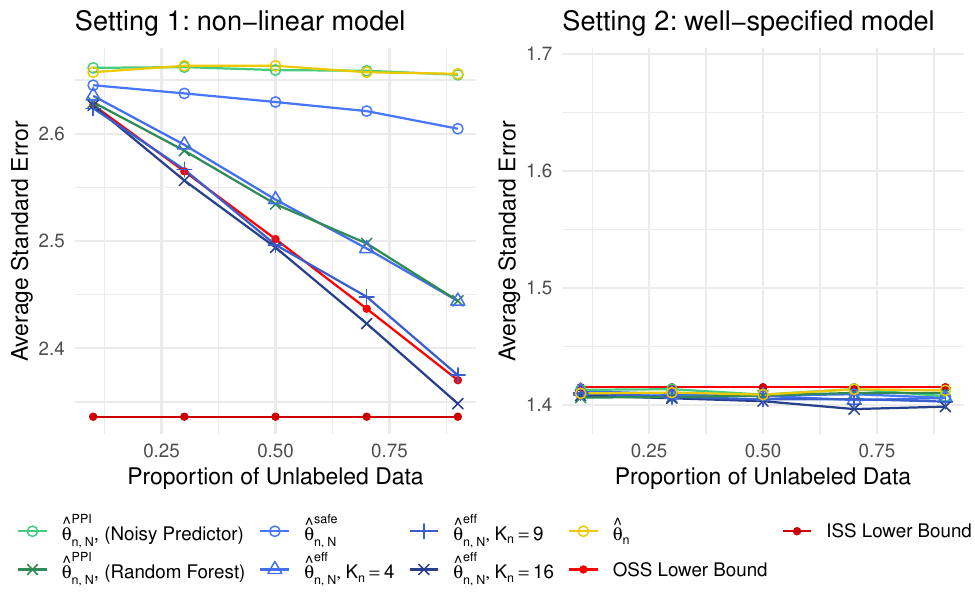}
    \caption{Comparison of standard error for methods of variance estimation, averaged over 1,000 simulated datasets.  \emph{Left:} When the conditional influence function is non-linear, $\ossthetanpest$ is efficient with a sufficient number of basis functions. \emph{Right:} For a well-specified estimation problem in the sense of Definition~\ref{def:well_specification}, no semi-supervised method improves upon the supervised estimator, which aligns with Corollary~\ref{cor:well_specified_OSS}.}
    \label{fig:var}
\end{figure}

Figure~\ref{fig:var} displays the standard error of each method, averaged over 1,000 simulated datasets. 

The results are similar to those of Section~\ref{sec:simu}. In Setting 1 (non-linear model), $\ossthetanpest$ achieves the efficiency lower bound in the OSS setting with a sufficient number of basis functions. On the other hand, $\ossthetaest$ with $g(x) = x$ improves upon the supervised estimator, but is not efficient as it cannot accurately approximate the non-linear conditional influence function. $\ossthetaestppi$ with random forest predicted responses also performs well, while the version of those methods with a pure-noise prediction model is no worse than the supervised estimator.  For Setting 2 (well-specified model), no method improves upon the OSS lower bound, as indicated by Corollary~\ref{cor:well_specified_OSS}.

\subsection{Kendall's $\tau$}
\label{subsec:apdx_simu_kendall}
We consider Example~\ref{exmp:kendall}. The supervised estimator is $$\thetaest = \frac{2}{n(n-1)}\sum_{i<j}I\set{(U_i-U_j)(V_i-V_j)>0}$$ with influence function
$$\IF(z) = 2\bP^*_Y\set{(U-u)(V-v)>0} - 2\theta^*.$$
The conditional influence function is 
$$\cIF(x) = 2\bP^*\set{(U^\prime-u)(V^\prime-V)>0\mid X=x} - 2\theta^*,$$
where $Y^\prime = (U^\prime,V^\prime)$ is an independent copy of $Y$. We generate the response $Y = (U,V)$ as $U = \bE[U\mid X] + \epsilon$ and $V = \bE[V\mid X] + \epsilon^\prime$, where $(\epsilon, \epsilon^\prime)^\top \sim N(0,\*I_2)$. We set $\bE[U\mid X = x] = \bE[V\mid X = x] := h(x)$, and consider two settings for $h(x)$:
\begin{enumerate}
    \item \textit{Setting 1 (non-linear model):} $h(x) = -1.70x_1x_2 -6.94x_1x_2^2-1.35x_1^2x_2+2.28x_1^2x_2^2$.
    \item \textit{Setting 2 (well-specified model, in the sense of Definition~\ref{def:well_specification}):} $h(x) = 0$.
\end{enumerate}
For each model, we set $n = 1,000$, and vary the proportion of unlabeled observations, $\gamma =\frac{N}{n+N} \in \set{0.1,0.3,0.5,0.7,0.9}$. Due the complexity of the data-generating model, we did not estimate the ISS and OSS lower bound.

Table~\ref{tab:apdx_kendall_CI} reports the coverage of 95\% confidence intervals for each method. All methods achieve the nominal coverage.

\begin{table}[ht]
\centering
\begin{tabular}{c|l|rrrrr}
  \hline
 Setting & Method & $\gamma = 0.05$ & $\gamma = 0.25$ & $\gamma = 0.5$ & $\gamma = 0.75$ & $\gamma = 0.95$\\ 
  \hline
  &$\thetaest$ & 0.943 & 0.950 & 0.928 & 0.942 & 0.939 \\ 
  &  $\ossthetaest$ & 0.952 & 0.938 & 0.942 & 0.944 & 0.947 \\ 
  & $\ossthetanpest, K_n = 4$ & 0.955 & 0.948 & 0.947 & 0.944 & 0.942 \\ 
  1 & $\ossthetanpest, K_n = 9$ & 0.937 & 0.951 & 0.937 & 0.941 & 0.929 \\ 
  & $\ossthetanpest, K_n = 16$ & 0.954 & 0.955 & 0.933 & 0.947 & 0.942 \\ 
   & $\ossthetaestppi$ (Random Forest)& 0.946 & 0.946 & 0.942 & 0.930 & 0.942 \\ 
 & $\ossthetaestppi$ (Noisy Predictor) & 0.937 & 0.948 & 0.945 & 0.937 & 0.936 \\ 
 \hline
   & $\thetaest$  & 0.956 & 0.966 & 0.933 & 0.951 & 0.942 \\ 
   & $\ossthetaest$ & 0.946 & 0.934 & 0.944 & 0.949 & 0.936 \\ 
   & $\ossthetanpest, K_n = 4$ & 0.948 & 0.954 & 0.944 & 0.940 & 0.941  \\ 
   2 & $\ossthetanpest, K_n = 9$ & 0.947 & 0.950 & 0.933 & 0.938 & 0.938  \\ 
   & $\ossthetanpest, K_n = 16$ & 0.935 & 0.953 & 0.937 & 0.942 & 0.947 \\ 
     & $\ossthetaestppi$ (Random Forest) & 0.949 & 0.950 & 0.944 & 0.940 & 0.944 \\ 
  & $\ossthetaestppi$ (Noisy Predictor) &  0.949 & 0.951 & 0.942 & 0.943 & 0.949\\ 
   \hline
\end{tabular}
\caption{95\% coverage of confidence intervals for methods of Kendall's $\tau$ in two settings, averaged over 1,000 simulated datasets. For details of numerical experiments, see Section~\ref{subsec:apdx_simu_kendall}.}
\label{tab:apdx_kendall_CI}
\end{table}

\begin{figure}
    \centering
    \includegraphics[width=0.8\linewidth]{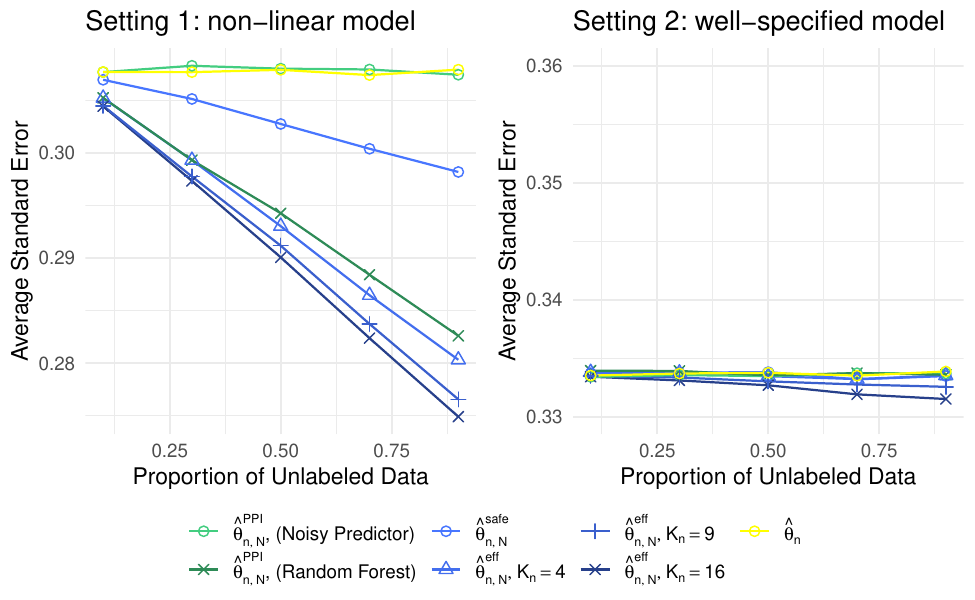}
    \caption{Comparison of standard error for methods that estimate Kendall's $\tau$, averaged over 1,000 simulations.  \emph{Left:} When the conditional influence function is non-linear, $\ossthetanpest$ is efficient with a sufficient number of basis functions. \emph{Right:} For a well-specified estimation problem in the sense of Definition~\ref{def:well_specification}, no semi-supervised methods improves upon the supervised estimator, as shown in   Corollary~\ref{cor:well_specified_OSS}.}
    \label{fig:kendall}
\end{figure}
The results are similar to those of Section~\ref{sec:simu}. In Setting 1 (non-linear model), $\ossthetanpest$ with $K_n = 16$ has the lowest standard error. $\ossthetaest$ with $g(x) = x$ improves upon the supervised estimator, but is not as efficient as $\ossthetanpest$. $\ossthetaestppi$ with random forest predicted responses also performs well, while its counterpart with pure-noise prediction models has comparable standard error to the supervised estimator.  For Setting 2 (well-specified model), no methods improves upon the OSS lower bound as indicated by Corollary~\ref{cor:well_specified_OSS}. We note that the  apparent slight improvement of standard error in the second panel by  $\ossthetanpest$ is the result of a relatively small sample size. 
\end{appendix}
%
%

%


\end{document}